\documentclass[11pt,reqno]{amsart}
\usepackage{t1enc}
\usepackage[latin1]{inputenc}
\usepackage{amsmath,latexsym,amssymb,graphicx,dsfont,amsthm,amsfonts}
\usepackage{color}
\usepackage{umoline}

\newcommand{\R}{\mathbb{R}}

\newtheorem{Th}{Theorem}[section]
\newtheorem{Lemma}[Th]{Lemma}
\newtheorem{Cor}[Th]{Corollary}
\newtheorem{Prop}[Th]{Proposition}

\newtheorem{Ex}[Th]{Example}
\newtheorem{Ass}[Th]{Assumption}

\newcommand{\E}{\mathbb{E}}
\newcommand{\Prob}{\mathbb{P}}

\newcommand{\calL}{\mathcal{L}}
\newcommand{\calA}{\mathcal{A}}
\newcommand{\calB}{\mathcal{B}}

\newcommand{\calI}{\mathcal{I}}

\newcommand{\N}{\mathbb{N}}

\newcommand{\W}{\mathbf{W}}
\newcommand{\Y}{\mathbf{Y}}
\newcommand{\ZZ}{\mathbf{Z}}
\newcommand{\e}[1]{\mathbf{e}_{#1}}

\newcommand{\he}[1]{\hat{\mathbf{e}}_{#1}}
\newtheorem{Rem}[Th]{Remark}

\numberwithin{equation}{section}

\begin{document}

\title[Asymptotic Entropy of Random Walks on Regular Languages]{Asymptotic Entropy of Random Walks on Regular Languages over a Finite Alphabet}

\author{Lorenz A. Gilch}


\address{Lorenz Gilch: Graz University of Technology, Steyrergasse 30/III/C304,
  8010 Graz, Austria}

\email{Lorenz.Gilch@freenet.de}
\urladdr{http://www.math.tugraz.at/$\sim$gilch/}
\date{\today}
\subjclass[2000]{Primary: 60J10; Secondary: 28D20} 
\keywords{random walks, regular languages, entropy, analytic}

\maketitle



\begin{abstract}
We prove existence of asymptotic entropy of random walks on regular languages
over a finite alphabet and we give formulas for it. Furthermore, we show that the entropy varies
real-analytically in terms of probability measures of constant support,
which describe the random walk. This setting applies, in particular, to random
walks on virtually free groups.
\end{abstract}

\section{Introduction}

Let $\calA$ be a finite alphabet and let $\calA^\ast$ be the set of all finite
words over the alphabet $\calA$, where $o$ denotes the empty word. Consider a transient
Markov chain $(X_n)_{n\in\N_0}$ on $\calA^\ast$ with $X_0=o$ such that at each
instant of time the last $K\in\N$ letters of the
current word may be replaced by a word of length of at most $2K$ and the
transition probabilities depend only on the last $K$ letters of the
current word and on the replacing word. For better visualization and ease of presentation, we also consider the random walk on
$\calA^\ast$ as a random walk on an undirected graph $\mathcal{G}$.
Denote by $\pi_n$ the
distribution of $X_n$. We are interested whether the sequence
$\frac{1}{n}\mathbb{E}[-\log\pi_n(X_n)]$ converges, and if so to describe the
limit. If it exists, it is called the \textit{asymptotic entropy}, which was
introduced by Avez \cite{avez72}. The aim of this paper is to prove existence
of the asymptotic entropy, to describe it as the rate of escape w.r.t. the Greenian distance and
to prove its real-analytic behaviour when varying the transition probabilities of
constant support. 
\par
We outline some background on this topic. Random Walks on regular languages
have been studied by e.g. Lalley \cite{lalley} and Malyshev \cite{malyshev}
amongst others. Concerning asymptotic entropy it is well-known by Kingman's
subadditive ergodic theorem (see Kingman \cite{kingman}) that the entropy
exists for random walks on groups if $\mathbb{E}[-\log \pi_1(X_1)]<\infty$. In contrast to this fact existence of the
entropy on more general structures is not known a priori. In our setting we are not
able to apply the subadditive ergodic theorem since we neither have 
subadditivity nor a global composition law of words if the random
walk is performed on a proper subset of $\calA^\ast$ (that is, not every word
$w\in\calA^\ast$ can be reached from $o$ with positive probability). This forces us to use other
techniques like generating functions techniques. These generating functions are
power series with probabilities as coefficients, which describe the
characteristic behaviour of the underlying random walks. The technique of
our proof of existence of the entropy was motivated by Benjamini and Peres
\cite{benjamini-peres94}, where it is shown that for random walks on groups the
entropy equals the rate of escape w.r.t. the Greenian distance; compare also
with Blach\`ere, Ha\"issinsky and Mathieu \cite{blachere-haissinsky-mathieu}.
In particular, we will also show 
that the  asymptotic entropy $h$ is the rate of escape w.r.t. a distance function in terms
of Green functions, which in turn yields that $h$ is also the rate of escape
w.r.t. the Greenian distance. Moreover, we prove convergence in probability and convergence in $L_1$ of the sequence
$-\frac{1}{n} \log \pi_n(X_n)$ to $h$, and we show also that $h$ can be computed along
almost every sample path as the limes inferior of the aforementioned
sequence. The question of almost sure convergence of $-\frac{1}{n}\log \pi_n(X_n)$ to some constant $h$, however, remains open. Similar results concerning existence
and formulas for the entropy are proved in Gilch and M\"uller \cite{gilch-mueller} for random walks
on directed covers of graphs and in Gilch \cite{gilch:11} for random walks on
free products of graphs. Furthermore, we give formulas for the entropy which
allow numerical computations and also exact calculations in some special
cases. The main idea in our proofs is to fix a priori a sequence of nested
cones in the associated graph $\mathcal{G}$ and to track the random walk's way to infinity through these
cones. Similar ideas have been used independently by Woess \cite{woess12} for
context-free pairs of groups.
\par
Kaimanovich and Erschler asked whether drift and entropy of random walks
vary continuously (or
even analytically) when varying the probabilities of the random walk
with keeping the support of single step transitions constantly. In
this article we also show that $h$ is real-analytic in
terms of the parameters describing the random walk on $\calA^\ast$. This fact
applies, in particular, to the case of bounded range random walks on virtually
free groups, which goes beyond the scope of previous results related to the
question of analyticity. At this
point let us summarize several papers
concerning continuity and analyticity of the drift and entropy that have been
published recently: e.g., see Ledrappier \cite{ledrappier12}, \cite{ledrappier12-2},
Ha\"issinsky, Mathieu and M\"uller \cite{haissinsky-mathie-mueller12}, Gilch \cite{gilch:11}. The
recent survey article of Gilch and Ledrappier \cite{gilch-ledrappier13} collects several
results about analyticity of drift and entropy of random walks on groups.
\par
The basic reasoning of our proofs follows a similar argumentation as in
\cite{gilch-mueller} and \cite{gilch:11}, but since a straight-forward adaption is
not possible we have to do more effort in the present setting:
we will show that the entropy equals the rate of escape w.r.t. some special length function,
and we deduce the proposed properties analogously. For the proof of analyticity of
the entropy we will extract a hidden Markov chain from our random walk and we
will apply a result of Han and Marcus \cite{han-marcus06}.
The plan of the paper is as follows: in Sections \ref{sec:notation} and
\ref{sec:genfun} we define the random walk on
$\calA^\ast$ and the associated generating functions. Section
\ref{sec:cones} explains the construction of cones in the present context. In
Sections \ref{sec:exit-times} and \ref{sec:entropy} we prove existence of the
asymptotic entropy and give a formula for it, while in Section
\ref{sec:entropy-calculation} we give estimates and a more explicit formula in some special case. Section \ref{sec:entropy-analyticity} shows real-analyticity of the entropy.

\section{Random Walks on Regular Languages}
\label{sec:notation}

\subsection{Definitions and Main Results}

Let $\calA$ be a finite alphabet and denote by $\mathcal{A}^\ast$ the set of
all finite words over $\mathcal{A}$. We write $o$ for the empty word and
$\mathcal{A}^n$, $n\in\N$, for the set of all words over $\calA$ consisting of
exactly $n$ letters. For two
words $w_1,w_2\in\calA^\ast$, $w_1w_2$ denotes the concatenated word. A
\textit{random walk on a regular language} is a Markov chain $(X_n)_{n\in\N_0}$ on the set
$\calA^\ast=\bigcup_{n\geq 1} \calA^n\cup \{o\}$, whose transition probabilities obey the following rules:
\begin{enumerate}
\item[(i)] Only the last two letters of the current word may be modified.
\item[(ii)] Only one letter may be adjoined or deleted at one instant of time.
\item[(iii)] Adjunction and deletion may only be done at the end of the current word.
\item[(iv)] Probabilities of modification, adjunction or deletion depend only on the
  last two letters of the current word and on the substitute letters.
\end{enumerate}
Compare with Lalley \cite{lalley} and Gilch \cite{gilch:08}. In other words, at
each step the last two letters of the current word may be replaced by a
non-empty word of
length of at most $3$ and the transition probabilities depend only on the last
two letters of the current word and the replacing word of length of at most $3$.
More formally, the
transition probabilities of the Markov chain $(X_n)_{n\in\N_0}$ can be written
as follows, where $w\in\calA^\ast$, $a_1,a_2,b_1,b_2,b_3\in \calA$:
\begin{eqnarray}
\Prob[X_{n+1}=wb_1b_2 \mid X_n=wa_1a_2]  &=&  p(a_1a_2,b_1b_2),\nonumber\\
\Prob[X_{n+1}=wb_1b_2b_3 \mid X_n=wa_1a_2]  &=&  p(a_1 a_2,b_1b_2b_3),\nonumber\\
\Prob[X_{n+1}=wb_1 \mid X_n=wa_1 a_2]  & =&  p(a_1 a_2,b_1),\nonumber\\
\Prob[X_{n+1}=b_1 \mid X_n=a_1]  & = &  p(a_1,b_1),\label{equ:random-walk}\\
\Prob[X_{n+1}=b_1b_2 \mid X_n=a_1]  & = &  p(a_1,b_1b_2),\nonumber\\
\Prob[X_{n+1}=o \mid X_n=a_1]  &=&  p(a_1,o),\nonumber\\
\Prob[X_{n+1}=b_1 \mid X_n=o]  & = &  p(o,b_1),\nonumber\\
\Prob[X_{n+1}=o \mid X_n=o] & =&  p(o,o).\nonumber
\end{eqnarray}
Not all of these probabilities need to be strictly positive. Initially, we set
$X_0:=o$.  If we start the random walk at $w\in \calA^\ast$ instead of $o$, we write $\Prob_w[\,\cdot\,]:=\Prob[\,\cdot \mid X_0=w]$.
 For $w_1,w_2\in\calA^\ast$, the $n$-step transition probabilities are denoted by
$p^{(n)}(w_1,w_2):=\Prob_{w_1}[X_n=w_2]$. The set of \textit{accessible words} from $o$ is given by
$$
\calL = \bigl\lbrace w\in\calA^\ast \mid \exists n\in\N: \Prob[X_n=w\mid X_0=o]>0\bigr\rbrace.
$$
We will also think of the random walk $(X_n)_{n\in\N_0}$ as a nearest
neighbour random walk on an \textit{undirected graph} $\mathcal{G}$, where the
vertices are the elements of $\calL$ and undirected edges are between two
vertices if and only if one can walk from one word to the other one in a single
step. For this purpose, we need the following assumption:
\begin{Ass}[Weak symmetry]\label{ass:weak}
For all $u,v\in\calA^\ast$ we assume that $\mathbb{P}_u[X_1=v]>0$ implies $\mathbb{P}_v[X_1=u]>0$. We call this property \textit{weak symmetry}.
\end{Ass}
In particular, Assumption \ref{ass:weak} yields 
irreducibility of the random walk on $\calL$. Moreover, this assumption will be
necessary for the construction of a sequence of cones in the graph
$\mathcal{G}$ which track the random
walk's way to infinity. As the interested reader will see, weak symmetry can obviously be
weakened in some way but for reason of better readability we keep this natural
assumption; for a discussion on this assumption, we refer to Appendix \ref{sub:weak-symmetry}.
\par
Since the purpose of the paper is the investigation of the asymptotic behaviour of
transient random walks, we obviously need that $\calL$ is
infinite in our setting. It is an easy exercise to check that the set $\calL$ is a \textit{regular language} over the alphabet $\calA$, that
is, the words are accepted by a finite-state automaton. For more details on
regular languages, we refer e.g. to Hopcraft and Ullman \cite{hopcraft-ullman}. Since we make no further use of the
theory of languages, we will not discuss this in more detail but we remark the
recursive structure of regular languages. Let us note that bounded range random walks on \textit{virtually
free groups} constitute a special case of our setting, and our results directly
apply; see Subsection \ref{subsub:vfg}. Thus, our results apply directly to a
large class of random walks on groups and go beyond recent
results for random walks on groups.
\par
\begin{Rem}
Observe that the assumption that transition probabilities depend only on
the \emph{last two} letters of the current word and that changes of the current
word involve only the last two letters may be weakened to dependence and
changes of the
last $K\in\N$ letters by blocking words of length of at most $K$ to new single letters
(see \cite[Section 3.3]{lalley} for further details and comments; note that it is not sufficient to consider the case where the transition probabilities/changes of words involve only the last letter in order to be able to apply this recoding trick!). In the \emph{$K$-dependent case} the
general transition probabilities have the form
\begin{equation}\label{equ:K-random-walk}
\Prob\bigl[X_{n+1}=wy \mid X_n=wx] =p(x,y),
\end{equation}
where $w,x,y\in\calA^\ast$ with $x$ being a word consisting of $K$ letters and
$y$ being a  word consisting of at most $2K$ letters. In this paper we will
restrict our attention to the case of dependence on the last two letters of the
current word as defined via (\ref{equ:random-walk}). 
If we  make further assumptions on our random walk in the following, we will show that it does not depend on the
fact if we use the ``blocked letter language'' (that is, dependence on the last
two letters as given by (\ref{equ:random-walk}) after an application of the
``recoding trick'') or the general case (dependence
on the last $K$ letters as given by (\ref{equ:K-random-walk})), that is, no required properties are lost when
switching from the $K$-dependent case to the ``blocked letter language''.
It will turn out that the $K$-dependent case works completely analogously as the ``blocked
letter language'' case; however, the derived equations and formulas are much more
complex, so we restrict ourselves onto the case where the random walk is defined
as at the beginning of this section. In particular, there
is no additional gain in the techniques and proofs when investigating the
$K$-dependent case. For further comments, see Appendix \ref{sub:K-dependent}.
\end{Rem}
\par
We introduce some notation. The \textit{natural word length} of any $w\in\calA^\ast$ is denoted by
$|w|$. If $w\in\calA^\ast$ and
$k\in\mathbb{N}$ with $|w|\geq k$ then $w[k]$
denotes the $k$-th letter of $w$, and $[w]$ denotes the last two letters of $w$
when $w\neq o$ is not a single letter. 
\par
Malyshev \cite{malyshev} proved that the rate of escape w.r.t. the
natural word length exists for irreducible random walks on regular languages, that is, there is a non-negative constant $\ell$ such that
$$
\lim_{n\to\infty}\frac{|X_n|}{n} = \ell \quad \textrm{ almost surely.}
$$
Here, $\ell$ is called the \textit{rate of escape}. Furthermore, by \cite{malyshev}
follows that $\ell$ is strictly positive if and only if
$(X_n)_{n\in\N_0}$ is transient. In \cite{gilch:08} there are explicit formulas
for the rate of escape w.r.t. more general length functions. 
\par
Another characteristic number of random walks is the asymptotic entropy. Denote by $\pi_n$ the distribution of $X_n$. If there is a non-negative
constant $h$ such that the limit
$$
h=\lim_{n\to\infty} -\frac{1}{n}\mathbb{E}\bigl[ \log \pi_n(X_n)\bigr]
$$
exists, then $h$ is called the \textit{asymptotic entropy}. Since
we only have a partial composition law for concatenation of two words (if
$\mathcal{L}\subset \calA^\ast$) and since
we have no subadditivity and transitivity of the random walk, we are not able
to apply -- as in the case of
random walks on groups -- Kingman's subadditive ergodic theorem in order to show existence of $h$. 
It is, however, easy to see that the entropy equals zero if the random walk is
recurrent (see Corollary \ref{cor:recurrent-h-zero}). Therefore, from now on we
will only consider \textit{transient} random walks $(X_n)_{n\in\N_0}$.
\begin{Rem} 
Observe that $\lim_{n\to\infty} -\frac{1}{n}\log
  \pi_n(X_n)$ is not necessarily deterministic: take two homogeneous trees
of different degrees $d_1,d_2\geq 3$; identify
their root with one single root which becomes $o$ and consider the simple
random walk on this new inhomogeneous tree with starting point $o$. Obviously, this random walk can be
modelled as a random walk on a regular language. Then the limit $\lim_{n\to\infty} -\frac{1}{n}\log
  \pi_n(X_n)$ depends on the fact in which of the two subtrees the random walks
  goes to infinity. Hence, the sequence $-\frac{1}{n}\log
  \pi_n(X_n)$ converges with probability
  $d_1/(d_1+d_2)$ to $\log (d_1-1)$ and with probability
  $d_2/(d_1+d_2)$ to $\log (d_2-1)$; this can, e.g., be calculated by the
  formulas given in \cite{gilch:11}.
\end{Rem}
We have to make another assumption on the transition probabilities:
\begin{Ass}[Suffix-irreducibility]\label{ass:suffix}
We assume that the random walk on $\calL$ is \textit{suffix-irreducible},
that is, for all $w=w_0a_0b_0\in\calL$ with $w_0\in\calA^\ast, a_0b_0\in\calA^2$ and for all $ab\in\calA^2$ there is $n\in\N$ and $w_1\in\calA^\ast$ such that
$$
\Prob\Bigl[ X_n=w_0w_1ab, \forall k \leq n: |X_k|\geq |w|\, \Bigl|\, X_0=w\Bigr]>0.
$$
\end{Ass}
This assumption excludes degenerate cases and will guarantee existence of $\ell$;
compare with \cite[End of Section 2.1]{gilch:08}. We remark that famous previous papers about random
walks on regular languages (in particular, the basic ones of \cite{malyshev} and \cite{lalley})
require stronger assumptions than this non-degeneracy assumption. Later on it
will be clear that one can relax this condition in some way  without needing
additional techniques or ideas for the proofs. Hence, for purpose of ease and
better readability, we keep this assumption until further notice. We will give
further comments on this assumption in Appendix \ref{subsec:remarks-1}.
\par
The main idea behind our proofs will be the construction of an a priori fixed
sequence of cones (that is, special subsets of $\calL$), from which we extract a subsequence of nested cones which gives the
information how the random walk tends to infinity. 
This extraction will be done
via a hidden Markov chain $(\Y_k)_{k\in\N}$ with an underlying positive
recurrent Markov chain: the asymptotic entropy $H(\mathbf{Y})$ of the process $(\Y_k)_{k\in\N}$ is
 given by (\ref{equ:H(ZZ)}). The average distance between two
nested cones will be denoted by $\lambda$ which is given by (\ref{equ:lambda}):
if $X_{\e{k}}$ denotes the word (i.e., the vertex in $\mathcal{G}$) where the
$k$-th nested subcone is finally entered with no further exits of this
cone, then $\lambda=\mathbb{E}[|X_{\e{2}}|-|X_{\e{1}}|]$.
Our first main result concerns existence of the asymptotic entropy, which is
finally proven in Section \ref{sec:entropy}:
\begin{Th}\label{th:existence-entropy}
Consider a transient random walk $(X_n)_{n\in\N_0}$ on a regular language, which satisfies
Assumptions \ref{ass:weak} and \ref{ass:suffix}. Then the asymptotic entropy $h$ of $(X_n)_{n\in\N_0}$ exists and equals 
$$
h=\frac{\ell\cdot H(\mathbf{Y})}{\lambda},
$$
where $H(\mathbf{Y})$ is given by (\ref{equ:H(ZZ)}) and $\lambda$ by (\ref{equ:lambda}).
\end{Th}
Recall that the random walk is described by the values in (\ref{equ:random-walk}).
A natural question is whether the entropy varies regularly if the parameters in
(\ref{equ:random-walk}) are varied slightly and if positive transition
probabilities remain positive by this variation. The following result gives
an answer to this question, where the proof is given in Section \ref{sec:entropy-analyticity}:
\begin{Th}\label{thm:entropy-continuity}
For transient random walks on regular languages satisfying Assumptions \ref{ass:weak} and \ref{ass:suffix}, the entropy $h$ varies real-analytically under all probability measures of constant support.
\end{Th}
Moreover, we can also describe the asymptotic entropy in the following way:
\begin{Cor}\label{cor:entropy-convergence}
We have the following types of convergence:
\begin{enumerate}
\item For almost every trajectory of the random walk $(X_n)_{n\in\N_0}$,
$$
h=\liminf_{n\to\infty} -\frac{1}{n} \log \pi_n(X_n).
$$
\item Convergence in probability:
$$
-\frac{1}{n}\log \pi_n(X_n) \xrightarrow{\Prob} h.
$$ 
\item Convergence in $L_1$:
$$
-\frac{1}{n}\log \pi_n(X_n) \xrightarrow{L_1} h.
$$
\end{enumerate}
\end{Cor}
The \textit{Greenian distance} between two words $w_1,w_2\in\calL$ is defined as 
$$
d_{\mathrm{Green}}(w_1,w_2):=-\log \Prob[\exists n\in\N_0: X_n=w_2 \mid X_0=w_1].
$$
Analogously to the situation for random walks on groups, we get the following
result, which is finally proven at the end of Section \ref{sec:entropy}:
\begin{Cor}\label{cor:green-distance}
The entropy is the rate of escape with respect to the Greenian distance, that
is,
$$
h=\lim_{n\to\infty} -\frac{1}{n} d_{\mathrm{Green}}(o,X_n) \quad \textrm{almost surely.}
$$
\end{Cor}
Further results are given in Section \ref{sec:entropy-calculation}, where we
show that $h>0$ (Corollary \ref{cor:h-zero}) for non-degenerate transient random walks, give an inequality between entropy, drift and growth (Theorem \ref{prop:h-inequality}) and give an exact formula in some special case (Theorem \ref{thm:unambiguous}).

\subsection{Examples}

We give two classical examples for regular languages.

\subsubsection{Stacks}

In computer science theory \textit{stacks} play an important role for modelling
algorithms. Stacks are
data structures and are also called ``Last In--First Out'' queues.
They can be considered as
lists, where elements may be adjoined or removed only at the end of the
list. Adjoining and removing elements is done randomly (assuming \ref{ass:weak} and \ref{ass:suffix}). For instance, the
single letters may represent different kinds of main and subprocedures, which may be called
recursively. When a subprocedure has finished the associated letter may be
deleted. When some main procedure has finally terminated (including all
subprocedures and their recursive calls) then a single letter may be
added which won't be removed any more. In terms of transient random walks this
means that prefixes of arbitrary lengths remain constant for the rest of the random process if
the associated procedures are terminated; the tail of the current word
represents some open subprocedures still to be done. The asymptotic entropy gives
then the average additional amount of information which one gets with each
terminated main procedure:
each very likely called procedure reduces entropy and each unlikely called
procedure increases the entropy. The overall average gives the asymptotic entropy.

\subsubsection{Virtually Free Groups}
\label{subsub:vfg}

Virtually free groups are groups which contain a free group as a subgroup of
finite index. This kind of groups can be implemented as regular languages: let
$\Gamma$ be a virtually free group which contains the free group $\mathbb{F}_d$
with $d$ generators as a subgroup of index $[\Gamma: \mathbb{F}_d]=k$. Let
$\mathbb{F}_d$ be generated by the elements $a_1,a_1^{-1},\dots,a_d,a_d^{-1}$,
and let $h_1,\dots,h_k$ be representants of the $k$ different left co-sets of
$\Gamma$. That is, each element $x\in\Gamma$ can be written as
\begin{equation}\label{equ:vfg}
x_1x_2\dots x_m h_j,
\end{equation}
where $m\in\N_0$, $j\in\{1,\dots,k\}$ and $x_1,\dots,x_m\in\{a_1,a_1^{-1},\dots,a_d,a_d^{-1}\}$
such that $x_i^{-1}\neq x_{i+1}$ for all $i\in\{1,\dots,m-1\}$.
\par
Let $\mu$ be a finitely supported probability measure on $\Gamma$ and let
$(\xi_n)_{n\in\mathbb{N}}$ be an i.i.d. sequence of $\Gamma$-valued random
variables with distribution $\mu$. 
A random walk $(X_n)_{n\in\mathbb{N}_0}$ on $\Gamma$ is then given by 
$$
X_{n+1}=X_n\xi_{n+1}
$$
with $X_0$ being the neutral element of $\Gamma$. In other words,
the single-step transition probabilities are given by
$p(x,y):=\mu(x^{-1}y)$ for $x,y\in\Gamma$. After each step one calculates a
reduced expression as in (\ref{equ:vfg}) for the current word; hence, only the last $K$ letters of
the current word may be changed, where $K$ has to be chosen sufficiently large
in dependence on the largest word of $\mathrm{supp}(\mu)$. Compare also with the detailed example of free products
by amalgamation in \cite[Section 3.1]{gilch:08}.

\section{Generating Functions}
\label{sec:genfun}
For $w_1,w_2\in\calA^\ast$, $z\in\mathbb{C}$, the \textit{Green function} is defined as 
$$
G(w_1,w_2|z) := \sum_{n\geq 0} p^{(n)}(w_1,w_2)\cdot z^n
$$
and the \textit{last visit generating function} as
$$
L(w_1,w_2|z) := \sum_{n\geq 0}\Prob\bigl[X_n=w_2,\forall m\in\{1,\dots,n\}:
X_m\neq w_1 \bigl|X_0=w_1\bigr]\cdot z^n.
$$
By conditioning on the last visit to $w_1$, an important relation
between these functions is given by
\begin{equation}\label{equ:G-L}
G(w_1,w_2|z) = G(w_1,w_1|z) \cdot L(w_1,w_2|z).
\end{equation}
In the following we introduce further generating functions, which also have
 been used analogously in \cite{gilch:08}.
Define for $a,b,c,d,e\in \calA$ and real $z>0$ 
$$
H(ab,c|z)  :=    \sum_{n\geq 1} \Prob\bigl[X_n=c,\forall m<n:
|X_m|>1 \bigl| X_0=ab\bigr]\cdot z^n
$$
and
\begin{eqnarray*}
\bar L(ab,cde|z) &  :=&   \sum_{n\geq 1} \Prob\bigl[X_n=cde,|X_{n-1}|=2,\forall m\in\{1,\dots,
n\}: |X_m|\geq 2, \bigl| X_0=ab\bigr]\cdot z^n,\\
\overline{G}(ab,cd|z) &:= & \sum_{n\geq 0} \Prob\bigl[X_n=cd,\forall m\in\{1,\dots,
 n\}: |X_m|\geq 2 \bigl| X_0=ab\bigr]\cdot z^n.
\end{eqnarray*}
We write $\bar L(ab,cde):=\bar L(ab,cde|1)$.
These generating functions can be computed in two steps: first, one solves the following
system of equations which arises by case distinction on the first step:
\begin{eqnarray}\label{h-equations}
H(ab,c|z) &=& p(ab,c)\cdot z +  \sum_{de\in
  \calA^2} p(ab,de)\cdot z\cdot H(de,c|z)\nonumber\\
 && \ + \sum_{def\in \calA^3} p(ab,def)\cdot z \cdot \sum_{g\in \calA} H(ef,g|z)\cdot H(dg,c|z);
\end{eqnarray}
compare with \cite{lalley} and \cite{gilch:08}. 
 The system (\ref{h-equations}) consists of equations of quadratic order, and
 therefore the functions $H(\cdot,\cdot|z)$ are algebraic, if the transition
 probabilities are algebraic. 
We now get the functions $\overline{G}(ab,cd|z)$ by solving the following
linear system of equations which also arises by case distinction on the first step:
 \begin{eqnarray*}
 \overline{G}(ab,cd|z) &=& \delta_{ab}(cd) + \sum_{c_1d_1\in \calA^2}
 p(ab,c_1d_1) \cdot z\cdot \overline{G}(c_1d_1,cd|z) +\\
 && \ + \sum_{c_1d_1e_1\in \calA^3} p(ab,c_1d_1e_1)\cdot
 z\cdot \sum_{f\in \calA} H(d_1e_1,f|z)\cdot \overline{G}(c_1f,cd|z).
 \end{eqnarray*}
Finally, we get 
\begin{equation}\label{equ:L-formel}
\bar L(ab,cde|z) =  \sum_{a_1b_1\in\calA^2} \overline{G}(ab,a_1b_1|z)\cdot z
\cdot p(a_1b_1,cde).
\end{equation}
Obviously, it is sufficient to consider only those functions $H(ab,\cdot|z)$,
$\overline{G}(ab,\cdot|z)$ and $L(ab,\cdot|z)$ such that there exists some $w_0\in\calA^\ast$ with
$w_0ab\in\calL$; the remaining functions do not play a role for our random walk. 
Moreover, one can compute the Green functions of the form $G(o,w|z)$,
$w\in\calL$ with $|w|\leq 3$, by
 solving
 \begin{eqnarray*}
 G(w_1,w_2|z) &=& \delta_{w_1}(w_2) + \sum_{w_3\in \calA^\ast:|w_3|\leq 3}
 p(w_1,w_3) \cdot z\cdot G(w_3,w_2|z) +\\
 && \ + \mathds{1}_{3}(w_1)\cdot \sum_{cde\in \calA^3} p(w_1[2]w_1[3],cde)\cdot
 z\cdot \sum_{f\in \calA} H(de,f|z)\cdot G(w_1[1]cf,w_2|z),\nonumber
 \end{eqnarray*}
 where $w_1,w_2\in\calA^\ast$ with $|w_1|,|w_2|\leq 3$ and $\mathds{1}_{3}(w_1):=1$,
 if $|w_1|=3$, and $\mathds{1}_{3}(w_1):=0$ otherwise. 
\par
We also define for $ab\in\calA^2$:
$$
\xi(ab) := \Prob\bigl[\forall n\geq 0: |X_n|\geq 2 \,\bigl|\, X_0=ab\bigr]=1-\sum_{f\in\calA} H(ab,f|1).
$$
When starting at a word $wab\in\calL$, where
$w\in\calA^\ast$, $\xi(ab)$ is the probability that the process
$(X_n)_{n\in\N_0}$ will not visit any words of length
$|wab|-1$ or smaller. In this case the prefix $w$ will remain constant for the rest of
the process. Observe that, for transient random walks, $\xi(ab)>0$ for all $ab\in\calA^2$ due to
Assumption \ref{ass:suffix}.
We define a ``length function'' on $\calL$ by
\begin{equation}\label{def:l-definition}
l(w) := -\log L(o,w|1)\quad \textrm{ for } w\in\calL.
\end{equation}
For $n\geq 2$ and $a_1,\dots,a_n\in\calA$, the functions $L(o,a_1\dots a_n|z)$ can be rewritten as 
\begin{equation}
\sum_{b,b_0,c_0\in\calA} L(o,b|z)\cdot z\cdot p(b,b_0c_0)
\sum_{\substack{b_1,\dots,b_{n-2}\in\calA,\\
c_1,\dots,c_{n-2}\in\calA}} 
\prod_{i=1}^{n-2} \bar L(b_{i-1}c_{i-1},a_ib_ic_i|z)\cdot
\overline{G}(b_{n-2}c_{n-2},a_{n-1}a_n|z);\label{equ:L-expansion}
\end{equation}
each path from $o$ to $a_1\dots a_n$ is decomposed to the last times when the
sets $\calA,\calA^2,\dots,\calA^{n-1}$ are visited, that is, the factor $\bar L(b_{i-1}c_{i-1},a_ib_ic_i|z)$
corresponds to the parts of the paths from $o$ to $a_1\dots a_n$ between the
final exits of the sets $\calA^{i}$ and $\calA^{i+1}$.

\section{Cones}
\label{sec:cones}

\subsection{Definitions of Cones and Properties}\label{subsec:cone-def}

In this section we introduce the structure of cones in our setting.
A \textit{path} in $\calA^\ast$  is a sequence of words $\langle
w_0,w_1,\dots,w_m\rangle$, $m\in\mathbb{N}$, in $\calA^\ast$ such
that $\mathbb{P}_{w_{i-1}}[X_1=w_{i}]>0$ for all $1\leq i\leq m$. By weak
symmetry, we have that, for each such path, the reversed sequence of words
$\langle
w_m,w_{m-1},\dots,w_0\rangle$ is also a path.
For $n\in\N$, define $\calA^\ast_{\geq n}:=\bigl\lbrace w\in\calA^\ast \bigl| |w|\geq n\bigr\rbrace$.
For any $w_0\in\calA^\ast_{\geq 2}$, we define the
\textit{cone} rooted at $w_0$ as
$$
C(w_0):=\left\lbrace w\in\calA^\ast_{\geq |w_0|} \Biggl|
\begin{array}{c}
\exists m\in\N_0\ \exists \textrm{ path }
\langle w_0,w_1,\dots,w_{m-1}, w\rangle \\
\textrm{ with }  w_1,\dots,w_{m-1}\in \calA^\ast_{\geq |w_0|}
\end{array}\right\rbrace. 
$$
In other words, when we consider the associated graph $\mathcal{G}$ then the
cone $C(w_0)$ can be viewed as the subgraph of $\mathcal{G}$ which is the connected component
containing $w_0$ after removing all vertices $w'\in\mathcal{A}\setminus
\calA^\ast_{\geq |w_0|}$ and the adjacent edges to these $w'$. In particular, we have $w_0\in C(w_0)$.
If $w_1\in C(w_0)$
then we have $C(w_1)\subseteq C(w_0)$: indeed, let be $w_2\in C(w_1)$; therefore, $|w_2|\geq |w_1|\geq |w_0|$
and there are paths $\langle w_0,w_1',\dots,w_k',w_1\rangle$ through words
$w_1',\dots,w_k'\in\calA^\ast_{\geq |w_0|}$ and $\langle
w_1,w_1'',\dots,w_l'',w_2\rangle$ through words
$w_1'',\dots,w_l''\in\calA^\ast_{\geq |w_1|}\subseteq \calA^\ast_{\geq
  |w_0|}$. Hence, there is a path $\langle
w_0,w_1',\dots,w_k',w_1,w_1'',\dots,w_l'',w_2\rangle$ through words in
$\calA^\ast_{\geq |w_0|}$, that is, $w_2\in C(w_0)$ yielding $C(w_1)\subseteq
C(w_0)$. The cone $C(w_1)$ is then called a \textit{subcone} of $C(w_0)$.
\par
Observe that each element $w\in
C(w_0)$ has the form $w=a_1\dots a_{m-2}\bar w$, where $w_0=a_1\dots a_m$ with
$m\geq 2$, $a_1,\dots,a_m\in\calA$ and where $\bar w\in\calA^\ast_{\geq 2}$:
indeed, by definition each $w\in C(w_0)$ can be reached from $w_0$ by a path
through words of length bigger or equal than $|w_0|$. Thus, the first $m-2$
letters  are \textit{not} changed along such a path.
\par
By the suffix-irreducibility Assumption \ref{ass:suffix}, we have the following important property for cones: let be $w\in\calA^\ast$ and $ab,cd\in\calA^2$; then the cone $C(wab)$ has a proper subcone
$C(wxcd)\subset C(wab)$ with a suitable choice of
$x\in\calA^\ast\setminus\{o\}$. 
\par
Recall that $[w]$ denotes the last two letters
of a word $w\in\calA^\ast_{\geq 2}$.
We say that two cones
$C(w_1)$ and $C(w_2)$, $w_1,w_2\in\calA^\ast$, are \textit{isomorphic} if
$C([w_1])=C([w_2])$. The following lemma explains why we call these cones ``isomorphic'': 
\begin{Lemma}\label{lemma:cone-isomorphism}
Let be $w_1=a_1\dots a_m$, $w_2=b_1\dots b_n\in\calA^\ast_{\geq 2}$ with 
$a_1,\dots,a_m,b_1,\dots,b_n\in\calA$ such that $C(w_1)$ and $C(w_2)$ are
isomorphic. Then:
\begin{enumerate}
\item The mapping $\varphi: C(w_1)\to C(w_2)$ defined by
$$
\varphi(a_1\dots a_{m-2}\bar w)=b_1\dots b_{n-2}\bar w \quad \textrm{ for }
\bar w\in \calA^\ast_{\geq 2}\textrm{ with }a_1\dots a_{m-2}\bar w\in C(w_1)
$$
is a bijection  which preserves the adjacency relation, that is, $p(w',w'')>0$ if and only if
$p\bigl(\varphi(w'),\varphi(w'')\bigr)>0$ for all $w',w''\in C(w_1)$.
\item The cones are isomorphic as subgraphs of $\mathcal{G}$.
\end{enumerate}
\end{Lemma}
\begin{proof}
Proof of (1): since $C(w_1)$ and $C(w_2)$ are isomorphic we have $C([w_1])=C([w_2])$, and thus
$[w_1]=a_{m-1}a_m\in C([w_1])=C([w_2])$. Hence, there is a path $\langle [w_2], u_1,\dots,u_k,a_{m-1}a_m\rangle$ through words $u_1,\dots,u_k\in \calA^\ast_{\geq 2}$.
%
If $w'=a_1\dots a_{m-2}\bar w\in C(w_1)$ with $\bar w\in \calA^\ast_{\geq 2}$ then there is a path $\langle
w_1,w_1',\dots,w_l',w'\rangle$ through words $w_1',\dots,w_l' \in \calA^\ast_{\geq |w_1|}$. This
yields that $w_i'$ has the form $w_i'=a_1\dots a_{m-2}w_i''$ with some
$w_i''\in\calA^\ast_{\geq 2}$, that is, the path $\langle
a_{m-1}a_m,w_1'',\dots,w_l'',\bar w \rangle$ has positive probability to be performed. But this implies that 
\begin{eqnarray*}
&&\langle w_2=b_1\dots b_{n-2}[w_2],b_1\dots b_{n-2}u_1,\dots,b_1\dots b_{n-2}u_k,
b_1\dots b_{n-2}a_{m-1}a_m,\\
&&\quad b_1\dots b_{n-2}w_1'',\dots,b_1\dots b_{n-2}w_l'',b_1\dots
b_{n-2}\bar w\rangle
\end{eqnarray*}
is a path through words in $\calA^\ast_{\geq |w_2|}$,
that is, $b_1\dots b_{n-2}\bar w\in C(w_2)$. Thus, $\varphi$ is well-defined.
\par
Since any $w\in C(w_1)$ and its image 
$\varphi(w)$ differ only by different (constant) prefixes the mapping $\varphi$
is obviously a bijection. 
Moreover, if $w=a_1\dots a_{m-2}c_1\dots c_k\in
C(w_1)$ with $c_1,\dots,c_k\in \calA$, $k\geq 2$, and 
$\hat w= a_1\dots a_{m-2}c_1\dots
c_{k-2}w'\in C(w_1)$ with $w'\in\calA^\ast$, $1\leq |w'|\leq 3$, and $(k-2)+|w'|\geq 2$ (otherwise $\hat w\notin C(w_1)$), then 
$$
p(w,\hat w)=p(c_{k-1}c_k,w')=p(b_1\dots b_{n-2}  c_1\dots c_k,b_1\dots b_{n-2}  c_1\dots c_{k-2}w')=p\bigl(\varphi(w),\varphi(\hat w)\bigr).
$$
This yields $(1)$.
\par
Proof of $(2)$: this follows directly from $(1)$ by the bijection $\varphi$ and
 the fact that the adjacency relation is given through positive single-step transition probabilities. Hence, $C(w_1)$ and $C(w_2)$ are
isomorphic as subgraphs of $\mathcal{G}$. 
\end{proof}
The lemma says implicitly that the words of two isomorphic cones differ only by
different prefixes. Moreover, there is a natural 1-to-1 correspondence of paths
inside $C(w_1)$ and paths in an isomorphic cone $C(w_2)$ where obviously each
such path in $C(w_1)$ and the corresponding path in the other isomorphic cone $C(w_2)$ have the same
probability: let be $\langle w_0',w_1',\dots,w_m'\rangle$ a path in $C(w_1)$;
then $\langle\varphi(w_0'),\varphi(w_1'),\dots,\varphi(w_m')\rangle$ is a path
in $C(w_2)$ and
$$
\Prob\bigl[X_1=w_1',\dots,X_m=w_m'\bigl| X_0=w_0']
= \Prob\bigl[X_1=\varphi(w_1'),\dots,X_m=\varphi(w_m')\bigl| X_0=\varphi(w_0')].
$$
We remark that $C(w)$ and $C(w')$,
$w,w'\in\calA^\ast_{\geq 2}$, with $C([w])\neq
C([w'])$ can still be isomorphic as subgraphs of
$\mathcal{G}$ but we will still distinguish them as elements of different isomorphism
classes according to our definition of isomorphism of cones.
\par
%
%
Our construction of cones ensures that different cones are either nested in each other or disjoint as the next lemma will show:
\begin{Lemma}\label{lem:cone-properties}
Let be $w_1,w_2\in\calA^\ast_{\geq 2}$. Then
the cones $C(w_1)$ and $C(w_2)$ are either nested in each other, that is,
  $C(w_1)\subseteq C(w_2)$ or $C(w_2)\subseteq C(w_1)$, or they are disjoint, that is, $C(w_1)\cap C(w_2)=\emptyset$. If we even have $|w_1|=|w_2|$ then we have $C(w_1)=C(w_2)$ or $C(w_1)\cap C(w_2)=\emptyset$.
\end{Lemma}
\begin{proof}
W.l.o.g. assume that $|w_1|\leq |w_2|$. Moreover, assume that the cones $C(w_1)$ and $C(w_2)$ are not nested in
each other and that $C(w_1)\cap C(w_2)\neq \emptyset$. Let be $w_0\in C(w_1)\cap
C(w_2)$. Then there is a path
$\langle w_1,w_1',\dots,w_k',w_0\rangle$ through words
$w_1',\dots,w_k'\in\calA^\ast_{\geq |w_1|}$ and there is a path $\langle w_2,w_1'',\dots,w_l'',w_0\rangle$ through words
$w_1'',\dots,w_l''\in\calA^\ast_{\geq |w_2|}\subseteq \calA^\ast_{\geq |w_1|}$. By weak symmetry, there is a
path $\langle w_1,w_1',\dots,w_k',w_0,w_l'',\dots,w_1'',w_2\rangle$ through
words in $\calA^\ast_{\geq |w_1|}$, and hence $w_2\in C(w_1)$ which in turn implies
$C(w_2)\subseteq C(w_1)$, a contradiction. This yields the first part of the lemma.
\par
In order to prove the second part assume that $|w_1|=|w_2|$ and w.l.o.g. $C(w_1)\subseteq C(w_2)$. It remains to show that we have then $C(w_1)=C(w_2)$. Since $w_1\in C(w_2)$ there is a path $\langle w_2,\bar w_1,\dots,\bar w_m,w_1\rangle$ through words $\bar w_1,\dots,\bar w_m\in\calA^\ast_{\geq |w_2|}$.
If $w\in C(w_2)$ then there is a path $\langle w_2,\hat w_1,\dots,\hat w_n,w\rangle$ through words $\hat w_1,\dots,\hat w_n\in\calA^\ast_{\geq |w_2|}$. Thus, there is a path
$$
\langle w_1,\bar w_m,\dots,\bar w_1,w_2, \hat w_1,\dots,\hat w_n,w\rangle
$$
though words in $\calA^\ast_{\geq |w_2|}=\calA^\ast_{\geq |w_1|}$. Hence, $C(w_2)\subseteq C(w_1)$ which yields $C(w_2)= C(w_1)$.
\end{proof}
At this point let us mention that the weak symmetry
Assumption \ref{ass:weak} is
crucial here: if this assumption is dropped then two cones $C(w_1)$ and
$C(w_2)$, where $w_1,w_2\in\calA^\ast_{\geq 2}$ with $|w_1|=|w_2|$ and $C(w_1)\cap C(w_2)\neq \emptyset$ may be non-isomorphic. This
case makes everything much more difficult in our proofs since the property of cones from the last lemma (either nestedness or disjointness) is lost and since we want to track
the random walk's way to infinity by distinguishing which of the (disjoint) cones are
successively finally entered on its way to infinity. The author is however confident that one can adapt the situation if weak symmetry does not hold but this would need much more effort with loss of good readability of our proofs and no additional gain of the techniques; for further comments see Appendix \ref{sub:weak-symmetry}.
\par
Since isomorphism of cones depends only on the last two letters of their
roots, we have obviously only finitely many different isomorphism classes of
cones. These isomorphism classes can be described by two-lettered words
$ab\in\calA^2$: first, for each isomorphism class of cones we
fix some $ab$ representing the class of $C(ab)$.
 Let $\mathcal{J}\subseteq \calA^2$ be a system of
representants of the different isomorphism classes of cones. Thus, for every $w\in\calA^\ast_{\geq
  2}$ there is some unique $ab\in \mathcal{J}$ such that $C([w])=C(ab)$. 
Then we write $\tau\bigl(C(w)\bigr):=ab$ for the \textit{cone type} (or
\textit{isomorphism class}) of the cone
$C(w)$. 
The \textit{boundary} of $C(w)$ is given by the set 
$$
\partial C(w)=\bigl\lbrace w_0\in
C(w)\, \bigl|\, |w_0|=|w|,\exists w'\in\calA^\ast\setminus C(w): p(w,w')>0\bigr\rbrace.
$$
We have $\{[w] \mid w\in\partial C(w_1)\}=\{[w] \mid
w\in\partial C(w_2)\}$ for two ismorphic cones $C(w_1)$ and $C(w_2)$ with
$w_1,w_2\in\calA^\ast_{\geq 2}$, which follows from the following fact: if $x_1\in \partial
C(w_1)$ and $w'\in\calA^\ast\setminus C(w_1)$ with $p(x_1,w')>0$, then there is, due to
\ref{lemma:cone-isomorphism}.(1), some $x_2\in C(w_2)$ with $[x_1]=[x_2]$ 
and $p([x_2],a)=p([x_1],a)>0$, where $a\in\calA$ is the last letter of
$w'$. This implies existence of some $w''\in\calA^\ast \setminus C(w_2)$ with
$p(x_2,w'')>0$.
\par
We say that the graph $\mathcal{G}$ is \textit{expanding} 
if each cone $C(w_0)$, $w_0\in\calL$, contains two proper
disjoint subcones, that is, if there exist subcones $C(w_1),C(w_2)\subsetneq C(w_0)$,
$w_1,w_2\in\calL$, with $C(w_1)\cap C(w_2)=\emptyset$. We call the random walk
\textit{expanding} if the associated graph $\mathcal{G}$ is expanding.
The results below do \textit{not} depend on whether the random walk is expanding or not. 
At the end, however, we will see that the non-expanding case leads to zero
entropy.
\par
Finally, let us remark that in the case of $K$-dependent random walks on
$\calA^\ast$  suffix-irreducibi\-lity can be defined analogously and cones can be defined in the exactly same way; the different
cone types would be defined by words of length $K$. In Appendix
\ref{sub:K-dependent} we will check that suffix-irreducibility and the the
``expanding'' property are inherited by the blocked letter language if these
properties are satisfied for the $K$-dependent random walk.

\subsection{Covering of Cones by Subcones}
\label{subsec:covering}

The next task is to cover (up to a finite complement) any cone $C(w)$, $w\in\calL$, by a
finite set of pairwise disjoint subcones
$C_1,\dots ,C_{n(w)}\subset C(w)$ such that 
$$
\bigl\lbrace \tau(C_1),\dots,\tau(C_{n(w)})\bigr\rbrace =\mathcal{J} \quad \textrm{ and } \quad \Bigl|C(w)\setminus
\bigcup_{i=1}^{n(w)} C_i\Bigr|<\infty,
$$
that is, every cone type appears among these subcones and the subcones cover
$C(w)$ up to finitely many words.  We then call
$C_1,\dots,C_{n(w)}$ a \textit{covering} of $C(w)$. 
In the next subsection we show how to
construct this covering when $\mathcal{G}$ is expanding; in Subsection \ref{sub:non-expanding} we consider the case when 
$\mathcal{G}$ is  \textit{not} expanding.

\subsubsection{Covering for Expanding Random Walks}

Suppose we are given a cone $C(w)$ with $w=w_0a_0b_0\in\calL$, where 
$w_0\in\calA^\ast$ and $a_0b_0\in\calA^2$. Inside this cone we can find subcones
of the form $C(w_0w'ab)$ for \textit{each} $ab\in\calA^2$ with suitable
$w'\in\calA^\ast\setminus\{o\}$ depending on $ab$ due to suffix-irreducibility. Now we want to find subcones  of each type $ab\in\mathcal{J}$ which are even pairwise disjoint. We proceed as follows to find these pairwise disjoint cones of all types:
since we assume in this subsection that $\mathcal{G}$ is expanding there are paths from $w=w_0a_0b_0$
inside $\calA^\ast_{\geq |w|}$ to words $w_0w_1a_1b_1$ and $w_0w_2a_2b_2$,
where $w_1,w_2\in\calA^\ast\setminus \{o\}$, $a_1b_1,a_2b_2\in\calA^2$ and
$C(w_0w_1a_1b_1)\cap C(w_0w_2a_2b_2)=\emptyset$. Then we have found a subcone
of type $\tau(C(a_1b_1))$, and we search for other cone types in the
subcone $C(w_0w_2a_2b_2)$ in the same way. Obviously, a subcone in $C(w_0w_2a_2b_2)$ does not
intersect $C(w_0w_1a_1b_1)$. Iterating this step leads to a finite set
$\{C_1,\dots,C_{|\mathcal{J}|}\}$ of subcones of $C(w)$ such that
$\{\tau(C_1),\dots,\tau(C_{|\mathcal{J}|})\}=\mathcal{J}$ and $C_i\cap C_j=\emptyset$ for
$i,j\in\{1,\dots,|\mathcal{J}|\}$ with $i\neq j$.
After we have found these non-intersecting subcones of
all types in $C(w)$ we cover the cone $C(w)$ by further disjoint subcones: let be
$D=1+\max \{|w'| \mid w'\in
  \bigcup_{i=1}^{|\mathcal{J}|} \partial C_i\}$;
%
define $M_D=\{w'\in C(w) \mid  |w'|= D\}$. Then we can choose a subset
      $M:=\{w_1',\dots,w_k'\}\subseteq M_D$ such that for all
      $i,j\in\{1,\dots,k\}$ with $i\neq j$ and all
      $n\in\{1,\dots,|\mathcal{J}|\}$ we have: $C(w_i')\cap C_n=\emptyset$, 
$C(w'_i)\cap C(w'_j)=\emptyset$ and 
$$
C(w)\setminus \biggl(
\bigcup_{m=1}^{\mathcal{|J|}} C_m \cup \bigcup_{n=1}^{k} C(w_n')\biggr)
$$
is finite. This is done as follows: write $M_D:=\{x_1,\dots,x_N\}$ and set
$M_0:=\emptyset$. For every $i\in \{1,\dots,N\}$, perform the following steps with increasing $i$: 
if $x_i\in \bigcup_{j=1}^{|\mathcal{J}|}C_j \cup \bigcup_{x\in M_{i-1}}C(x)$, then drop $x_i$ and set $M_i:=M_{i-1}$. 
Otherwise, set $M_i:=M_{i-1}\cup \{x_i\}$. In the latter case we cannot have
$C_j\subset C(x_i)$ for some $j\in\{1,\dots,|\mathcal{J}|\}$ due to the choice 
of $D$ (words in $\partial C_j$ have word length smaller than $D$ and all words
in $C(x_i)$ have length of at least $D$) and also not $C(x_i)\subset C_j$, which
would lead to the contradiction $x_i\in C_j$ otherwise. We also cannot have $C(x_j)\subset C(x_i)$ for $j<i$ because this implies, by Lemma \ref{lem:cone-properties}, $C(x_i)=C(x_j)$ and therefore $x_i\in C(x_j)$.
 At the end of this procedure we get some $M_N$ and set $M:=M_N$. Since every
 path from $w$ to infinity inside $C(w)$ has to pass through a word of length $D$ we have
 ensured that each $w'\in C(w)$ with $|w'|=D$ lies in one of the cones
 $C_1,\dots,C_{\mathcal{J}}, C(x)$, $x\in M$. Thus, the set $C(w)\setminus
 \bigcup_{m=1}^{|\mathcal{J}|} C_m \cup \bigcup_{x\in M} C(x)$ is finite and the covering of 
 $C(w)$ is given by the subcones
$$
C_1,\dots,C_{\mathcal{|J|}}, C(x), x\in M.
$$
See Figure \ref{cone-covering} for better visualization.
\begin{figure}[htp]
\begin{center}
\includegraphics[width=7cm]{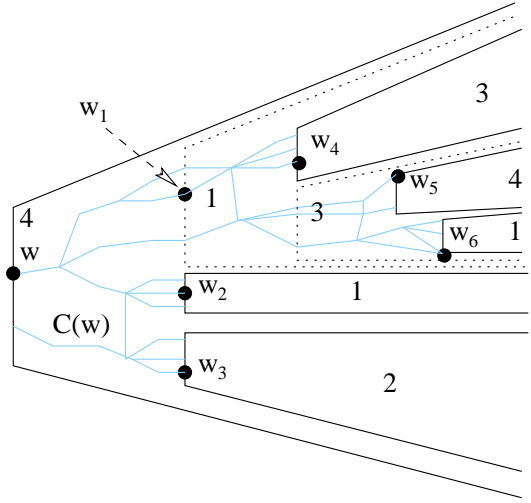}
\end{center}
\caption{Covering of cones by subcones: the numbers represent the four different
  cone types; the cones with the solid boundary lines belong to the covering of
  $C(w)$. The construction of a covering is done as follows: e.g., we find three
  three cones in $C(w)$ whose union covers $C(w)$ up to a finite set, say the
  cones $C(w_1)$ (type $1$), $C(w_2)$ (type $1$) and $C(w_3)$ (type
  $2$). We keep the cones $C(w_2)$ and $C(w_3)$ for the covering of $C(w)$ and
  search for cones of type $3$ and $4$ in the subcone $C(w_1)$. After having found cones of type
$3$ and $4$ in $C(w_1)$ (for instance, the cones $C(w_4)$ and $C(w_5)$) we take additional disjoint cones in $C(w_1)$ (in
the picture the innermost type-1 cone $C(w_6)$ only) into the covering such that the
complement of the union of all subcones in the covering is finite. That is, the
covering of $C(w)$ consists of the cones $C(w_2)$, $C(w_3)$, $C(w_4)$, $C(w_5)$
and $C(w_6)$.}
\label{cone-covering}
\end{figure}
\par
The crucial point now is that we fix a covering for each cone type such that
the relative positions of the subcones in the 
covering of some cone $C(w)$ do \textit{not} depend on the choice of the specific root $w\in\calL$ on
the boundary of $C(w)$ but only on $\tau(C(w))$: first, for each $ab\in\mathcal{J}$, choose any $w_{ab}\in\calA^\ast$
such that $w_{ab}ab\in\calL$ and fix some covering for $C(w_{ab}ab)$, say the
cones
$C(w_{ab}v_1),\dots, C(w_{ab}v_k)$, where
$v_1,\dots,v_k\in\mathcal{A}^\ast_{\geq 3}$. If
$w=w_0a_1b_1\in\calL$ with $w_0\in\calA^\ast$, $a_1b_1\in\calA^2$ and
$\tau(C(w))=ab=\tau(C(w_{ab}ab))$ then we set the covering of
$C(w)$ as the one which is inherited from the covering of $C(w_{ab}ab)$ by
the relative location of the subcones, that is, we set the covering of $C(w)$
as the set of subcones $C(w_0v_1),\dots, C(w_{0}v_k)$.
\begin{Lemma}
The set of  subcones $C(w_0v_1),\dots, C(w_{0}v_k)$ is a covering of $C(w)$.
\end{Lemma}
\begin{proof}
First, $C(w_0v_1),\dots, C(w_{0}v_k)$ are subcones of $C(w)$ since $ab\in
C([w])$ (yielding $w_0ab\in\partial C(w)$) and due to the following conclusion: for
each $i\in\{1,\dots,k\}$, there is a path from $w_{ab}ab$ to $w_{ab}v_i$
through words in $\calA^\ast_{\geq |w_{ab}ab|}$, which implies that there is a
path from $ab$ to $v_i$ through words in $\calA^\ast_{\geq 2}$ yielding
existence of a path from  
$w=w_0[w]$
via $w_0ab$ to $w_0v_i$ through words in $\calA^\ast_{\geq |w|}$. That is,
$C(w_0v_i)\subset C(w)$.
\par
Since $\mathcal{J}=\{\tau(C(v_1)),\dots,\tau(C(v_k))\}$ the set of subcones $\{C(w_0v_1),\dots,C(w_0v_k)\}$
contains all different types. The next step is to show disjointness of the
cones $C(w_0v_1),\dots, C(w_{0}v_k)$. Assume w.l.o.g. that
$C(w_0v_1)\subsetneq C(w_0v_2)$. Then there exists a path from $w_0v_2$ to
$w_0v_1$ through words in $\calA^\ast_{\geq |w_0v_2|}$. This implies that there
exists a path from $v_2$ to
$v_1$ through words in $\calA^\ast_{\geq |v_2|}\subseteq \calA^\ast_{\geq 3}$, which implies that there
exists a path from $w_{ab}v_2$ to
$w_{ab}v_1$ through words in $\calA^\ast_{\geq |w_{ab}v_2|}$ yielding
$C(w_{ab}v_1)\subseteq C(w_{ab}v_2)$, a contradiction to the choice of
$C(w_{ab}v_1)$, $C(w_{ab}v_2)$ in the covering of $C(w_{ab}ab)$. Thus, the cones
$C(w_0v_1),\dots, C(w_{0}v_k)$ are pairwise disjoint.
\par
Analogously, we show that $C(w)\setminus \bigcup_{i=1}^k C(w_0v_i)$ is
finite. Assume that this set difference is \textit{not} finite. Then for every
$N\in\mathbb{N}$ with $N\geq 3$, there exists some $\bar w_N\in\calA^\ast$ with
$|\bar w_N|=N$ and $w_0\bar w_N\in\calA^\ast\cap
\overline{\bigcup_{i=1}^k C(w_0v_i)}$ 
such that there is a path from $w=w_0[w]$ to $w_0\bar w_N$
through words in $\calA^\ast_{\geq |w|}$. Since $[w]\in C(ab)$ there is a path
from $ab$ to $[w]$ through words in $\calA^\ast_{\geq 2}$ implying that there
exists a path from $ab$ to $\bar w_N\in
\overline{\bigcup_{i=1}^k C(v_i)}$ through words in $\calA^\ast_{\geq 2}$. But
this implies that there
exists a path from $w_{ab}ab$ to $w_{ab}\bar w_N\in
\overline{\bigcup_{i=1}^k C(w_{ab}v_i)}$ through words in $\calA^\ast_{\geq
  |w_{ab}ab|}$. This gives a contradiction since $C(w_{ab}ab)\setminus
\bigcup_{i=1}^k C(w_{ab}v_i)$ is finite and therefore $N$ cannot be large. This yields the claim.
\end{proof}
Hence, the covering of a cone depends only on its cone type, which describes the relative location
of its subcones in its interior.
\par
We can also cover $\calL$ (up to a finite set) by a finite number of
non-intersecting subcones, where each cone type appears. To this end, we
just apply the algorithm explained above and take pairwise disjoint cones of the form $C(w)$ with
$w\in\calL$ and $|w|\geq 2$. We denote by
$C_1^{(0)},\dots,C_{n_0}^{(0)}$ the covering of $\calL$, which contains all
types in $\mathcal{J}$ and which satisfies $\big|\calL \setminus \bigcup_{i=1}^{n_0} C_i^{(0)} \bigr|<\infty$.

\subsubsection{Non-Expanding Random Walks}
\label{sub:non-expanding} 

Now we explain how to proceed if $\mathcal{G}$ is \textit{not} expanding, that is, there
is a cone $C(w)$, $w\in\calL$, which does \textit{not} contain two proper
disjoint subcones. Recall that due to suffix-irreducibility there is, for every
$ab\in\mathcal{J}$, a subcone $C(w_1)\subset C(w)$ with $[w_1]=ab$. Thus, all
cones do \textit{not} have two proper disjoint subcones, because otherwise we
get a contradiction to the choice of $w$. 
This non-expanding case may, in particular, occur if $\calL$ is a proper subset of $\calA^\ast$.
%
Take now disjoint cones $C(a_1b_1),\dots,C(a_db_d)$, where $d\in\N$,
$a_1b_1,\dots,a_db_d\in\calA^2$ with $C(a_ib_i)\cap C(a_jb_j)=\emptyset$ for
all $i,j\in\{1,\dots,d\}$ with $i\neq j$ and $\calL \setminus \bigcup_{k=1}^d
C(a_kb_k)$ is finite. As already mentioned above the cones $C(a_ib_i)$,
$i\in\{1,\dots,d\}$, do \textit{not} contain two proper disjoint
subcones. Thus, we can then cover any cone $C(w)$, $w\in\calA^\ast_{\geq 2}$, by the subcone
$C(w_1)$ for any $w_1\in C(w)$ with $|w_1|=|w|+1$ and $p(w,w_1)>0$.
\par
\begin{Ex}\label{ex:nonexpanding}
In order to illustrate this situation we give a short example for this case:
let $\calA=\{a,b\}$,
$p(o,a)=p(a,o)=p(o,b)=p(b,o)=p(a,ab)=p(b,ba)=\frac12$ and
$p(ab,aba)=\frac{2}{3},p(ba,b)=\frac{1}{3}$,
$p(ba,bab)=\frac{3}{4},p(ab,a)=\frac{1}{4}$. The set $\calL$
is then given by all words of the form $ababa\dots ba$, $ababa\dots bab$,
$baba\dots bab$ and $baba\dots baba$. The random walk is transient and
satisfies the Assumptions \ref{ass:weak} and \ref{ass:suffix}. We have
$C(ab)\cap C(ba)=\emptyset$ and $C(ab)=C(aba)\cup \{ab\}$ and $C(ba)=C(bab)\cup \{ba\}$.
\end{Ex}
The next step is to show that a non-expanding random walk converges to one of
finitely many infinite words. More precisely, since we consider transient
random walks, $|X_n|$ tends almost surely to infinity. Therefore, we must have that the
prefixes of arbitrary length of $X_n$
stabilize for $n$ large enough, that is, for each $N\in\N$ there exists almost
surely some index $n_N\in\N$ such that the prefixes of length $N$ of $X_{n_N},X_{n_N+1},
X_{n_N+2}, \dots$, remain constant forever. Thus, $(X_n)_{n\in\N_0}$ tends to
some infinite (random) word $X_\infty\in\calA^{\N}$.
\begin{Lemma}
If $(X_n)_{n\in\N_0}$ is non-expanding, then the support of $X_\infty$ is finite.
\end{Lemma}
\begin{proof}
First, assume that $X_\infty$ starts with positive probability with the letter
$a_0\in\calA$. Assume also that $\Prob[\forall
n\geq 1: X_n\in C(a_0b_0c_0) \mid X_0=a_0b_0c_0]>0$ for some
$b_0c_0\in\calA^2$ with $a_0b_0c_o\in\calL$. We denote by $A$ the event that $X_\infty$ starts with the letter
$a_0$ and that the random walk finally enters $C(a_0b_0c_0)$ on its way to
infinity. Then $\mathbb{P}[A]>0$.
On this event $A$, assume now that the random walk tends with
positive probability to some
infinite words with prefixes $wa_1$ and $wa_2$, where $w\in\calA^\ast_{\geq 2}$
starts with the letter $a_0$ and $a_1,a_2\in\calA$ with $a_1\neq a_2$.
Then there must be words $wa_1b_1c_1, wa_2b_2c_2\in C(a_0b_0c_0)$,
$b_1c_1,b_2c_2\in\calA^2$, such that 
$$
\mathbb{P}\bigl[ \exists n\in\N: X_n=wa_ib_ic_i,\forall m\geq n: X_m\in
C(wa_ib_ic_i) \,\bigl|\, A\bigr]>0 \textrm { for } i\in\{1,2\}.
$$
Obviously, $C(wa_1b_1c_1)\cap C(wa_2b_2c_2)=\emptyset$. But this leads to the
contradiction that $C(a_0b_0c_0)$ has two proper disjoint subcones.
Therefore, $C(wa_1b_1c_1)\cap \calL=\emptyset$ or $C(wa_2b_2c_2)\cap
\calL=\emptyset$, yielding that the letter $a_1$
(or $a_2$) is deterministic on the event $A$. By induction, the infinite limiting word
$X_\infty$ is deterministic on the event $A$, and it depends only on $a_0$ and $b_0c_0$. Since there
are only finitely many possibilities for $a_0$ and $b_0c_0$, the limiting word $X_\infty$
can only take finitely many values.
\end{proof}
The last lemma and suffix-irreducibility directly imply that the support of the random walk
is a proper subset of $\calA^\ast$ if $(X_n)_{n\in\N_0}$ is non-expanding. The
limiting words in Example \ref{ex:nonexpanding} are $ababab\ldots$ and $bababa\ldots$.
%
%
%

\section{Last Entry Times}
\label{sec:exit-times}
In this section we prove a law of large numbers, which turns out to describe the asymptotic
entropy in the later section. For this purpose,
we define last entry times (compare with \cite{gilch:08}), for which we derive a law of large
numbers. In this section we will assume that $(X_n)_{n\in\mathbb{N}_0}$ is
transient and we will assume Assumptions \ref{ass:weak} and
\ref{ass:suffix}, where we make explicit comments when these assumptions are
essential at some points. Throughout this section, we will also use the following notations:
$w_0,w_1,w_2\in\calA^\ast\setminus\{o\}$ and $a,b,c,d,a_1,b_1,a_2,b_2,\ldots\in\calA$.

\subsection{Last Entry Time Process}
\label{subsec:lastentry}
We define the following \textit{last entry times}. Let $\mathbf{e}_0$ be the first
time at which the random
walk visits $\bigcup_{i=1}^{n_0} \partial C^{(0)}_i$ and stays in one of the
cones $C^{(0)}_1,\dots,C^{(0)}_{n_0}$ afterwards forever, that is,
$$
\e{0} := \inf \bigl\lbrace m\in\N_0 \,\bigl|\, \exists i\in\{1,\dots,n_0\}
\, \forall n\geq m: X_n\in C_i^{(0)} \bigr\rbrace.
$$
In particular, $X_{\e{0}}\in \bigcup_{i=1}^{n(0)} \partial C_i^{(0)}$ and
$X_{\e{0}-1}\notin \bigcup_{i=1}^{n(0)} C_i^{(0)}$. In other words, at time
$\e{0}$ the random walk finally enters one of the cones $C_i^{(0)}$ with no further exits.
Inductively, if $X_{\mathbf{e}_k}=w\in\calL$ for $k\geq 0$ and if $C(w)$ has the covering (determined only
by the type of $C(w)$) consisting of the subcones $C_1^{(k)},\dots,C_{n(w)}^{(k)}$
as explained in Section \ref{sec:cones}, then
$$
\e{k+1}:= \inf \bigl\lbrace m>\e{k} \,\bigl|\, \exists i\in\{1,\dots,n(w)\}
\, \forall n\geq m: X_n\in C_i^{(k)}\bigr\rbrace.
$$
In particular, $X_{\e{k+1}}\in \bigcup_{i=1}^{n(w)} \partial C_i^{(k+1)}$ and
$X_{\e{k+1}-1}\notin \bigcup_{i=1}^{n(w)} \partial C_i^{(k)}$.
Transience of $(X_n)_{n\in\mathbb{N}_0}$ yields $\e{k}<\infty$ for all
$k\in\N_0$ almost surely.
Observe that $X_n$, $n\geq \e{k}$, has the prefix $w_0$ if $X_{\e{k}}=w_0ab$.
Define the \textit{relative increments} $(\W_k)_{k\in\N_0}$ between two last entry times as follows: 
set $\W_0:=X_{\e{0}}$; for $k\geq 1$: if $X_{\mathbf{e}_{k-1}}=w_0ab$ and $X_{\mathbf{e}_{k}}=w_0w_1cd$, then set
$\W_k:=w_1cd$. 
Since we
have only finitely many different cone types and the subcones of the covering of
any cone $C$ are nested at uniformly bounded distance (w.r.t. minimal path lengths) to $\partial C$, the random variables $\W_k$ can
take only finitely many different values. Observe that we can reconstruct the
values of the $X_{\e{k}}$'s from the values of the $\W_{k}$'s: if
$\W_{l}=w_la_lb_l$ for $l\leq k$ then $X_{\e{k}}=w_0w_1\dots w_ka_kb_k$.
\par
For $w\in\calL$, define
$$
\mathcal{S}(w):=\bigcup_{i=1}^{n(w)} \partial C_i,
$$
where $C_1,\dots,C_{n(w)}$ is the covering of $C(w)$ according to Section
\ref{sec:cones}. Observe that $\mathcal{S}(w_1)=\mathcal{S}(w_2)$ if $C(w_1)=C(w_2)$.
Define for
$x=a_1\dots a_k\in\calA^\ast$ and $y=a_1\dots a_{k-2}b_{k-1}b_k\dots b_{k+d}\in
C(x)$ with $d\geq 1$ and $d=d(x,y):=|y|-|x|$:
$$
\mathds{L}(x,y)  := \sum_{n\geq 0} \Prob\Bigl[X_n=y,X_{n-1}\notin C(y),\forall m\in\{1,\dots,n\}: X_m\in
C(x)\Bigl| X_0=x\Bigr].
$$
If $d=1$ then $\mathds{L}(x,y)=\bar L(a_{k-1}a_k,b_{k-1}b_{k}b_{k+1})$. If
$d \geq 2$ then $\mathds{L}(x,y)$ can be rewritten as
\begin{equation}\label{equ:mathdsL}
\sum_{\substack{y_1,\dots,y_{d-1}\in\calA^3:\\ y_i[1]=b_{k-2+i}}}\bar
L(a_{k-1}a_k,y_1)\cdot 
\prod_{j=1}^{d-2}
\bar L(y_j[2]y_j[3],y_{j+1})\cdot  \bar L(y_{d-1}[2]y_{d-1}[3],b_{k+d-2}b_{k+d-1}b_{k+d});
\end{equation}
the last equation follows from the fact that $\mathds{L}(x,y)$ depends on $x$
only by its last two letters $a_{k-1}a_k$ and by decomposition of the paths from
$x$ to $y$ w.r.t the last times
when the sets $\calA^{k}$, $\calA^{k+1},\dots,\calA^{k+d-1}$ are visited on the
way from $x$ to $y$. That is, the $l$-th factor in (\ref{equ:mathdsL})
corresponds to the part of the path from $x$ to $y$ between the last entry of
$\calA^\ast_{\geq k+l-1}$ at the word $a_1\dots a_{k-2}b_{k-1}\dots b_{k+l-3}y_{l-1}[2]y_{l-1}[3]$ and
the last entry to $\calA^\ast_{\geq k+l}$ at the word $a_1\dots a_{k-2}b_{k-1}\dots
b_{k+l-2}y_{l}[2]y_{l}[3]$ (with $y_0[2]y_0[3]=a_{k-1}a_k$ and $y_d=b_{k-2}b_{k-1}b_k$). Moreover, $\mathds{L}(x,y)=\mathds{L}(a_{k-1}a_k,b_{k-1}b_k\dots b_{k+d})$.
\par
If $x_1\in \calL$, $x_2\in\mathcal{S}(x_1)$ and  $x_3\in\mathcal{S}(x_2)$ then 
$$
\mathds{L}(x_1,x_3) = \sum_{y\in \partial C(x_2)} \mathds{L}(x_1,y)\cdot \mathds{L}(y,x_3)
$$
by decomposition w.r.t. the last visit of the set $\partial C(x_2)$ since
$C(x_3)\subset C(x_2)\subset C(x_1)$. In particular, if
$\mathbb{P}[X_{\e{k}}=x_1,X_{\e{k+1}}=x_2,\dots,X_{\e{k+l}}=x_{l+1}]>0$ for
$x_1,\dots,x_{l+1}\in\calL$ then we have
\begin{eqnarray}
&&\mathbb{P}[X_{\e{k}}=x_1,X_{\e{k+1}}=x_2,\dots, X_{\e{k+l}}=x_{l+1}]\nonumber \\
&=&
\sum_{x_0\in\calL \setminus C(x_1)} G(o,x_0 | 1)\cdot p(x_0,x_1) \cdot
\mathds{L}(x_1,x_2)\cdot\ldots \cdot  \mathds{L}(x_l,x_{l+1})\cdot \xi([x_{l+1}])\label{equ:X-decomposition}
\end{eqnarray}
by decomposition on the final entries of the cones $C(x_1),\dots,C(x_{l+1})$.
We obtain the following important observation:
\begin{Prop} \label{prop:W-process}
The process $\bigl(\W_{k}\bigr)_{k\geq 1}$ is a Markov chain with transition probabilities
$$
q(x,y) := \begin{cases}
\frac{\xi([y])}{\xi([x])}\mathds{L}(x,y), & \textrm{if } y\in\mathcal{S}(x),\\
0, & \textrm{otherwise}.
\end{cases}
$$
\end{Prop}
\begin{proof}
Let be $w_0,\dots,w_{k+1}\in\calA^\ast\setminus\{o\}$ such that 
$w_0\in \bigcup_{j=1}^{n_0}\partial
C_j^{(0)}$, $w_{i+1}\in S(w_i)$ for all
$i\in\{0,\dots,k\}$ and $\Prob[\W_{0}=w_0,\dots, \W_{k+1}=w_{k+1}]>0$. For any such sequence \mbox{$\underline{w}=(w_0,\dots,w_{k+1})$,} we set $x_0(\underline{w}):=w_0$ and inductively: if
$x_{k-1}(\underline{w})=y_{k-1}a_{k-1}b_{k-1}$ with $y_{k-1}\in\calA^\ast$ and
$a_{k-1}b_{k-1}\in\calA^2$ then set $x_k(\underline{w}):=y_{k-1}w_k$. That is,
if $\W_{k}=w_k$ then $X_{\e{k}}=x_k(\underline{w})$. 
Then:
\begin{eqnarray*}
&& \mathbb{P}\bigl[\W_1=w_1,\dots,\W_k=w_k\bigr]
= \sum_{w_0\in \bigcup_{j=1}^{n_0}\partial C_j^{(0)}}
\mathbb{P}\bigl[\W_0=w_0,\dots,\W_k=w_k\bigr]\\
&=& \sum_{w_0\in \bigcup_{j=1}^{n_0}\partial C_j^{(0)}} \mathbb{P}\bigl[X_{\e{0}}=w_0,X_{\e{1}}=x_1(\underline{w}),\dots,X_{\e{k}}=x_k(\underline{w})\bigr]\\
&=&
\sum_{w_0\in \bigcup_{j=1}^{n_0}\partial C_j^{(0)}} \sum_{w'\in\calL\setminus C(w_0)}
G(o,w'|1)\cdot p(w',w_0)\cdot \prod_{i=1}^k 
\mathds{L}(x_{i-1}(\underline{w}),x_i(\underline{w}))\cdot \xi([x_k(\underline{w})])\\
&=&
\sum_{w_0\in \bigcup_{j=1}^{n_0}\partial C_j^{(0)}} \sum_{w'\in\calL\setminus C(w_0)}
G(o,w'|1)\cdot p(w',w_0)\cdot \prod_{i=1}^k 
\mathds{L}(w_{i-1},w_i)\cdot \xi([w_k]).
\end{eqnarray*}
The last equation arises from (\ref{equ:X-decomposition}) by decomposing
the paths by the last entries to the sets
$\partial C_i$, where $C_i$ denotes the cone with $X_{\e{i}}\in\partial C_i$.
Now we obtain:
\begin{eqnarray*}
&&\mathbb{P}\bigl[\W_{k+1}=w_{k+1} \mid \W_1=w_1,\dots,\W_k=w_k\bigr]\\
&=&
\frac{\mathbb{P}\bigl[\W_1=w_1,\dots,\W_k=w_k,\W_{k+1}=w_{k+1}\bigr]}{\mathbb{P}\bigl[\W_1=w_1,\dots,\W_k=w_k\bigr]}\\
&=& \frac{\sum_{w_0\in \bigcup_{j=1}^{n_0}\partial C_j^{(0)}} \sum_{w'\in \calL\setminus C(w_0)}
G(o,w'|1)\cdot p(w',w_0)\cdot \prod_{i=1}^{k+1} \mathds{L}(w_{i-1},w_i)\cdot
\xi([w_{k+1}])}{\sum_{w_0\in \bigcup_{j=1}^{n_0}\partial C_j^{(0)}} \sum_{w'\in\calL\setminus C(w_0)}
G(o,w'|1)\cdot p(w',w_0)\cdot \prod_{i=1}^k\mathds{L}(w_{i-1},w_i)
\cdot \xi([w_k])}\\
&=&q(x,y).
\end{eqnarray*}
\end{proof}
Define the set
$$
\mathcal{W}_0:=\bigl\lbrace w\in\calA^\ast \bigl| 
\exists w_0\in\calA^\ast,ab\in\calA^2 \textrm{ with }  
\Prob[\W_{0}=w_0ab, \W_1=w]>0 
\bigr\rbrace /subseteq \calA^\ast_{\geq 3}.
$$
For the next proof we need the following properties: 
if $a_1b_1,a_2b_2\in\calA^2$ with $\tau(C(a_1b_1))=\tau(C(a_2b_2))$ then we
have $C(a_1b_1)=C(a_2b_2)$ (see Lemma \ref{lem:cone-properties}) and therefore
$a_2b_2\in C(a_1b_1)$. In this case we also have
$\mathds{L}(a_1b_1,w)>0$ for $w\in\calA^\ast_{\geq 3}$ if and only if $\mathds{L}(a_2b_2,w)>0$. This follows
from the simple fact that $a_2b_2\in C(a_1b_1)$ implies that there are paths
from $a_1b_1$ to $a_2b_2$ (and vice versa) through words in $\calA^\ast_{\geq
  2}$. 
\begin{Lemma}\label{lem:supp-W0}
For all $k\geq 1$, $\mathrm{supp}(\Prob[\W_k=\cdot])=\mathcal{W}_0$.
\end{Lemma}
\begin{proof}
By definition, we obviously have
$\mathrm{supp}(\Prob[\W_1=\cdot])=\mathcal{W}_0$. For $k>1$ we show both inclusions.
Let be $y\in\mathcal{W}_0$. Then there are $w_0\in\calA^\ast$ and
$ab\in\calA^2$ with $w_0ab\in\bigcup_{j=1}^{n_0}\partial C_j^{(0)}$ and
$w_0y\in\mathcal{S}(w_0ab)$ and
\begin{eqnarray*}
\Prob[\W_{0}=w_0ab,\W_{1}=y]&=& \sum_{w'\in\calL\setminus C(w_0ab)} G(o,w')\cdot
p(w',w_0ab)\cdot
\mathds{L}(w_0ab,w_0y)\cdot \xi([y])\\
&=&
\sum_{w'\in\calL\setminus C(w_0ab)} G(o,w')\cdot
p(w',w_0ab)\cdot
\mathds{L}(ab,y)\cdot \xi([y])>0.
\end{eqnarray*}
Take now any $\bar w\bar a \bar b\in\mathrm{supp}(\Prob[X_{\e{k-2}}=\cdot])$.
Since the covering of every cone contains subcones of all different types, the
cone $C(\bar w\bar a \bar b)$ has in its covering a cone of type $\tau(C(ab))$. Hence, there
are $w_k\in\calA^\ast$, $a_kb_k\in\calA^2$ with $\bar w w_ka_kb_k\in \mathcal{S}(\bar w\bar a \bar b)$,
$\tau(C(a_kb_k))=\tau(C(ab))$
and  $m_k\in\N$ such that $p^{(m_k)}(o,\bar w w_ka_kb_k)>0$. 
Thus,
\begin{eqnarray*}
\Prob[\W_{k}=y] &\geq & \Prob[X_{\e{k-1}}=\bar w w_ka_kb_k,\W_k=y]\\
&=& \sum_{w'\in\calL\setminus C(\bar w w_ka_kb_k)}
G(o,w')\cdot p(w',\bar w w_ka_kb_k)\cdot\mathds{L}(\bar w w_ka_kb_k, \bar w w_ky)\cdot \xi([y])\\
&=& \sum_{w'\in\calL\setminus C(\bar w w_ka_kb_k)}
G(o,w')\cdot p(w',\bar w w_ka_kb_k)\cdot\mathds{L}(a_kb_k, y)\cdot \xi([y]).
\end{eqnarray*}
By the remark before the lemma, we have $\mathds{L}(a_kb_k, y)>0$ and therefore $\Prob[\W_{k}=y] >0$, yielding $\mathcal{W}_0\subseteq \mathrm{supp}(\Prob[\W_k=\cdot])$.
\par
For the other direction, take any $y\in \mathrm{supp}(\Prob[\W_k=\cdot])$. 
Then there is some $w_{k-1}ab\in\calL$ such that 
\begin{eqnarray*}
0& < & \Prob[X_{\e{k-1}}=w_{k-1}ab,X_{\e{k}}=w_{k-1}y]\\
&=& \sum_{w'\in\calL \setminus C( w_{k-1}ab)}G(o,w')\cdot p(w',w_{k-1}ab)\cdot
\mathds{L}(w_{k-1}ab,w_{k-1}y)\cdot \xi([y]).
\end{eqnarray*}
In particular, $\mathds{L}(ab,y)>0$.
Since the initial covering of $\calL$ contains a cone of type
$\tau(C(ab))$ there are  $w_0\in\calA^\ast$, $a_0b_0\in\calA^2$ and
some $m\in\N$ such that $w_{0}a_0b_0\in\bigcup_{i=1}^{n_0}\partial C_i^{(0)}$, $\tau(C(a_0b_0))=\tau(C(ab))$ and 
$p^{(m)}(o,w_{0}a_0b_0)>0$. Observe again 
that $\mathds{L}(a_0b_0, y)>0$ by the remark before the lemma. Therefore,
\begin{eqnarray*}
\Prob[\W_{1}=y] &\geq & \Prob[\W_{0}=w_0a_0b_0,\W_1=y]
=\Prob[X_{\e{0}}=w_0a_0b_0,\W_1=y]\\
&=&\sum_{w'\in\calL\setminus C(w_0a_0b_0)} G(o,w')\cdot  p(w',w_0 a_0b_0)\cdot\mathds{L}(w_0a_0b_0, w_0y)\cdot \xi([y])\\
&=& \sum_{w'\in\calL\setminus C(w_0a_0b_0)} G(o,w')\cdot  p(w',w_0 a_0b_0)\cdot\mathds{L}(a_0b_0, y)\cdot \xi([y])>0.
\end{eqnarray*}
This yields $\mathrm{supp}(\Prob[\W_k=\cdot])\subseteq
\mathrm{supp}(\Prob[\W_1=\cdot])=\mathcal{W}_0$ and the claim of the lemma follows.
%
\end{proof}
With the last lemma we can show:
\begin{Lemma}\label{lemma:aperiodic}
The Markov chain $(\W_k)_{k\in\N}$ is positive recurrent and aperiodic.
\end{Lemma}
\begin{proof}
Since $\mathcal{W}_0$ is finite it suffices to show that the process
$(\W_k)_{k\in\N}$ is irreducible and aperiodic. First we show
irreducibility. Let be
$w_1=w'a_1b_1,w_2\in\mathcal{W}_0$. Then there is some $w_0 a_0 b_0\in\bigcup_{j=1}^{n_0}\partial C_j^{(0)}$ such that 
\begin{eqnarray*}
\Prob[\W_1=w_2]&\geq & \Prob[X_{\e{0}}=w_0a_0b_0,\W_1=w_2]\\
&=& \sum_{w'\in\calL\setminus C(w_0a_0b_0)} G(o,w')p(w',w_0a_0b_0) \mathds{L}(w_0a_0 b_0,w_0w_2)\xi([w_2])>0.
\end{eqnarray*}
In particular, $\mathds{L}(a_0b_0,w_2)=\mathds{L}(w_0a_0b_0,w_0w_2)>0$. By construction of coverings, $C(a_1b_1)$ has a
subcone  of type $\tau(C(a_0b_0))$ in its covering, say the cone
$C(\tilde w)$ with $\tilde w\in C(a_1b_1)\cap\mathcal{W}_0$ and $\mathds{L}(a_1b_1,\tilde
w)>0$. 
Then:
\begin{eqnarray}\label{equ:Wk-iireducible}
\Prob[\W_3=w_2 \mid \W_1=w_1] &\geq &  q(w_1,\tilde w) \cdot q(\tilde w, w_2)\\
&=&\mathds{L}(a_1b_1,\tilde w)\mathds{L}([\tilde w], w_2) \frac{\xi([w_2])}{\xi(a_1b_1)}>0, \nonumber
\end{eqnarray}
which follows from the fact that
$\mathds{L}([\tilde w], w_2)>0$ due to $[\tilde w]\in C(a_0b_0)$ and $\mathds{L}(a_0 b_0,
w_2)>0$ (recall the remark before Lemma \ref{lem:supp-W0}). This proves irreducibility and thus positive recurrence of $(\W_k)_{k\in\N}$.
\par
In order to see aperiodicity of the process $(\W_k)_{k\in\N}$ choose in the
proof above $w_1=w_2$, which yields that the period of $(\W_k)_{k\in\N}$ is
either $1$ or $2$. Now let be $w\in\mathcal{W}_0$ and take any $\hat
w\in\mathcal{W}_0$ with $q(w,\hat w)>0$. Then according to
(\ref{equ:Wk-iireducible}) we get
$$
\Prob[\W_4=w,\W_2=\hat w \mid \W_1=w]= q(w,\hat w)\cdot \Prob[\W_3=w \mid
\W_1=\hat w]>0,
$$
which implies aperiodicity.
\end{proof}

For sake of better identification of the cones, we now switch to a more
suitable representation of cones and coverings.
We identify the different cone types by numbers
$\mathcal{I}:=\{1,\dots,r\}\subset\mathbb{N}$. If $C(w)$ is a cone of type $i\in\mathcal{I}$, then the
 covering of $C(w)$ (according to Subsection \ref{subsec:covering}) has
 $n(i,j)$ subcones of type $j\in\mathcal{I}$. We denote these subcones of type
 $j$ by $C_{j_{i,k}}=C_{j_{i,k}}(w)\subset C(w)$ with $1\leq k \leq n(i,j)$ or we just identify them by $j_{i,1},\dots,j_{i,n(i,j)}$, which correspond to the subcones of type
$j$ with different locations inside $C(w)$. In particular, we choose this
enumeration of the subcones of type $j$ in a consistent way: if $C(w_{ab}v_m)$
belongs to the covering of $C(ab)$, $i=\tau(C(ab))$, with $C(w_{ab}v_m)$ being the $k$-th cone of
type $j$ in the covering of $C(ab)$ (identified by $j_{i,k}$ w.r.t. $ab$), then
the $k$-th subcone of type $j$ in the covering of any cone $C(w_0ab)$ is the
subcone $C(w_0v_m)$; compare with the construction of the covering of any cone
$C(w)$ starting from the covering of the cone $C(w_{ab}ab)$  in Subsection
\ref{subsec:covering}. That is, by this enumeration of subcones we ensure that
the relative position of $C_{j_{i,k}}(w)$ in the interior of $C(w)$ is always
the same for any $w\in\calL$ with $i=\tau(C(w))$.
We will sometimes omit the root $w$ in the
notation of the subcones when it will be clear from the context and when only the
relative position of a subcone in some given cone will be of importance.
\par 
We now track the random walk's way to infinity by looking which of the cones
are finally entered successively. For this purpose, define $\mathbf{i}_k:=j_{i,l}$
if $\tau(C(X_{\e{k-1}}))=i$ and $X_{\e{k}}\in \partial
C_{j_{i,l}}(X_{\e{k-1}})$. If we set additionally $\mathbf{i}_0:=C(X_{\e{0}})$,
then the sequence $(\mathbf{i}_k)_{k\in\N_0}$ tracks the random walk's way to infinity. 
\par
At this point we recall the relation between $\W_k$ and $X_{\e{k}}$: if
$X_{\e{0}}=\mathbf{W}_0=w_0a_0b_0$ and $\W_{1}=w_1a_1b_1$ then $X_{\e{1}}=w_0w_1a_1b_1$; in
general, if $X_{\e{k-1}}=wa_{k-1}b_{k-1}$ and $\W_{k}=w_ka_kb_k$ then
$X_{\e{k}}=ww_ka_kb_k$. That is, there is a natural bijection of trajectories
of $(\W_k)_{k\in\N_0}$ and $(X_{\e{k}})_{k\in\N_0}$.
In particular, the values of the
$\W_k$'s determine the values of the $\mathbf{i}_k$'s uniquely, since the last two
letters of $\W_{k-1}$ describe $\tau(C(X_{\e{k-1}}))$ and $\W_k$ describes
$\tau(C(X_{\e{k}}))$ and the corresponding number in the enumeration of subcones. For a better
visualization of the values of $\mathbf{i}_k$, see Figure \ref{fig:subcones-numbering}.
\begin{figure}[htp]
\begin{center}
\includegraphics[width=7cm]{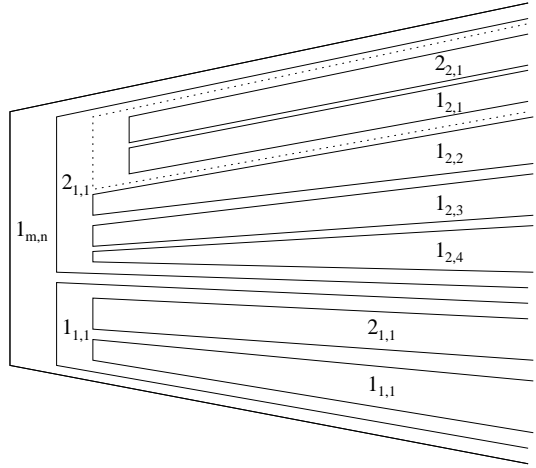}
\end{center}
\caption{Numbering of subcones: the cones with the solid boundary belong to the
covering while the cone with the dotted line does not.}
\label{fig:subcones-numbering}
\end{figure}

In other words, the random variables $\mathbf{i}_k$ collect the information of
the different cones which are
entered successively by the random walk $(X_n)_{n\in\N_0}$ on its way to
infinity, while the $\W_k$'s keep, in addition, the information where the single subcones are
finally entered.
\par
Define 
$$
\mathcal{W}:=\left\lbrace 
(j_{m,n},x) \Biggl| 
\begin{array}{c}
x\in\mathcal{W}_0, \exists w_0\in\calL: \Prob[\W_0=w_0,\W_1=x]>0, \\
\tau(C([w_0]))=m, \tau(C([x]))=j, 1\leq n\leq n(m,j)\\ 
\textrm{ with } x\in \partial C_{j_{m,n}}([w_0])
  \\
\end{array}
\right\rbrace.
$$
In other words, $(j_{m,n},x)\in \mathcal{W}$ if
$x\in\mathcal{W}_0$ with $\tau(C(x))=j$ and if there is
$w_0a_0b_0\in\calL$ such that $\tau(C(a_0b_0))=m$, $\Prob[X_{\e{0}}=w_0a_0b_0,X_{\e{1}}=w_0x]>0$
and $C(x)$ being the $n$-th subcone of type $j$ in the covering of $C(a_0b_0)$.
\begin{Prop}\label{prop:i-W-markov}
The process $\bigl((\mathbf{i}_k,\mathbf{W}_k) \bigr)_{k\in\N}$ is a positive
recurrent, aperiodic Markov chain on the state space $\mathcal{W}$. Moreover, for
$(i_{m,n},w_{1}),(j_{s,t},w_2)\in\mathcal{W}$, the transition probabilities are given by
\begin{equation}\label{equ:Q-def}
\Prob\Bigl[ (\mathbf{i}_k,\mathbf{W}_k)=(j_{s,t},w_2)\Bigl|
(\mathbf{i}_{k-1},\mathbf{W}_{k-1})=(i_{m,n},w_{1})\Bigr]=
\begin{cases}
q(w_{1},w_2), & \textrm{if } s=i,\\
0, & \textrm{if } s\neq i.
\end{cases}
\end{equation}
\end{Prop}
\begin{proof}
Since the values of the $\mathbf{i}_k$'s are uniquely
determined by the values of the $\mathbf{W}_k$'s and since the process $(\W_k)_{k\in\N}$ is
a Markov chain, we also have that $\bigl((\mathbf{i}_k,\mathbf{W}_k)
\bigr)_{k\in\N}$ is Markovian with the proposed transition probabilities.
\par
First, we show that $\mathrm{supp}(\Prob[(\mathbf{i}_k,\mathbf{W}_k)=\cdot])=\mathcal{W}$ for $k\geq
1$. For bthis purpose, let be \mbox{$(j_{i,n},x)\in\mathrm{supp}(\Prob[(\mathbf{i}_k,\mathbf{W}_k)=\cdot])$.} Then there is some
$w_{k-1}a_{k-1}b_{k-1}\in \calL$ with 
\begin{eqnarray*}
&&\Prob[X_{\e{k-1}}=w_{k-1}a_{k-1}b_{k-1},\W_k=x]\\
&=&\sum_{w'\in\calL\setminus C(w_{k-1}a_{k-1}b_{k-1})} G(o,w')p(w',w_{k-1}a_{k-1}b_{k-1})\mathds{L}(a_{k-1}b_{k-1},x)\xi([x])>0,
\end{eqnarray*}
$\tau(C(a_{k-1}b_{k-1}))=i$ and $C(x)$ being the $n$-th subcone of type $j$ in the
covering of the cone $C(a_{k-1}b_{k-1})$. If $k=1$ then $(j_{i,n},x)\in\mathcal{W}$.
In the case $k>1$ take any $w_0a_0b_0\in\calL$ with
$\Prob[\W_0=w_0a_0b_0]>0$ and $\tau(C(w_0a_0b_0))=i$. Since $a_{k-1}b_{k-1}\in
C(a_0b_0)$ we also have $\mathds{L}(a_{0}b_{0},x)>0$ since
$\mathds{L}(a_{k-1}b_{k-1},x)>0$ (recall the remark before Lemma \ref{lem:supp-W0}). Then:
$$
\Prob[\W_0=w_0a_0b_0,\W_1=x]= \sum_{w'\in\calL\setminus C(w_{0}a_{0}b_{0})} G(o,w')p(w',w_{0}a_{0}b_{0})\mathds{L}(a_{0}b_{0},x)\xi([x])>0,
$$
yielding $(j_{i,n},x)\in\mathcal{W}$.
\par
For the other inclusion, let be $(j_{i,n},x)\in\mathcal{W}$. Then there is some
$w_{0}a_{0}b_{0}\in \calL$ with 
$$
\Prob[\W_0=w_{0}a_{0}b_{0},\W_1=x]=\sum_{w'\in\calL\setminus C(w_{0}a_{0}b_{0})} G(o,w')p(w',w_{0}a_{0}b_{0})\mathds{L}(a_{0}b_{0},x)\xi([x])>0,
$$
$\tau(C(a_{0}b_{0}))=i$ and $C(x)$ being the $n$-th subcone of type $j$ in the
covering of $C(a_{0}b_{0})$. If $k=1$ then $(j_{i,n},x)\in
\mathrm{supp}(\Prob[(\mathbf{i}_1,\mathbf{W}_1)=\cdot])$. In the case $k>1$ take any
$w_{k-2}a_{k-2}b_{k-2}\in\calL$ with
$\Prob[X_{\e{k-2}}=w_{k-2}a_{k-2}b_{k-2}]>0$. Then $C(w_{k-2}a_{k-2}b_{k-2})$
  has in its covering a subcone $C(w_{k-1}a_{k-1}b_{k-1})$ of type $i$. Since
  $a_{k-1}b_{k-1}\in C(a_0b_0)$ we have $\mathds{L}(a_{k-1}b_{k-1},x)>0$ due to
  $\mathds{L}(a_{0}b_{0},x)>0$ (once again recall the remark before Lemma \ref{lem:supp-W0})
  and
  $C(x)$ is the $n$-th subcone of type $j$ in the covering of
  $C(a_{k-1}b_{k-1})=C(a_0b_0)$. Hence,
\begin{eqnarray*}
&&\Prob[(\mathbf{i}_k,\W_k)=(j_{i,n},x)]\geq \Prob[X_{\e{k-1}}=w_{k-1}a_{k-1}b_{k-1},X_{\e{k}}=w_{k-1}x]\\
&\geq & \sum_{w'\in\calL\setminus C(w_{k-1}a_{k-1}b_{k-1})} G(o,w')p(w',w_{k-1}a_{k-1}b_{k-1})\mathds{L}(a_{k-1}b_{k-1},x)\xi([x])>0,
\end{eqnarray*}
yielding $\mathcal{W}\subseteq \mathrm{supp}(\Prob[(\mathbf{i}_k,\mathbf{W}_k)=\cdot])$, and therefore
$\mathcal{W}= \mathrm{supp}(\Prob[(\mathbf{i}_k,\mathbf{W}_k)=\cdot])$.
\par
The next task is to show irreducibility, which implies positive recurrence due
to finiteness of $\mathcal{W}$. Let
be $(i_{m,n},w_{1})$, $(j_{s,t},w_2)\in\mathcal{W}$. Take any $\bar
w\in\mathcal{W}_0$ such that $q(w_1,\bar w)>0$ and $\tau(C(\bar w))=s$, which
exists by construction of coverings. Then $w_2\in\partial C_{j_{s,t}}([\bar
w])$, that is, $C(w_2)$ is the $t$-th subcone of type $j$ in the covering of
$C([\bar w])$, yielding $q(\bar w,w_2)>0$. Hence,
\begin{eqnarray}
&&\Prob[(\mathbf{i}_3,\W_3)=(j_{s,t},w_2)\mid (\mathbf{i}_1,\W_1)=(i_{m,n},w_{1})]\label{equ:prop-5.4}\\
&\geq&  \Prob[\W_3=w_2, \W_2=\bar w \mid
(\mathbf{i}_1,\W_1)=(i_{m,n},w_{1})]\nonumber\\
&=& q(w_1,\bar w) \cdot q(\bar w, w_2)>0.\nonumber
\end{eqnarray}
Here, we used the fact that $\mathbf{i}_3=j_{s,t}$ is uniquely
determined by $w_1,\bar w,w_2$ and that this probability does not depend on $m$
and $n$. This yields irreducbility of the process
$\bigl((\mathbf{i}_k,\mathbf{W}_k) \bigr)_{k\in\N}$.
\par 
It follows that the period of the process is at most $2$. In order to see
aperiodicity, take any $\underline{w_1},\underline{w_2}\in\mathcal{W}$ with
$\Prob[(\mathbf{i}_2,\W_2)=\underline{w_2}\mid
(\mathbf{i}_1,\W_1)=\underline{w_1}]>0$. Then we get analogously to (\ref{equ:prop-5.4}):
\begin{eqnarray*}
&&\Prob[(\mathbf{i}_4,\W_4)=\underline{w_1},(\mathbf{i}_2,\W_2)=\underline{w_2}\mid
(\mathbf{i}_1,\W_1)=\underline{w_1}]\\
&=&\Prob[(\mathbf{i}_2,\W_2)=\underline{w_2}\mid
(\mathbf{i}_1,\W_1)=\underline{w_1}]\cdot \Prob[(\mathbf{i}_4,\W_4)=\underline{w_1}\mid
(\mathbf{i}_2,\W_2)=\underline{w_2}]>0.
\end{eqnarray*}
That is, the period of the process is $1$. This finishes the proof.
\end{proof}
Let us recall that  the values of the $\mathbf{i}_k$'s are uniquely
determined by the values of the $\mathbf{W}_k$'s; however, we will explicitely
keep  the values of the $\mathbf{i}_k$'s in the notation of the process for
sake of convenience. 
%
%
%
Observe that the process $(\mathbf{i}_k)_{k\in\mathbb{N}}$ is, in
general, not Markovian. This relies on the fact that $(\mathbf{i}_k)_{k\in\mathbb{N}}$ can be seen as a function of the
process $(\W_k)_{k\in\N}$: the values of the $\W_k$'s determine the values of
the  $\mathbf{i}_k$'s but not vice versa. 
\par
Define the following projection for $(i_{k,l},w_1),(j_{m,n},w_2)\in\mathcal{W}$:
\begin{equation}\label{equ:pi-def}
\pi\bigl((i_{k,l},w_1),(j_{m,n},w_2)\bigr) 
:=
\begin{cases}
(i,j_{i,n})=: (i,j_{n}),& \textrm{if } m=i,\\
(i,j_{i,1})= (i,j_{1}),& \textrm{if } m\neq i.
\end{cases}
\end{equation}
Here, $j_l$ represents the $l$-th subcone of type $j$ in the covering of a cone of type $i$,
namely the cone represented by $j_{i,l}$. We now define the \textit{hidden
Markov chain} $(\Y_k)_{k\in\N}$  by 
$$
\Y_k:=\pi\bigl((\mathbf{i}_{k},\W_{k}),(\mathbf{i}_{k+1},\W_{k+1})\bigr).
$$
In other words, $(\Y_k)_{k\in\mathbb{N}}$ traces once again the random walk's way to infinity in terms of
which subcones are entered successively \textit{without} distinguishing which of the
cone boundary points are the last entry time points $X_{\e{k}}$. At this point
let us mention that the second branch in the definition of $\pi(\cdot,\cdot)$
is not used for defining $\Y_k$, but it will be of interest in Section
\ref{sec:entropy-analyticity}. Furthermore, observe that $X_{\e{0}}$ and $(\Y_k)_{k\in\N}$
allow to reconstruct $(\mathbf{i}_k)_{k\in\N}$.
\par 
Define
$$
\mathcal{W}_\pi:=\bigl\lbrace 
(s,t_n) \bigl| s,t\in\mathcal{I}, 1\leq n\leq n(s,t)\bigr\rbrace.
$$
That is, $t_n$ corresponds
to the $n$-th subcone of type $t$ in the covering of a cone of type $s$.
\begin{Lemma}
For all $k\geq 1$, $\mathrm{supp}(\Prob[\Y_k=\cdot])=\mathcal{W}_\pi$.
\end{Lemma}
\begin{proof}
The inclusion $\mathrm{supp}(\Prob[\Y_k=\cdot])\subset \mathcal{W}_\pi$ is
obvious by definition of $\Y_k$ and $\mathcal{W}_\pi$. Now we show the other inclusion.
Let be $(s,t_n)\in\mathcal{W}_\pi$. Take any $w_{k-1}a_{k-1}b_{k-1}\in\mathcal{W}_0$
with \mbox{$\Prob[\W_{k-1}=w_{k-1}a_{k-1}b_{k-1}]>0$.} Then there exists 
$w_ka_kb_k\in \mathcal{W}_0$ with $\tau(C(a_kb_k))=s$ and $q(w_{k-1}a_{k-1}b_{k-1},w_ka_kb_k)>0$
 due to the construction of coverings. Moreover, there is
$w_{k+1}a_{k+1}b_{k+1}\in \mathcal{W}_0$ with
$q(w_{k}a_{k}b_{k},w_{k+1}a_{k+1}b_{k+1})>0$ such that
$C(w_{k+1}a_{k+1}b_{k+1})$ is the $n$-th cone of type $t$ in $C(a_kb_k)$. Thus,
\begin{eqnarray*}
&&\Prob[\Y_k=(s,t_n)] \\
&\geq & \Prob[\W_{k-1}=w_{k-1}a_{k-1}b_{k-1},
  \W_k=w_ka_kb_k,\W_{k+1}=w_{k+1}a_{k+1}b_{k+1}]\\
&=&  \Prob[\W_{k-1}=w_{k-1}a_{k-1}b_{k-1}]\cdot
q(w_{k-1}a_{k-1}b_{k-1},w_ka_kb_k)\cdot q(w_ka_kb_k,w_{k+1}a_{k+1}b_{k+1})>0,
\end{eqnarray*}
yielding $(s,t_n)\in\mathrm{supp}(\Prob[\Y_k=\cdot])$.
\end{proof}
%
%
%
Since the process $(\mathbf{i}_k,\mathbf{W}_k)_{k\in\mathbb{N}}$ is positive recurrent, it
has an invariant probability \mbox{measure $\nu$.} Let
$(\mathbf{i}_k^{(\nu)},\mathbf{W}_k^{(\nu)})_{k\in\mathbb{N}}$ be a Markov
chain with transition probabilities given by (\ref{equ:Q-def}) but with initial
distribution $\nu$. The corresponding hidden Markov chain $(\Y_k^{(\nu)})_{k\in\N}$ is given by
$$
\Y_k^{(\nu)}:=\pi\bigl((\mathbf{i}_{k}^{(\nu)},\W_{k}^{(\nu)}),(\mathbf{i}_{k+1}^{(\nu)},\W_{k+1}^{(\nu)})\bigr).
$$
In the next section we will link the hidden Markov chains $(\Y_k)_{k\in\N}$ and $(\Y_k^{(\nu)})_{k\in\N}$.

\subsection{Entropy of the Hidden Markov Chain related to the Last Entry Time Process}

In this subsection we derive existence of the asymptotic entropy of the hidden
Markov chains $(\Y_k^{(\nu)})_{k\in\N}$ and $(\Y_k)_{k\in\N}$.
\par
First, consider the hidden Markov chain $(\Y_k^{(\nu)})_{k\in\N}$: this process is
stationary and ergodic since the underlying Markov chain
$\bigl(\mathbf{i}_k^{(\nu)},\mathbf{W}_k^{(\nu)}\bigr)_{k\in\mathbb{N}}$ is
stationary, positive recurrent and aperiodic.
%
Hence, there is a
constant $H(\Y)\geq 0$ such that 
\begin{equation}\label{equ:H(ZZ)}
\lim_{k\to\infty} -\frac{1}{k}\log
\Prob[\Y_1^{(\nu)}=\underline{y}_1,\dots,\Y_k^{(\nu)}=\underline{y}_k]=H(\Y)
\end{equation}
for almost every realisation $(\underline{y}_1,\underline{y}_2,\dots )\in\mathcal{W}_\pi^{\mathbb{N}}$ of
$(\Y_k^{(\nu)})_{k\in\N}$; see e.g. Cover and Thomas \cite[Theorem
16.8.1]{cover-thomas}. The number $H(\Y)$ is called the \textit{asymptotic
  entropy of the (positive recurrent) process} $(\Y_k^{(\nu)})_{k\in\N}$.
We now deduce an analogous statement for the process $(\Y_k)_{k\in\N}$.
\begin{Prop}\label{prop:Y-equal-Y-hat}
For almost every realisation $(\underline{y}_1,\underline{y}_2,\dots )\in\mathcal{W}_\pi^{\mathbb{N}}$ of $(\Y_k)_{k\in\N}$,
$$
\lim_{k\to\infty} -\frac{1}{k}\log \Prob\bigl[\Y_1=\underline{y}_1,\dots,\Y_k=\underline{y}_k\bigr]=H(\Y).
$$
\end{Prop}
\begin{proof}
The processes $(\Y_k^{(\nu)})_{k\in\N}$ and $(\Y_k)_{k\in\N}$ differ only by
the inital distributions of $(\mathbf{i}_1^{(\nu)}, \W_1^{(\nu)})$ and $(\mathbf{i}_1, \W_1)$. 
Moreover, there are constants $c,C>0$ such that
$$
c\cdot \mathbb{P}[(\mathbf{i}_1,\mathbf{W}_1)=(i_{m,n},x)] \leq   \nu(i_{m,n},x) \leq C\cdot \mathbb{P}[(\mathbf{i}_1,\mathbf{W}_1)=(i_{m,n},x)]
$$
for all $(i_{m,n},x)\in\mathcal{W}$. Denote
by $\mu_1$ the distribution of $(\mathbf{i}_1,\W_1)$. We now get for almost every trajectory
$(\underline{y}_1,\underline{y}_2,\dots )\in\mathcal{W}_\pi^{\mathbb{N}}$ of $(\Y_k)_{k\in\N}$:
\begin{eqnarray*}
&&H(\Y) = \lim_{k\to\infty} -\frac{1}{k} \log  \Prob\bigl[
\Y_1^{(\nu)}=\underline{y}_1,\dots,\Y_k^{(\nu)}=\underline{y}_k\bigr]\\
&=& \lim_{k\to\infty} -\frac{1}{k} \log
\sum_{\substack{\underline{w}_1,\dots,\underline{w}_{k+1}\in\mathcal{W}:\\
    \pi(\underline{w}_j,\underline{w}_{j+1})=\underline{y}_j\\
\textrm{for } 1\leq j\leq k}}
\nu(\underline{w}_1)
\Prob[(\mathbf{i}_l,\W_l)=\underline{w}_l\textrm{ for
} 2\leq l\leq k+1\mid (\mathbf{i}_1,\W_1)=\underline{w}_1]\\
&=& \lim_{k\to\infty} -\frac{1}{k} \log
\sum_{\substack{\underline{w}_1,\dots,\underline{w}_{k+1}\in\mathcal{W}:\\
    \pi(\underline{w}_j,\underline{w}_{j+1})=\underline{y}_j\\
\textrm{for } 1\leq j\leq k}}
\mu_1(\underline{w}_1)  \Prob[(\mathbf{i}_l,\W_l)=\underline{w}_l \textrm{ for
} 2\leq l\leq k+1 \mid (\mathbf{i}_1,\W_1)=\underline{w}_1]
\end{eqnarray*}
\begin{eqnarray*}
&=& \lim_{k\to\infty} -\frac{1}{k} \log
\sum_{\substack{\underline{w}_1,\dots,\underline{w}_{k+1}\in\mathcal{W}:\\
    \pi(\underline{w}_j,\underline{w}_{j+1})=\underline{y}_j\\
\textrm{for } 1\leq j\leq k}}
\Prob\bigl[(\mathbf{i}_1,\mathbf{W}_1)=\underline{w}_1,\dots,(\mathbf{i}_{k+1},\mathbf{W}_{k+1})=\underline{w}_{k+1}\bigr] \\
&=&\lim_{k\to\infty} -\frac{1}{k} \log \Prob\bigl[
\Y_1=\underline{y}_1,\dots,\Y_k=\underline{y}_k\bigr].
\end{eqnarray*}
\end{proof}
As a consequence we obtain the next statement:
\begin{Cor}
$$
\lim_{k\to\infty} -\frac{1}{k}\int
\log\Prob\bigl[\Y_1=\underline{y}_1,\dots,\Y_k=\underline{y}_k\bigr]\,d\Prob(\underline{y}_1,\underline{y}_2,\dots) =H(\Y).
$$
\end{Cor}
\begin{proof}
Since $|\mathcal{W}|<\infty$ by definition, there is $\varepsilon_0>0$ such that,
for all $\underline{w}_1,\underline{w}_2\in\mathcal{W}$,
$$
\Prob[(\mathbf{i}_2,\W_2)=\underline{w}_2\mid
(\mathbf{i}_1,\W_1)=\underline{w}_1]>0 \ \textrm{ implies } \ 1\geq  \Prob[(\mathbf{i}_2,\W_2)=\underline{w}_2\mid (\mathbf{i}_1,\W_1)=\underline{w}_1] \geq \varepsilon_0. 
$$
If $(\underline{y}_1,\dots,\underline{y}_k)\in\mathcal{W}_\pi^k$ with
$\Prob[\Y_1=\underline{y}_1,\dots,\Y_k=\underline{y}_k]>0$ then there are
\mbox{$\underline{w}_1,\dots,\underline{w}_{k+1}\in\mathcal{W}$} with
$\pi(\underline{w}_j,\underline{w}_{j+1})=\underline{y}_j$ for $1\leq j\leq
k$ and $\Prob\bigl[(\mathbf{i}_1,\W_1)=\underline{w}_1,\dots,(\mathbf{i}_{k+1},\W_{k+1})=\underline{w}_{k+1}\bigr]>0$. Therefore,
\begin{eqnarray*}
0& \leq & -\frac{1}{k}
\log\Prob\bigl[\Y_1=\underline{y}_1,\dots,\Y_k=\underline{y}_k\bigr]\\
&\leq & 
-\frac{1}{k}
\log\Prob\bigl[(\mathbf{i}_1,\W_1)=\underline{w}_1,\dots,(\mathbf{i}_{k+1},\W_{k+1})=\underline{w}_{k+1}\bigr]\\
&\leq & -\frac{1}{k}\log(c\cdot \varepsilon_0^k)=
-\frac{1}{k}\log
c-\log\varepsilon_0\leq -\log c -\log\varepsilon_0,
\end{eqnarray*}
where $c=\min_{\underline{w}\in\mathcal{W}} \mathbb{P}[(\mathbf{i}_1,\mathbf{W}_1)=\underline{w}]$.
Therefore, we may exchange integral and limit, which yields the claim together with
Proposition \ref{prop:Y-equal-Y-hat}.
\end{proof}
Let be $w\in\calL$ with $|w|\geq 2$. Define
$$
\hat l(w) := -\log \sum_{w'\in \partial C(w)} L(o,w'|1).
$$
We obtain the following law of large numbers:
\begin{Prop}\label{prop:hat-l-convergence}
$$
\lim_{k\to\infty} \frac{\hat l(X_{\e{k}})}{k}=H(\Y) \quad \textrm{almost surely}.
$$
\end{Prop}
\begin{proof}
Let be $k\in\N$ and assume for the moment that $\W_l=y_la_lb_l$, where $y_l\in\calA^\ast\setminus\{o\}$ and $a_lb_l\in\calA^2$
for $0\leq l\leq k$. That is, $X_{\e{l}}=y_0y_1\dots y_la_lb_l$. 
Furthermore, assume that $\Y_1=(j,t^{(1)})$, where $j=\tau(C(a_1b_1))$, and $\Y_l=(s^{(l)},t^{(l)})$
for $2\leq l\leq k$, where the values of $s^{(2)},\dots,s^{(k-1)}$ and $t^{(1)},\dots,t^{(k-1)}$ are determined
by the values of $\W_l=y_la_lb_l$. Vice versa, given $X_{\e{1}}$ the values of
$s^{(2)},\dots,s^{(k-1)}$ and $t^{(1)},\dots,t^{(k-1)}$ determine uniquely the
cones $C(y_la_lb_l)$: indeed, $X_{\e{1}}$ and $t^{(1)}$ determine uniquely
$C(X_{\e{2}})$ and therefore also $C(\W_2)=C(y_2a_2b_2)$; inductively, given
$C(X_{\e{l}})$ of type $s^{(l)}$ then $t^{(l)}$ determines uniquely
$C(X_{\e{l+1}})$ and $C(\W_{l+1})=C(y_{l+1}a_{l+1}b_{l+1})$. We mark it by
$(\ast)$ when we make use of this ``transition''.
\par
Recall that the covering of $\calL$ consists of  $n_0$
subcones $C_i^{(0)}$, $1\leq i\leq n_0$. Each $C_i^{(0)}$ has again a covering
consisting of $n(\tau(C_i^{(0)}),j)$ subcones of type $j$. We enumerate all these
subcones of type $j$ by $C_{j,k}^{(1)}$ with $1\leq k\leq N_j:=\sum_{i=1}^{n_0}
n(\tau(C_i^{(0)}),j)$, that is, we enumerate all subcones of type $j$ which
appear in the coverings of all $C_i^{(0)}$, $1\leq i\leq n_0$.
%
%
\par
Since $\mathcal{W}_0$ is finite, there is some constant $c>0$ such that
$$
c\cdot \mathbb{P}[X_{\e{1}}=x] \leq \mathbb{P}[X_{\e{1}}=y] 
$$
for all $x,y\in\bigcup_{k=1}^{N_j}\partial C_{j,k}^{(1)}\subseteq \mathrm{supp}(\Prob[X_{\e{1}}=\cdot])$.
\par
\underline{Claim:} for almost every realisation  $(x_1,\underline{y}_1,\underline{y}_2,\dots)$ of $(X_{\e{1}},\Y_1,\Y_2,\dots)$,
\begin{equation}\label{equ:H-equation-e1}
H(\Y)=\lim_{k\to\infty} -\frac{1}{k} \log \Prob\bigl[C(X_{\e{1}})=C(x_1),\Y_1=\underline{y}_1,\dots,\Y_k=\underline{y}_k \bigr].
\end{equation}
First, we have for $k\geq 2$:
\begin{eqnarray*}
&& N_j\cdot \Prob\bigl[X_{\e{1}}\in C(y_0y_1a_1b_1),\Y_1=(j,t^{(1)}),\Y_2=(s^{(2)},t^{(2)}),\dots,\Y_{k-1}=(s^{(k-1)},t^{(k-1)}) \bigr]
\nonumber\\
&\overset{(\ast)}{=}& N_j \cdot \sum_{x\in \partial C(y_0y_1a_1b_1)}
\sum_{\substack{w_2,\dots,w_k\in\mathcal{W}_0:\\ w_i\in \partial
    C(y_ia_ib_i)\\ \textrm{for all } 2\leq i\leq k}}
\Prob[X_{\e{1}}=x,X_{\e{2}}=y_0y_1w_2,\dots,X_{\e{k}}=y_0\dots y_{k-1}w_k]\\
&= & N_j \cdot \sum_{\substack{x\in \partial
    C(y_0y_1a_1b_1);\\ w_2,\dots,w_k\in\mathcal{W}_0:\\ w_i\in \partial C(y_ia_ib_i)\\ \textrm{for all } 2\leq i\leq k}}
\Prob[X_{\e{1}}=x] \Prob[X_{\e{2}}=y_0y_1w_2,\dots,X_{\e{k}}=y_0\dots
y_{k-1}w_k\mid X_{\e{1}}=x]\\
&=& N_j \cdot \sum_{x\in \partial C(y_0y_1a_1b_1)}
\sum_{\substack{w_2,\dots,w_k\in\mathcal{W}_0:\\ w_i\in \partial C(y_ia_ib_i)\\ \textrm{for all } 2\leq i\leq k}}
\Prob[X_{\e{1}}=x] q(y_1[x],w_2) \prod_{i=3}^k q(w_{i-1},w_i)\\
&=& \sum_{l=1}^{N_j}\sum_{x\in \partial C(y_0y_1a_1b_1)}
\sum_{\substack{w_2,\dots,w_k\in\mathcal{W}_0:\\ w_i\in \partial C(y_ia_ib_i)\\ \textrm{for all } 2\leq i\leq k}}
\Prob[X_{\e{1}}=x] \Prob\bigl[\W_2=w_2\bigl| [X_{\e{1}}]=[x]\bigr] \prod_{i=3}^k
q(w_{i-1},w_i).
\end{eqnarray*}
For a moment, let be $\partial
C(y_0y_1a_1b_1)=\{y_0y_1c_1d_1,\dots,y_0y_1c_\kappa d_\kappa\}$. Then for all
$l\in\{1,\dots,N_j\}$ there is some $v_l\in\calA^\ast$ such that
$\partial C_{j,l}^{(1)}=\{v_lc_1d_1,\dots,v_lc_\kappa d_\kappa\}$. Therefore, for
every $x\in\partial C(y_0y_1a_1b_1)$ and each $l\in\{1,\dots,N_j\}$ there is
exactly one $\hat x_l\in\partial C_{j,l}^{(1)}$ with $[\hat x_l]=[x]$, 
$\Prob[X_{\e{1}}=x]\geq c\cdot \Prob[X_{\e{1}}=\hat x_l]$ and
$\Prob\bigl[\W_2=w_2\bigl| [X_{\e{1}}]=[x]\bigr]=\Prob\bigl[\W_2=w_2\bigl|
[X_{\e{1}}]=[\hat x_l]\bigr]$ for all $w_2\in\mathcal{W}_0$. The last
equation follows from the fact that the probabilities depend on $X_{\e{1}}$ only by its last two letters $[X_{\e{1}}]$ in the condition.
We write $\hat
x_l$ for this mapping $(x,l)\mapsto \hat x_l$. Hence,
\begin{eqnarray}
&&N_j\cdot \Prob\left[
\begin{array}{c}
X_{\e{1}}\in C(y_0y_1a_1b_1),\Y_1=(j,t^{(1)}),\\
\Y_2=(s^{(2)},t^{(2)}),\dots,\Y_{k-1}=(s^{(k-1)},t^{(k-1)})
\end{array} \right]\nonumber\\
&\geq & \sum_{l=1}^{N_j} \sum_{x\in \partial C(y_0y_1a_1b_1)}
\sum_{\substack{w_2,\dots,w_k\in\mathcal{W}_0:\\ w_i\in\partial C(y_ia_ib_i)\\ \textrm{for all } 2\leq i\leq k}} 
c \Prob[X_{\e{1}}=\hat x_l]   \Prob\bigl[\W_2=w_2\mid [X_{\e{1}}]=[\hat x_1]\bigr] 
\prod_{i=3}^k q(w_{i-1},w_i)\nonumber\\
&= & \sum_{l=1}^{N_j} \sum_{w\in\partial C^{(1)}_{j,l}}
\sum_{\substack{w_2,\dots,w_k\in\mathcal{W}_0:\\ w_i\in\partial C(y_ia_ib_i)\\ \textrm{for all } 2\leq i\leq k}} 
c\cdot \Prob[X_{\e{1}}= w]  \cdot \Prob\bigl[\W_2=w_2\bigl| [X_{\e{1}}]= [w]\bigr] \cdot
\prod_{i=3}^k q(w_{i-1},w_i)\nonumber\\
&=& c \cdot
\Prob\bigl[\Y_1=(j,t^{(1)}),\Y_2=(s^{(2)},t^{(2)}),\dots,\Y_{k-1}=(s^{(k-1)},t^{(k-1)})
\bigr].\nonumber 
\end{eqnarray}
Vice versa, we obviously have
\begin{eqnarray}
&& \Prob\bigl[X_{\e{1}}\in C(y_0y_1a_1b_1),\Y_1=(j,t^{(1)}),\Y_2=(s^{(2)},t^{(2)}),\dots,\Y_{k-1}=(s^{(k-1)},t^{(k-1)}) \bigr]\nonumber\\
&\leq &
\Prob\bigl[\Y_1=(j,t^{(1)}),\Y_2=(s^{(2)},t^{(2)}),\dots,\Y_{k-1}=(s^{(k-1)},t^{(k-1)})
\bigr].\nonumber
\end{eqnarray}
This proves the claim (\ref{equ:H-equation-e1}).
\par
Recall from Equation (\ref{equ:G-L}) that $G(o,w|1)=G(o,o|1)L(o,w|1)$ for all $w\in\calL$ and that
$\xi(\cdot)$ can only take finitely many (non-zero) values. 
We now can conclude as follows:
\begin{eqnarray*}
&&\lim_{k\to\infty} \frac{\hat l(X_{\e{k}})}{k}= \lim_{k\to\infty}
-\frac{1}{k}\log \sum_{w'\in \partial C(y_0y_1\dots y_ka_kb_k)} L(o,w'|1) \\
&=&
\lim_{k\to\infty} -\frac{1}{k} \log \sum_{bc\in\calA^2: bc\in\partial
  C(a_kb_k)} L(o,y_0y_1\dots y_{k}bc|1)\\
&=& \lim_{k\to\infty} -\frac{1}{k} \log \biggl[ \sum_{w_1\in \partial C(y_0y_1a_1b_1)}\sum_{\substack{w_2,\dots,w_k\in\mathcal{W}_0: \\
    w_i\in\partial C(y_ia_ib_i)\\ \textrm{for all } 2\leq i\leq k }} \sum_{\substack{w'\in\calL: \\ w'\notin C(w_1)}} L(o,w'|1)
p(w',w_1) \prod_{i=2}^{k} \mathds{L}([w_{i-1}],w_i)\biggr]\\
&=& \lim_{k\to\infty} -\frac{1}{k} \log \biggl[ \sum_{\substack{w_1\in\partial C(y_0y_1a_1b_1);\\
w_2,\dots,w_k\in\mathcal{W}_0: \\
    w_i\in\partial C(y_ia_ib_i)\\ \textrm{for all } 2\leq i\leq k;\\
w'\in\calL \setminus C(w_1)}} 
G(o,w'|1)p(w',w_1)\xi([w_1])  \cdot \prod_{i=2}^k
\frac{\xi([w_i])}{\xi([w_{i-1}])}\mathds{L}([w_{i-1}],w_i)\biggr]
\end{eqnarray*}
\begin{eqnarray*}
&=&\lim_{k\to\infty} -\frac{1}{k} \log \biggl[ \sum_{w_1\in\partial C(y_0y_1a_1b_1)} \sum_{\substack{w_2,\dots,w_k\in\mathcal{W}_0: \\
    w_i\in\partial C(y_ia_ib_i)\\ \textrm{for all } 2\leq i\leq k}}
\Prob[X_{\e{1}}=w_1] q(y_1[w_1],w_2) \prod_{i=3}^k q(w_{i-1},w_i)\biggr]\\
&=& \lim_{k\to\infty} -\frac{1}{k} \log  \Prob\biggl[
\begin{array}{c}
X_{\e{1}}\in C(y_0y_1a_1b_1),\Y_1=(j,t^{(1)}),\\
\Y_2=(s^{(2)},t^{(2)}),\dots,\Y_{k-1}=(s^{(k-1)},t^{(k-1)})
\end{array} \biggr]\\
&=& \lim_{k\to\infty} -\frac{1}{k} \log
\Prob\bigl[\Y_1=(j,t^{(1)}),\Y_2=(s^{(2)},t^{(2)}),\dots,\Y_{k-1}=(s^{(k-1)},t^{(k-1)})\bigr] = H(\Y).
\end{eqnarray*}
The last equation follows from (\ref{equ:H-equation-e1}). 
We remark that  the first coordinate
of $\Y_1$ describes only the cone type of $X_{\e{1}}$ but
there may be several distinct cones of the same type $j\in\mathcal{I}$ with $j=\tau(C(X_{\e{1}}))$.
%
%
%
%
%
\end{proof}
Recall the definition of $l(w)= -\log L(o,w|1)$
for $w\in\calL$.
\begin{Cor}\label{cor:l-limit}
$$
\lim_{k\to\infty} \frac{l(X_{\e{k}})}{k}=H(\Y)\quad \textrm{almost surely}.
$$
\end{Cor}
\begin{proof}
It suffices to compare $\hat l(X_{\e{k}})$ with $l(X_{\e{k}})$. Assume
for a moment that $X_{\e{k}}=w_k$ with $w_k\in\calL$ and that $X_{\e{k}}$ is on the
boundary of the cone $C_k$. Then, the probability of walking \textit{inside} $C_k$
from any $w'\in \partial C_k$ to any $w-k\in\partial C_k$ (or vice versa) can
be bounded from below by some constant $\varepsilon_0$, because the
probabilities depend only on $[w_k],[w']\in\calA^2$: that is,
$$
\Prob_{w'}[\exists n\in\N: X_n=w_k, \forall m\leq n: X_n\in C(w')]\geq \varepsilon_0.
$$
Therefore,
\begin{eqnarray*}
L(o,X_{\e{k}}|1) &\leq& \sum_{w'\in\partial C_k} L(o,w'|1)=\hat l(X_{\e{k}}),\\
\hat l(X_{\e{k}})\cdot \varepsilon_0 & \leq & \sum_{w'\in\partial C_k} L(o,w'|1)
\cdot \Prob_{w'}[\exists n\in\N: X_n=w_k, \forall m\leq n: X_n\in C(w')]\\
 &\leq & |\calA^2|\cdot L(o,X_{\e{k}}|1).
\end{eqnarray*}
In the second inequality chain we extended paths from $o$ to $w'$ to paths from $o$ to
$w_k$ via $w'$ such that each such path is counted at most $|\calA^2|$ times.
Taking logarithms, dividing by $k$ and letting $k$ tend to infinity yields the claim.
\end{proof}
Now we come to an important law of large numbers. 
Denote by $\nu_0$ the invariant probabilty measure of
the positive recurrent Markov chain 
$(\mathbf{W}_k)_{k\in\mathbb{N}}$ and define
\begin{equation}\label{equ:lambda}
\lambda:=  \mathbb{E}[|\W_1^{(\nu)}|]-2=\sum_{w\in\mathcal{W}_0} \nu_0(w) \cdot \bigl(|w|-2\bigr).
\end{equation}
Then:
\begin{Prop}\label{prop:L-limit}
$$
\lim_{k\to\infty} \frac{l(X_n)}{n}=\ell\cdot \lambda^{-1} \cdot H(\Y) \quad
\textrm{almost surely}.
$$
\end{Prop}
\begin{proof}
Define
$$
\he{k} := \inf\bigl\lbrace m\in\N \bigl| \forall n\geq m: |X_n|=k\bigr\rbrace.
$$
Observe that
$\he{k}-1=\sup \bigl\lbrace m\in\N \bigl|  |X_m|=k-1\bigr\rbrace$. Transience yields $\he{k}<\infty$ almost surely for all $k\in\N$. By
\cite[Proposition 2.3]{gilch:08}, $k/(\he{k}-1)$ tends to the rate of escape $\ell$ as
$k\to\infty$; hence, $k/\he{k}\to\ell$ as $k\to\infty$.
Define the \textit{maximal last entry times} at time $n\in\N$ as
\begin{eqnarray*}
\mathbf{k}(n) &:= & \max\{ k\in\N \mid \he{k}\leq n\},\\
\mathbf{t}(n) &:= & \max\{k\in\N \mid \e{k}\leq n\}.
\end{eqnarray*}
Obviously, $\mathbf{k}(n)\geq \mathbf{t}(n)$ and each last entry time $\e{k}$
corresponds (depending on the concrete realization) to exactly one $\he{l}$ with
$l\geq k$. First, we rewrite
\begin{equation}\label{equ:l/n}
\frac{l(X_n)}{n} = \frac{ l(X_n)-
  l(X_{\mathbf{e}_{\mathbf{t}(n)}})}{n}+
\frac{l(X_{\mathbf{e}_{\mathbf{t}(n)}})}{\mathbf{t}(n)}\cdot \frac{\mathbf{t}(n)}{\mathbf{k}(n)}\cdot \frac{\mathbf{k}(n)}{\he{\mathbf{k}(n)}}\cdot \frac{\he{\mathbf{k}(n)}}{n}.
\end{equation}
Let $\varepsilon_1$ be the minimal occuring positive single-step transition
probability. Define
$$
D:=\max\left\lbrace
|w_2|-|w_1| \biggl| \begin{array}{c}
\exists ab\in\calA^2: C(ab) \textrm{ has covering } C_1,\dots,C_{n(ab)},\\
w_1\in\partial C(ab),w_2\in \bigcup_{i=1}^{n(ab)} \partial C_i
\end{array}
\right\rbrace<\infty.
$$
Then we have $\hat{\mathbf{e}}_{\mathbf{k}(n)}\geq \e{\mathbf{t}(n)}\geq
\hat{\mathbf{e}}_{\mathbf{k}(n)-D}$ and $n/\e{\mathbf{t}(n)}\geq 1$.
This implies 
\begin{equation}\label{equ:exit-n}
1\leq \frac{n}{\e{\mathbf{t}(n)}}\leq
\frac{\he{\mathbf{k}(n)+1}}{\he{\mathbf{k}(n)-D}}=
\frac{\he{\mathbf{k}(n)+1}}{\mathbf{k}(n)}\frac{\mathbf{k}(n)-D}{\he{\mathbf{k}(n)-D}}\xrightarrow{n\to\infty}
\frac{1}{\ell}\cdot \ell =1 \quad \textrm{a.s.,}
\end{equation}
which in turn yields $(n-\e{\mathbf{t}(n)})/n\to 0$ as $n\to\infty$. Thus, the first quotient on the right hand side of (\ref{equ:l/n}) tends to zero since
\begin{eqnarray*}
L(o,X_n|1) \cdot \varepsilon_1^{n-\e{\mathbf{t}(n)}} &\leq &
L(o,X_{\mathbf{e}_{\mathbf{t}(n)}}|1)  \quad \textrm{(due to weak symmetry)},\\
L(o,X_{\mathbf{e}_{\mathbf{t}(n)}}|1)\cdot \varepsilon_1^{n-\e{\mathbf{t}(n)}} & \leq & L(o,X_n|1).
\end{eqnarray*}
Here we used the fact that one can walk from $X_{\mathbf{e}_{\mathbf{t}(n)}}$
to $X_n$ (or vice versa) in $n-\e{\mathbf{t}(n)}$ steps.
By Corollary \ref{cor:l-limit},
$l(X_{\mathbf{e}_{\mathbf{t}(n)}})/\mathbf{t}(n)$ tends to $H(\Y)$. 
On the other hand side, $\he{k}/k$
tends almost surely to $1/\ell$ and $\he{\mathbf{k}(n)}/n$ tends to $1$ almost
surely since $1\leq n/\he{\mathbf{k}(n)}\leq n/\e{\mathbf{t}(n)}\to 1$ by (\ref{equ:exit-n}).
It remains to investigate the limit
$\lim_{k\to\infty} \mathbf{k}(n)/\mathbf{t}(n)$. Clearly, 
$$
\frac{\mathbf{k}(n)}{\mathbf{t}(n)}=
\frac{|X_{\he{\mathbf{k}(n)}}|}{\mathbf{t}(n)}
=
\frac{1}{\mathbf{t}(n)}\Bigl(|X_{\e{1}}|+\sum_{i=1}^{\mathbf{t}(n)-1}
(|X_{\e{i+1}}|-|X_{\e{i}}|)+(|X_{\he{\mathbf{k}(n)}}|-|X_{\e{\mathbf{t}(n)}}|)\Bigr).
$$
Note that $0\leq |X_{\he{\mathbf{k}(n)}}|-|X_{\e{\mathbf{t}(n)}}|\leq D$ and
$0<|X_{\e{1}}|\leq D_1$ almost surely for some suitable constant $D_1$. Thus, it is sufficient to consider
$$
\frac{1}{k} \sum_{i=1}^k (|X_{\e{i+1}}|-|X_{\e{i}}|)= \frac{1}{k} \sum_{i=1}^k \bigl(|\W_i|-2\bigr).
$$
Since $(\W_k)_{k\in\N}$ is positive recurrent, the ergodic theorem yields almost surely
$$
\lim_{k\to\infty} \frac{1}{k} \sum_{i=1}^k \bigl(|\W_i|-2\bigr)
=\sum_{w\in\mathcal{W}_0} \nu_0(w) \bigl(|w|-2\bigr) =\lambda.
$$
This finishes the proof and gives the proposed formula.
\end{proof}

\section{Existence of Entropy}
\label{sec:entropy}


We now link Proposition \ref{prop:L-limit} with the asymptotic entropy of the
random walk $(X_n)_{n\in\N_0}$. For this purpose, we follow the reasoning of \cite{gilch:11}.
First, we need the following lemma:
\begin{Lemma}\label{lemma:R>1}
There is $R>1$ such that $G(w_1,w_2|R)<\infty$ for all $w_1,w_2\in\calL$.
\end{Lemma}
\begin{proof}
A simple adaption of the proof of \cite[Proposition 8.2]{lalley} shows that,
for $w_1,w_2\in\calL$,
$G(w_1,w_2|z)$ has radius of convergence $R(w_1,w_2)>1$. At this point we also need
the suffix-irreducibility Assumption \ref{ass:suffix}; see Subsection
\ref{subsec:remarks-1} for a comment on how to weaken this assumption. 
Since we assume the random walk $(X_n)_{n\in\N_0}$ to be irreducible, the
radius of convergence is independent from $w_1$ and $w_2$; hence,
$G(w_1,w_2|R)<\infty$ for all $w_1,w_2\in\calL$ and $R=R(w_1,w_2)$.
\end{proof}
Let us remark that we have also $\bar L(ab,cde|R)<\infty$,
$\overline{G}(ab,cd|R)<\infty$ and $L(o,a|R)<\infty$ for all $a,b,c,d,e\in
\calA$, since these generating functions are dominanted by Green functions.
In the following let be $\varrho\in[1,R)$.
\begin{Lemma}\label{lemma:prob-bound}
There are constants $D_1$ and $D_2>0$ such that for all $m,n\in\N_0$
$$
p^{(m)}(o,X_n) \leq D_1\cdot  D_2^n \cdot \varrho^{-m}.
$$
\end{Lemma}
\begin{proof}
Denote by $\mathcal{C}_\varrho$ the circle with radius $\varrho$ in the complex plane centered at
$0$. A straightforward computation shows  for $m\in\mathbb{N}_0$: 
$$
\frac{1}{2\pi i} \oint_{\mathcal{C}_\varrho} z^m \frac{dz}{z} = 
\begin{cases}
1, & \textrm{if } m=0,\\
0, & \textrm{if } m\neq 0.
\end{cases}
$$
Let be $w\in\mathcal{L}$. An application of Fubini's Theorem yields
\begin{eqnarray*}
\frac{1}{2\pi i} \oint_{\mathcal{C}_\varrho} G(o,w|z)\,z^{-m} \frac{dz}{z} & = &
\frac{1}{2\pi i} \oint_{\mathcal{C}_\varrho} \sum_{k\geq 0} p^{(k)}(o,w) z^k\,z^{-m}
\frac{dz}{z}\\
&=&
\frac{1}{2\pi i} \sum_{k\geq 0} p^{(k)}(o,w) \oint_{\mathcal{C}_\varrho} z^{k-m}
\frac{dz}{z} = p^{(m)}(o,w).
\end{eqnarray*}
Since $G(o,w|z)$ is analytic on $\mathcal{C}_\varrho$, we have $|G(o,w|z)|\leq G(o,w|\varrho)$ for all
$|z|=\varrho$. Thus, 
$$
p^{(m)}(o,w) \leq \frac{1}{2\pi}\cdot \varrho^{-m-1}\cdot  G(o,w|\varrho)\cdot  2\pi \varrho = G(o,w|\varrho)\cdot  \varrho^{-m}.
$$
Set $L:=1 \lor \max\bigl\lbrace \bar L\bigl(ab,cde|\varrho\bigr)\mid a,b,c,d,e\in\calA\bigr\rbrace$,
$C_0:=\varrho \cdot G(o,o|\varrho)\cdot \sum_{a\in\calA} L(o,a|\varrho)$ and
$C_1=\max\{\overline{G}(ab,cd|\varrho)\mid ab,cd\in\calA^2\}$. Equation
(\ref{equ:L-expansion}) provides for all $w\in\calL$ with $|w|\geq 2$
$$
G(o,w|\varrho) =G(o,o|\varrho)\cdot L(o,w|\varrho) 
\leq C_0\cdot |\mathcal{A}|^{2(|w|-2)}\cdot L^{|w|-2}\cdot C_1.
$$
Set $C_2:=C_0 \lor \max\{G(o,w|\varrho)|w\in\calL,|w|\leq 2\}$.
Since $|X_n|\leq n$, we obtain the proposed inequality by setting
$D_1:=C_1+ C_2$ and $D_2:=|\mathcal{A}|^2\cdot L$:
$$
p^{(m)}(o,X_n) \leq D_1\cdot  |\mathcal{A}|^{2|X_n|}\cdot L^{|X_n|} \cdot \varrho^{-m}
\leq D_1\cdot  |\mathcal{A}|^{2n}\cdot L^n \cdot \varrho^{-m} = D_1\cdot D_2^n \cdot \varrho^{-m}.
$$
\end{proof}
The following technical lemma will be used in the proof of the next theorem:
\begin{Lemma}\label{cut-lemma}
Let $(A_n)_{n\in\N}$, $(a_n)_{n\in\N}$, $(b_n)_{n\in\N}$ be sequences of
strictly positive numbers with $A_n=a_n+b_n$. Assume that $\lim_{n\to\infty}
-\frac{1}{n}\log A_n=c \in [0,\infty)$ and that $\lim_{n\to\infty} b_n/q^n
= 0$ for all $q\in(0,1)$. Then $\lim_{n\to\infty} -\frac{1}{n}\log a_n=c$. 
\end{Lemma}
\begin{proof}
A proof can be found in \cite[Lemma 3.5]{gilch:11}.
\end{proof}

\begin{Lemma}\label{lemma:sum-bounds}
For $n\in\N$, consider the function $f_n:\calL\to\R$ defined by
$$
f_n(w):=\begin{cases}
-\frac{1}{n}\log \sum_{m=0}^{n^2} p^{(m)}(o,w), & \textrm{if }
p^{(n)}(o,w)>0,\\
0, & \textrm{otherwise.}
\end{cases}
$$
Then there are constants $d$ and $D$ such that $d\leq f_n(w)\leq D$ for all
$n\in\N$ and $w\in\calL$.
\end{Lemma}
\begin{proof}
Let be $w\in\calL$ and $n\in\N$ with $p^{(n)}(o,w)>0$. For $w_1\in\calL$ and
$z>0$, define the first return generating function as
$$
U(w_1,w_1|z) := \sum_{n\geq 1}\Prob\bigl[X_n=w_1,\forall m\in\{1,\dots,n-1\}:
X_m\neq w_1\bigl|X_0=w_1\bigr]\cdot z^n.
$$ 
Recall the number $R>1$ from Lemma \ref{lemma:R>1}. Then
\begin{equation}\label{equ:G-estimate}
G(w,w|1)\leq \frac{1}{1-\frac{1}{R}};
\end{equation}
indeed, since $G(w,w|z)=\bigl(1-U(w,w|z)\bigr)^{-1}$ it must be that $U(w,w|z)<1$
for all $w\in\calL$ and all $z\in[0,R)$; moreover, $U(w,w|0)=0$, $U(w,w|z)$ is continuous, strictly increasing and
strictly convex for $z\in[0,R)$, so we must have $U(w,w|z)\leq 1/R$  for all
$z\in[0,R)$, providing (\ref{equ:G-estimate}).
\par
Define $F(o,w):=\sum_{n\geq 0} f^{(k)}(o,w)$, where $f^{(k)}(o,w)$ is the
probability of starting at $o$ and with the first visit to $w$ at time $k$. By
conditioning on the first visit to $w$ we get $G(o,w|1)=F(o,w)G(w,w|1)$. Therefore,
$$
\sum_{m=0}^{n^2} p^{(m)}(o,w)
\leq G(o,w|1)=
 F(o,w)\cdot G(w,w|1)  \leq 
\frac{1}{1-\frac{1}{R}}, 
$$
that is,
$$
f_n(w) \geq -\frac{1}{n} \log \frac{1}{1-\frac{1}{R}}\geq -\log \frac{1}{1-\frac{1}{R}}=:d.
$$
For the upper bound, observe that $w\in \calL$ with
$p^{(n)}(o,w)>0$ can be reached from $o$ in
$n$ steps with a probability of at least $\varepsilon_0^{n}$,
where 
$$
\varepsilon_0 :=\min\{p(w_1,w_2) \mid w_1,w_2\in\calA^\ast, p(w_1,w_2)>0\} >0
$$  
is independent from
$w$. Thus, the sum $\sum_{m=0}^{n^2} p^{(m)}(o,w)$ has a value
greater or equal to $\varepsilon_0^{n}$. Hence,
$f_n(x)\leq -\log\varepsilon_0=:D$.
\end{proof}
Now we can finally prove:

\begin{proof}[Proof of Theorem \ref{th:existence-entropy}]
%
Recall Equation (\ref{equ:G-L}). We can rewrite $\ell\cdot \lambda^{-1} \cdot H(\Y)$ as
\begin{eqnarray*}
&&\frac{\ell\cdot  H(\Y)}{\lambda} = \int \frac{\ell\cdot  H(\Y)}{\lambda}\,d\Prob 
=
\int \lim_{n\to\infty} -\frac{1}{n} \log
 L\bigl(o,X_n(\omega)\bigr|1\bigr)\,d\Prob(\omega) \\
&=& \int  \lim_{n\to\infty} -\frac{1}{n}
\log 
\frac{G\bigl(o,X_n(\omega)\bigr|1\bigr)}{G(o,o|1)}\,d\Prob(\omega) =
\int \lim_{n\to\infty} -\frac{1}{n} \log  G\bigl(o,X_n(\omega)|1\bigr)\,d\Prob(\omega).
\end{eqnarray*}
Recall that $\pi_n$ denotes the distribution of $X_n$. Since
$$
G(o,X_n|1) = \sum_{m\geq 0} p^{(m)}(o,X_n) \geq p^{(n)}(o,X_n) = \pi_n(X_n),
$$
we have
\begin{equation}\label{equ:liminf-h}
\frac{\ell\cdot  H(\Y)}{\lambda} \leq \int \liminf_{n\to\infty} -\frac{1}{n} \log \pi_n\bigl(X_n(\omega)\bigr)\,d\Prob(\omega).
\end{equation}
The next aim is to prove that $\limsup_{n\to\infty} -\frac{1}{n}\E\bigl[\log
\pi_n(X_n)\bigr] \leq \ell\cdot H(\Y)/\lambda$. 
We now apply Lemma \ref{cut-lemma} by setting 
$$
A_n:=\sum_{m\geq 0} p^{(m)}(o,X_n),\ 
a_n:=\sum_{m=0}^{n^2} p^{(m)}(o,X_n) \textrm{ and } b_n:=\sum_{m\geq n^2+1}
p^{(m)}(o,X_n).
$$
By Lemma \ref{lemma:prob-bound},
$$
b_n\leq \sum_{m\geq n^2+1} D_1 \cdot  D_2^n \cdot  \varrho^{-m}
= D_1\cdot D_2^n   \cdot
\frac{\varrho^{-n^2-1}}{1-\varrho^{-1}}. 
$$
Therefore, $b_n$ decays faster than any geometric sequence. Applying Lemma
\ref{cut-lemma} together with (\ref{equ:G-L}) gives almost surely
$$
\frac{\ell \cdot H(\Y)}{\lambda}=
\lim_{n\to\infty} -\frac{1}{n} \log L(o,X_n) =
\lim_{n\to\infty} -\frac{1}{n} \log G(o,X_n) =
\lim_{n\to\infty}  -\frac{1}{n} \log \sum_{m=0}^{n^2}
p^{(m)}\bigl(o,X_n\bigr).
$$
Due to Lemma \ref{lemma:sum-bounds} we can apply the Dominated Convergence Theorem and get:
\begin{eqnarray*}
&&\frac{\ell \cdot H(\Y)}{\lambda} =\int \frac{\ell \cdot H(\Y)}{\lambda}\,d\Prob 
 =   \int \lim_{n\to\infty} -\frac{1}{n} \log \sum_{m=0}^{n^2}
p^{(m)}(o,X_n)\, d\Prob \\
&=&
\lim_{n\to\infty} \int -\frac{1}{n} \log \sum_{m=0}^{n^2}
p^{(m)}(o,X_n)\, d\Prob
= \lim_{n\to\infty} -\frac{1}{n} \sum_{w\in\calL} p^{(n)}(o,w) \log
 \sum_{m=0}^{n^2}
p^{(m)}(o,w).
\end{eqnarray*}
For $w\in\calL$, define the following distribution $\mu_0$ on $\calL$: 
$$
\mu_0(w):=\frac{1}{n^2+1} \sum_{m=0}^{n^2} p^{(m)}(o,w).
$$
Recall that the non-negativity of the Kullback-Leibler divergence (in this
context also called \textit{Shannon's Inequality}) gives
$$
-\sum_{w\in\calL} p^{(n)}(o,w) \log \mu_0(w) \geq -\sum_{w\in\calL} p^{(n)}(o,w) \log p^{(n)}(o,w).
$$
Therefore,
\begin{eqnarray*}
\frac{\ell\cdot H(\Y)}{\lambda}& \geq & \limsup_{n\to\infty} -\frac{1}{n} \sum_{w\in\calL}
p^{(n)}(o,w) \log (n^2+1) -\frac{1}{n} \sum_{w\in\calL} p^{(n)}(o,w) \log
p^{(n)}(o,w) \\
&=& \limsup_{n\to\infty} -\frac{1}{n}\int \log \pi_n(X_n)\, d\Prob.
\end{eqnarray*}
Now we can conclude with (\ref{equ:liminf-h}) and Fatou's Lemma:
\begin{eqnarray*}
\frac{\ell\cdot H(\Y)}{\lambda} & \leq & \int \liminf_{n\to\infty} -\frac{1}{n}\log \pi_n(X_n) d\Prob \leq 
\liminf_{n\to\infty} \int -\frac{1}{n}\log \pi_n(X_n) d\Prob \\
& \leq &
\limsup_{n\to\infty} \int -\frac{1}{n}\log \pi_n(X_n) d\Prob \leq \frac{\ell\cdot H(\Y)}{\lambda}.
\end{eqnarray*}
Thus, the asymptotic entropy $h:=\lim_{n\to\infty} -\frac{1}{n} \E\bigl[\log \pi_n(X_n)\bigr]$ exists
and equals $\ell\cdot H(\Y)/\lambda$. 
\end{proof}
Finally, we can prove:
\begin{proof}[Proof of Corollary \ref{cor:entropy-convergence}]
The proofs of the statements in Corollary \ref{cor:entropy-convergence} are completely analogous to the proofs in \cite[Corollary 3.9, Lemma
3.10]{gilch:11}, where \cite[Lemma 3.10]{gilch:11} holds also in the case
$h=0$. 
\end{proof}
\begin{proof}[Proof of Corollary \ref{cor:green-distance}]
Recall the definition of $F(o,w)$ from the proof of Lemma
\ref{lemma:sum-bounds} and the equation $G(o,w|1)=F(o,w)G(w,w|1)$. This yields together with (\ref{equ:G-L}):
$$
\Prob[\exists n\in\N_0: X_n=w] = F(o,w) = \frac{G(o,w|1)}{G(w,w|1)}=\frac{G(o,o|1)}{G(w,w|1)}L(o,w|1).
$$
Since $1\leq G(X_n,X_n|1)\leq 1/(1-\frac{1}{R})$ with $R$ from Lemma
\ref{lemma:R>1}, we obtain the proposed result due to Proposition \ref{prop:L-limit}.
\end{proof}

\section{Calculation of the Entropy}
\label{sec:entropy-calculation}

In this section we collect several results about the asymptotic entropy. We show how the entropy can be calculated numerically or even exactly in some special cases, and we give some inequalities.

\subsection{Numerical Calculation and Inequalities}

In order to compute $h=\ell\cdot  H(\Y)/\lambda$ we have to calculate the
three factors: while there are formulas for $\ell$ (see \cite[Theorem 2.4]{gilch:08}) and
$\lambda$ (given by (\ref{equ:lambda})), it remains to explain how to calculate
$H(\Y)$. For this purpose, define for random variables $A_1,\dots,A_n$ on a
finite state space $\mathcal{W}_A$ the \textit{joint entropy} as
$$
H(A_1,\dots,A_n):=
-\sum_{a_1,\dots,a_n\in\mathcal{W}_A}
\Prob\bigl[A_1=a_1,\dots,A_n=a_n\bigr]\log
\Prob\bigl[A_1=a_1,\dots,A_n=a_n\bigr],
$$
and let the \textit{conditional entropy} $H(A_n|A_1,\dots,A_{n-1})$ be defined as
$$
-\sum_{a_1,\dots,a_n\in\mathcal{W}_A}
\Prob\bigl[A_1=a_1,\dots,A_n=a_n\bigr]\log
\Prob\bigl[A_n=a_n\bigl| A_1=a_1,\dots,A_{n-1}=a_{n-1}\bigr].
$$
Here, we set $0\cdot \log 0:=0$, since $x\log x\to 0$ as $x\to 0+$.
By Cover and Thomas \cite[Theorem 4.2.1]{cover-thomas}, we have $H(\Y)=\lim_{n\to\infty} \frac1n H(\Y^{(\nu)}_1,\dots,\Y^{(\nu)}_n)$.
In general, the computation of $H(\Y)$
is a hard task. But there is a simple way for a numerical calculation of $H(\Y)$, which follows from the inequalities
\begin{equation}\label{equ:H-inequalities}
 H\bigl(\Y^{(\nu)}_n \bigl|
 \bigl((\mathbf{i}_1^{(\nu)},\mathbf{W}_1^{(\nu)}),(\mathbf{i}_2^{(\nu)},\mathbf{W}_2^{(\nu)})\bigr),\Y^{(\nu)}_1,\dots,\Y^{(\nu)}_{n-1}\bigr)
\leq
H(\Y) \leq  H(\Y^{(\nu)}_n\mid \Y^{(\nu)}_1,\dots,\Y^{(\nu)}_{n-1})
\end{equation}
for all $n\in\N$; see \cite[Theorem 4.5.1]{cover-thomas}. In particular, it is
even shown that
$$
H(\Y^{(\nu)}_n\mid \Y^{(\nu)}_1,\dots,\Y^{(\nu)}_{n-1})-H\bigl(\Y^{(\nu)}_n \bigl| \bigl((\mathbf{i}_1^{(\nu)},\mathbf{W}_1^{(\nu)}),(\mathbf{i}_2^{(\nu)},\mathbf{W}_2^{(\nu)})\bigr),\Y^{(\nu)}_1,\dots,\Y^{(\nu)}_{n-1}\bigr) \xrightarrow{n\to\infty} 0.
$$
Hence, one can calculate $H(\Y)$ numerically up to an arbitrarily small
error. Obviously, this numerical approach depends strongly on the ability to
solve the system of equations given by (\ref{h-equations}). 
\par
We now investigate whether the entropy is non-zero or not.
\begin{Cor}\label{cor:h-zero}
If the random walk is expanding, then $h>0$. Otherwise, $h=0$.
\end{Cor}
\begin{proof}
Take any $(i_{k,l},w_1),(j_{p,q},w_2)\in\mathcal{W}$ with
$$
\Prob[(\mathbf{i}_1^{(\nu)},\W_1^{(\nu)})=(i_{k,l},w_1),(\mathbf{i}_2^{(\nu)},\W_2^{(\nu)})=(j_{p,q},w_2)]>0.
$$
The
values $(i_{k,l},w_1),(j_{p,q},w_2)$ determine the value of $\Y_1^{(\nu)}$ uniquely. In the expanding
case, there are at least two elements $(s_{j,m},w'),
(t_{j,n},w'')\in\mathcal{W}$ such that $w',w''\in C([w_2])$ with
$C(w')\cap C(w'')=\emptyset$ and
$q(w_2,w')>0$ and $q(w_2,w'')>0$, yielding $\pi\bigl((j_{p,q},w_2),(s_{j,m},w') \bigr) \neq \pi\bigl((j_{p,q},w_2),(t_{j,n},w'') \bigr)$. Let $w'$ be in the $m$-th cone
of type $s$ in the covering of $C([w_2])$. Then set
\begin{eqnarray*}
&&P\bigl((i_{k,l},w_1),(j_{p,q},w_2),(s_{j,m},w')\bigr)\\
&:=&\Prob\bigl[\Y_2^{(\nu)}=\bigl(\tau(C(w_2)),s_m\bigr) \,\bigl|\,
(\mathbf{i}_1^{(\nu)},\W_1^{(\nu)})=(i_{k,l},w_1),(\mathbf{i}_2^{(\nu)},\W_2^{(\nu)})=(j_{p,q},w_2)\bigr]\\
&\geq& q(w_2,w')>0.
\end{eqnarray*}
Since $q(w_2,w'')>0$ and $C(w')\cap C(w'')=\emptyset$, we also  have
$P\bigl((i_{k,l},w_1),(j_{p,q},w_2),(t_{j,n},w'')\bigr)>0$ implying $P\bigl((i_{k,l},w_1),(j_{p,q},w_2),(s_{j,m},w')\bigr)<1$. From (\ref{equ:H-inequalities}) follows then
\begin{eqnarray*}
H(\Y) &\geq &   H\bigl(\Y^{(\nu)}_2 \bigl|
\bigl((\mathbf{i}_1^{(\nu)},\mathbf{W}_1^{(\nu)}),(\mathbf{i}_2^{(\nu)},\mathbf{W}_2^{(\nu)})\bigr),\Y^{(\nu)}_1\bigr)\\
&\geq&  P\bigl((i_{k,l},w_1),(j_{p,q},w_2),(s_{j,m},w')\bigr) \log P\bigl((i_{k,l},w_1),(j_{p,q},w_2),(s_{j,m},w')\bigr)>0.
\end{eqnarray*}
Thus, we have shown that $h>0$ if $(X_n)_{n\in\N_0}$ is expanding.
\par
Now consider the case when the random walk on $\calL$ is \textit{not}
expanding. Then each cone has a covering consisting of only one single subcone. This implies that
the value $\tau(C(\W_1^{(\nu)}))=\mathbf{i}_1^{(\nu)}$ determines uniquely the values
$\tau(C(\W_k^{(\nu)}))$ for $k\geq 2$. Moreover, given the value of $\tau(C(\W_1^{(\nu)}))$
the values of
$\Y_k^{(\nu)}$, $k\geq 1$,
are deterministic. That is, $\Y_n^{(\nu)}$ is uniquely determined
by $\Y_1^{(\nu)}$, hence $\Prob[\Y_n^{(\nu)}=\cdot \mid \Y_1^{(\nu)}=(s,t_n)]\in\{0,1\}$.
This implies
%
$$
0\leq H(\Y) \leq H(\Y^{(\nu)}_n\mid \Y^{(\nu)}_1,\dots,\Y^{(\nu)}_{n-1})\leq
H(\Y^{(\nu)}_n\mid \Y^{(\nu)}_1)=0,
$$
where the last inequality follows from \cite[Theorem
2.6.5]{cover-thomas}. Thus, $h=0$.
\end{proof}
In order to get a complete picture, we show that the entropy is zero for recurrent random walks:
\begin{Cor}\label{cor:recurrent-h-zero}
If $(X_n)_{n\in\N_0}$ is recurrent then $h=0$.
\end{Cor}
\begin{proof}
Clearly, $-\frac1n\mathbb{E}\bigl[\log \pi_n(X_n)\bigr]\geq 0$. Assume now that
$\limsup_{n\to\infty} -\frac1n\mathbb{E}\bigl[\log \pi_n(X_n)\bigr]=c>0$. Then
there is a (deterministic) sequence $(n_k)_{k\in\N}$ such that, for any $\varepsilon_1\in(0,c)$,
\begin{equation}\label{equ:h=0-case}
-\frac{1}{n_k}\mathbb{E}\bigl[\log \pi_{n_k}(X_{n_k})\bigr]\geq c-\varepsilon_1>0
\end{equation}
for all $k\in\mathbb{N}$. Denote by $\varepsilon_0$ the minimal occuring positive
single-step transition probability of $(X_n)_{n\in\N_0}$. Then $-\frac{1}{n_k}\log
\pi_{n_k}(X_{n_k})\leq -\log \varepsilon_0$. Choose $N\in\N$ with
$1/N<c-\varepsilon_1$. Then there is some $\delta >0$ with
$$
\Prob\Bigl[-\frac{1}{n_k}\log
\pi_{n_k}(X_{n_k})\geq \frac{1}{N}\Bigr]\geq \delta \quad \forall k\in\N.
$$
To see this, assume that $\delta=\delta_{k}$ depends on $k$ with
$\liminf_{k\to\infty} \delta_{k}=0$: then we get with (\ref{equ:h=0-case}) 
$$
(-\log \varepsilon_0)\cdot \delta_k +(1-\delta_k)\frac{1}{N}\geq
-\frac{1}{n_k}\mathbb{E}\bigl[\log \pi_{n_k}(X_{n_k})\bigr]\geq c-\varepsilon_1;
$$
If $\delta_k$ tends to zero then we get a contradiction to the choice of $N$.
\par
Choose now $\varepsilon>0$ arbitrarily small with $\varepsilon<\delta$. Since $\ell=0$ in the recurrent case, there is some index $K\in\N$
such that for all $k\geq K$:
$$
\delta-\varepsilon \leq
\Prob\bigl[-\log\pi_{n_k}(X_{n_k})\geq n_k/N,|X_{n_k}|\leq
\varepsilon n_k\bigr]
\leq e^{-n_k/N}\cdot |\calA|^{\varepsilon n_k}
$$
which yields the inequality
$$
\frac{1}{N}+\frac{1}{n_k}\log(\delta-\varepsilon)
\leq \varepsilon \log |\calA|.
$$
But this gives a contradiction if we make $\varepsilon$ sufficiently small
since the right hand side tends to zero, but the left hand side to
$\frac{1}{N}$ as $k\to\infty$. Thus, $\limsup_{n\to\infty} -\frac1n\mathbb{E}\bigl[\log
\pi_n(X_n)\bigr]=0$, yielding $h=0$.
\end{proof}

Finally, we state an inequality which connects entropy, drift and growth. For
this purpose, define $\calA_{\leq n}^\ast=\{w\in\calA^\ast \mid |w|\leq n\}$ for
$n> 0$. The \textit{growth} of $\calA^\ast$ is then given by $g:=\lim_{n\to\infty}
\frac1n \log |\calA_{\leq n}^\ast|$. Since 
$|\calA^n| \leq |\calA_{\leq n}^\ast| \leq n |\calA^n|$, we have $g=\log |\calA|$.
We get the following connection between
entropy, drift and growth:
\begin{Th}\label{prop:h-inequality}
$h\leq \ell\cdot \log|\calA|$.
\end{Th}
\begin{proof}
Let be $\varepsilon>0$. By Corollary \ref{cor:entropy-convergence} (1), there is some $N_\varepsilon\in\N$ such that for all
$n\geq N_\varepsilon$:
$$
1-\varepsilon \leq
\Prob\bigl[-\log\pi_n(X_n)\geq (h-\varepsilon)n,|X_n|\leq
(\ell+\varepsilon)n\bigr]
\leq e^{-(h-\varepsilon)n}\cdot |\calA_{\leq (\ell+\varepsilon)n}^\ast|.
$$
Taking logarithms and dividing by $n$ gives
$$
(h-\varepsilon) +\frac1n\log(1-\varepsilon)\leq (\ell+\varepsilon)\cdot \frac{1}{(\ell+\varepsilon)n}\log |\calA_{\leq (\ell+\varepsilon)n}^\ast|.
$$
Making $\varepsilon$ arbirtraily small and sending $n\to\infty$ yields the proposed claim.
\end{proof}
Let us remark that similar inequalities have been proved by Kaimanovich and Woess \cite{kaimanovich-woess} for time and space homogeneous random walks and in \cite{gilch:11} for random walks on free products.

\subsection{Exact Formula for Unambiguous Cone Boundaries}

In this subsection we give an exact formula for the asymptotic entropy in some special case.
We call $ab\in\calA^2$ \textit{unambiguous} if  $\partial
C(ab)=\{ab\}$. In other words, whenever the random walk enters a subcone of
type $C(wab)$, $w\in\calA^\ast$, it must enter it through its single boundary
point $wab$. We call the cone type $\tau(C(ab))$ also unambiguous. Existence of an unambiguous cone allows
us to ``cut'' the random walk into i.i.d. pieces and to obtain a formula for the
entropy $H(\Y)$. For $n\in\N$, $x_2,\dots,x_n\in\mathcal{W}_0$ and unambiguous $ab\in\calA^2$ define
\begin{eqnarray*}
w(ab,x_2,\dots,x_n) & := &\Prob\bigl[\W_2=x_2,\dots,\W_n=x_n,[\W_{n}]=ab \bigl|
[\W_{1}]=ab\bigr],\\
\tilde w(ab,x_2,\dots,x_n) & := &
\sum_{\substack{y_2,\dots,y_n\in\mathcal{W}_0:\\ y_i\in \partial C(x_i)\\
    \textrm{for } 2\leq i \leq n}}\Prob\bigl[\W_2=y_2,\dots,\W_{n}=y_n,[\W_{n}]=ab \bigl|
[\W_{1}]=ab\bigr],
\end{eqnarray*}
In particular, $\tilde w(ab,x_2)=\Prob\bigl[\W_2=x_2,[\W_{2}]=ab\bigl|
[\W_{1}]=ab\bigr]$.
Recall that $\nu$ denotes the invariant probability measure of the process
$(\mathbf{i}_k,\W_k)_{k\in\N}$. For unambiguous $ab\in\calA^2$, set
$$
\nu_{ab}:=\sum_{(i_{m,n},x)\in\mathcal{W}: [x]=ab} \nu(i_{m,n},x).
$$
Then:
\begin{Th}\label{thm:unambiguous}
If $ab\in\calA^2$ is unambiguous, then
$$
H(\Y)=-\nu_{ab}
\sum_{n\geq 1} \sum_{\substack{x_2,\dots x_{n-1}\in\mathcal{W}_0:\\
    [x_i]\neq ab \,\mathrm{ for }\, 2\leq i \leq n-1}}
\sum_{\substack{x_n\in\mathcal{W}_0:\\ [x_n]=ab}} w(ab,x_2,\dots,x_n) \log \tilde w(ab,x_2,\dots,x_n). 
$$
\end{Th}
\begin{proof}
Write $\alpha:=\tau(C(ab))$. By Proposition \ref{prop:Y-equal-Y-hat}, we have 
$$
-\frac1n \log \Prob[\Y_1=\underline{y}_1,\dots, \Y_n=\underline{y}_n] \xrightarrow{n\to\infty} H(\Y)
$$
for almost every trajectory
$(\underline{y}_1,\underline{y}_2,\dots)\in\mathcal{W}_{\pi}^{\mathbb{N}}$. For any such
trajectory, we define 
$$
N_0:=\min\bigl\lbrace m\in \N \bigl| \tau(\W_{m+1})=\alpha\bigr\rbrace \textrm{
  and } N_k:=\min\bigl\lbrace m\in \N \bigl|
m>N_{k-1}, \tau(\W_{m+1})=\alpha\bigr\rbrace.
$$
Define $d(n):=\max\{ k\in\N_0 \mid N_k\leq n\}$. Since $\Y_{N_j}$ has the form $(t,\alpha_{t^{(n)},m})$ for some cone type $t\in\mathcal{I}$, $1\leq m \leq n(t,\alpha)$, and $[\W_{N_k+1}]=ab$ for all $k\in\N$ we can use the strong Markov
property as follows for all $ n\geq 1$ and almost every trajectory
$(\underline{y}_1,\underline{y}_2,\dots)\in\mathcal{W}_{\pi}^{\mathbb{N}}$ :
\begin{eqnarray*}
&&\Prob\bigl[\Y_{N_j+1}=\underline{y}_{N_j+1},\dots, \Y_{N_j+n}=\underline{y}_{n}\mid
\Y_1=\underline{y}_1,\dots, \Y_{N_j}=\underline{y}_{N_j}
\bigr] \\
&=& \Prob\bigl[\Y_{N_j+1}=\underline{y}_{N_j+1},\dots, \Y_{N_j+n}=\underline{y}_{n}\mid
[\W_{N_j+1}]=ab\bigr].
\end{eqnarray*}
In other words, the $\Y_k$'s collect only the information which cones are
entered successively, but we know that the $(N_j+1)$-th cone is entered through
a boundary point with last two letters $ab$; hence, one can restart the process
at some word ending with $ab$ in the above equation without changing probabilities.
Therefore, we can rewrite the following probability
$\Prob\bigl[\Y_1=\underline{y}_1,\dots, \Y_{d(n)}=\underline{y}_{d(n)}\bigr]$ as
$$
\Prob\bigl[\Y_1=\underline{y}_1,\dots, \Y_{N_0}=\underline{y}_{N_0}\bigr]
\prod_{i=0}^{d(n)-1} \Prob\bigl[\Y_{N_i+1}=\underline{y}_{N_i+1},\dots, \Y_{N_{i+1}}=\underline{y}_{N_{i+1}}\bigl| [\W_{N_{i}+1}]=ab\bigr].
$$
%
%
Observe that the terms $\log \Prob\bigl[\Y_{N_i+1}=\cdot,\dots,
\Y_{N_{i+1}}=\cdot\bigl| [\W_{N_{i}+1}]=ab\bigr]$, $i\in\N$, are i.i.d., since
one can think of starting at some $\W_k$ with $[\W_k]=ab$ and stopping at the first
time $l>k$ with $[\W_l]=ab$.
By the ergodic theorem for positive recurrent Markov chains, $d(n)/n$ tends almost surely to $\nu_{ab}$. Hence, if we
consider only the subsequence where $n$ equals one of the $N_k$'s we obtain the
following convergence for almost every trajectory
$(\underline{y}_1,\underline{y}_2,\dots)\in\mathcal{W}_\pi^{\N}$ by classical
ergodic theory:
\begin{eqnarray*}
&&-\frac1n \log \Prob\bigr[\Y_1=\underline{y}_1,\dots, \Y_{d(n)}=\underline{y}_{d(n)}\bigr]\\
&=& -\frac{d(n)}{n}\frac{1}{d(n)} \biggl[ \log
\Prob\bigl[\Y_1=\underline{y}_1,\dots, \Y_{N_0}=\underline{y}_{N_0}\bigr]\\
&&\quad +\sum_{i=0}^{d(n)-1}\log
\Prob\bigl[\Y_{N_i+1}=\underline{y}_{N_i+1},\dots, \Y_{N_{i+1}}=\underline{y}_{N_{i+1}}\bigl| [\W_{N_{i}+1}]=ab\bigr]\biggr]\\
%
&\xrightarrow{n\to\infty}&-\nu_{ab} \sum_{k\geq 1} \sum_{\substack{x_2,\dots,
    x_{k-1}\in\mathcal{W}_0:\\ [x_i]\neq ab\\
\textrm{for } 2\leq i\leq k-1}}
\sum_{\substack{x\in\mathcal{W}_0:\\
    [x]=ab }}w(ab,x_2,\dots,x_{k-1},x)\log \tilde w(ab,x_2,\dots,x_{k-1},x).
\end{eqnarray*}
This proves the claim.
\end{proof}

\section{Analyticity of Entropy}
\label{sec:entropy-analyticity}

The random walk on $\calA^\ast$ depends on \textit{finitely} many parameters which are
described by the transition probabilities $p(w_1,w_2)$,
$w_1,w_2\in\calA^\ast$ with $|w_1|\leq 2$ and $|w_2|\leq 3$; see (\ref{equ:random-walk}). That is, each
random walk on $\calA^\ast$ can be defined via a vector
$\underline{p}\in\mathbb{R}_+^{|\calB|}$, where
$$
\calB:=\Bigl\lbrace (w_1,w_2)\,\Bigl|\, w_1\in \calA \cup\calA^2\cup\{o\},w_2\in
\bigcup_{n=1}^3\calA^n\cup\{o\},\bigl||w_1|-|w_2|\bigr|\leq 1\Bigr\rbrace.
$$ 
In other words, the entry of $\underline{p}$ associated with the index
$(w_1,w_2)\in\calB$ describes the value of $p(w_1,w_2)$.
The support $\mathrm{supp}(\underline{p})$ of $\underline{p}$ is the
set of indices in $\calB$ corresponding to non-zero entries of
$\underline{p}$. Fix now any $\underline{p}_0\in\mathbb{R}_+^{|\calB|}$
such that $\underline{p}_0$ describes a well-defined, transient random walk on
$\calA^\ast$, and let $\mathcal{P}(\underline{p}_0)$ be the set of vectors
$\underline{p}\in\mathbb{R}^{|\calB|}$ with support $\mathrm{supp}(\underline{p}_0)$ which
allow well-defined, transient random walks on $\calA^\ast$. The set
$\mathcal{P}(\underline{p}_0)$ can be described by an open polygonal bounded
convex set in $\mathbb{R}^d$ with some suitable $d\leq |\calB |-1$ which depends on $\mathrm{supp}(\underline{p}_0)$; recall that
$\ell>0$ if and only if $(X_n)_{n\in\N_0}$ is transient, and from the formula
of $\ell$ in \cite[Theorem 2.4]{gilch:08} follows that $\ell$ varies
continuously in $\underline{p}$, yielding that there is some open neighbourhood
of $\underline{p}_0$ in $\mathbb{R}^d$ where $(X_n)_{n\in\N_0}$ remains still transient.
We now ask whether
the entropy mapping $\underline{p}\mapsto h=h_{\underline{p}}$ varies
real-analytically on $\mathcal{P}(\underline{p}_0)$. 
\par
In the next subsection we will introduce a new Markov chain which is related to
the last entry time process and leads under the projection $\pi(\cdot,\cdot)$ to a hidden
Markov chain with same distribution as $(\Y_k)_{k\in\N}$. Afterwards we will be
able to prove Theorem \ref{thm:entropy-continuity} in Subsection \ref{subsec:proof-analytic}.

\subsection{Modified Last Entry Time Process}

The aim of this subsection is the construction of a Markov chain related to the
last entry time process $(\mathbf{i}_k,\W_k)_{k\in\mathbb{N}}$ such that the transition matrix
has strictly positive entries and the modified process leads under
$\pi(\cdot,\cdot)$ (see (\ref{equ:pi-def})) to a hidden
Markov chain with same asymptotic entropy.
\par 
Let be $ab,a_1b_1,a_2b_2\in\calA^2$, and let $C_{j_{i,1}}$ be the first cone of
type $j$ in the covering of $C(a_1b_1)$ with $\tau(C(a_1b_1))=i$ and let $C_{j_{k,l}}$ be the $l$-th subcone of type
$j$ in the covering of $C(a_2b_2)$ with $\tau(C(a_2b_2))=k$. Assume that
$y_0\in\partial C_{j_{k,l}}$ with $[y_0]=ab$. Since $C_{j_{i,1}}$ and
$C_{j_{k,l}}$ are isomorphic, there is
some unique $\bar y_0^{[i,j,ab]}\in\calA^\ast$ such that $\bar y_0^{[i,j,ab]}
ab\in\partial C_{j_{i,1}}$; see \mbox{Section \ref{subsec:cone-def}.}
In the following we will sometimes omit the superindex $[i,j,ab]$ and use the notation $\bar y_0=\bar y_0^{[i,j,ab]}$ for describing this replacement.
\par
For $i,j\in \calI$ and $ab\in\calA^2$ with $\tau(C(ab))=j$, we write
$$
\#\{j_{s,t}\mid s\neq i,ab\}:=
\bigl|\bigl\lbrace
(j_{s,t},w)\in\mathcal{W} \bigl| [w]=ab,s\in\calI\setminus\{i\},1\leq t\leq n(s,j)
\bigr\rbrace\bigr|.
$$
It is not hard to see that $\#\{j_{s,t}\mid s\neq i,a_1b_1\}=\#\{j_{s,t}\mid s\neq i,a_2b_2\}$
if $\tau(C(a_1b_1))=\tau(C(a_2b_2))$ but this will not be relevant for our
proofs, so we omit further explanations.
Let be $(i_{k,l},x),(j_{m,n},y)\in\mathcal{W}$ with $[y]=ab\in\calA^2$. This
implies $\tau(C(x))=i$ and $y^{[i,j,ab]}\in\partial C_{j_{i,1}}$, where
$C_{j_{i,1}}$ is the first cone of type $j$ in the covering of $C([x])$. 
Define the following
transition probabilities on $\mathcal{W}$:
\begin{eqnarray*}
\hat q\bigl((i_{k,l},x),(j_{m,n},y)\bigr) :=
\begin{cases}
\frac{1}{\#\{j_{s,t}\mid s\neq i, ab\}+1} \frac{\xi([y])}{\xi([x])}\mathds{L}(x,y), & \textrm{if } m=i \land n=1,\\
\frac{\xi([y])}{\xi([x])}\mathds{L}(x,y),  & \textrm{if } m=i \land n\geq 2,\\
\frac{1}{\#\{j_{s,t}\mid s\neq i, ab\}+1} \frac{\xi([y])}{\xi([x])}\mathds{L}(x,\bar
y^{[i,j,ab]}ab),  &\textrm{if } m\neq i.
\end{cases}
\end{eqnarray*}
It is easy to see that these transition probabilities define a Markov
chain (inherited from the Markov chain $(\mathbf{i}_k,\W_k)_{k\in\N}$): in the
case $m=i \land n\geq 2$ we just have 
$$
\hat q\bigl((i_{k,l},x),(j_{m,n},y)\bigr)= \Prob\bigl[(\mathbf{i}_2,\W_2)=(j_{m,n},y) \mid (\mathbf{i}_1,\W_1)=(i_{k,l},x)\bigr];
$$
otherwise we have, for $(j_{i,1},y)\in \mathcal{W}$, 
\begin{eqnarray*}
&&\hat q\bigl((i_{k,l},x),(j_{i,1},y)\bigr)+
\sum_{\substack{(j_{s,t},w)\in\mathcal{W}: \\ s\neq i,[w]=ab,\\ 1\leq t\leq
    n(s,j)}} \hat q\bigl((i_{k,l},x),(j_{s,t},w)\bigr) \\
&=& 
\Prob\bigl[(\mathbf{i}_2,\W_2)=(j_{i,1},y) \mid (\mathbf{i}_2,\W_2)=(i_{k,l},x)\bigr]
\end{eqnarray*}
since $y=\bar y^{[i,j,ab]}ab$ by definition. In other words,
each step from $(i_{k,l},x)$ to $(j_{m,n},y)$ either behaves according to
(\ref{equ:Q-def}) (case $m=i$ and $n\geq 2$) or the step  from $(i_{k,l},x)$ to
$(j_{i,1},y)$ (when seen as a step of the process
$(\mathbf{i}_k,\W_k)_{k\in\N}$)) is split up into different equally likely
steps $(i_{k,l},x)$ to $(j_{m,n},\bar yab)$ with $m\neq i$ or $m=i\land n=1$.
Observe that the transitions depend only on $[x]$ in the first argument of
$\hat q(\cdot,\cdot)$.
By Proposition \ref{prop:i-W-markov},
the transition matrix $\widehat Q=\bigl(\hat q((i_{k,l},x),(j_{m,n},y))\bigr)$ is stochastic and governs a
positive recurrent, aperiodic Markov chain
$(\mathbf{t}_k,\mathbf{x}_k)_{k\in\N}$. In particular, $\widehat Q$ has strictly positive entries.
The initial distribution $\hat\mu_1$ of $(\mathbf{t}_1,\mathbf{x}_1)$ is  defined as 
$$
\hat\mu_1(i_{m,n},x):=\Prob[(\mathbf{i}_1,\mathbf{W}_1)=(i_{m,n},x)]>0
$$
for $(i_{m,n},x)\in\mathcal{W}$.
%
%
\par
The process $\bigl((\mathbf{t}_{k},\mathbf{x}_{k}),(\mathbf{t}_{k+1},\mathbf{x}_{k+1})
\bigr)_{k\in\mathbb{N}}$ is again a positive recurrent, aperiodic Markov chain whose transition
matrix is denoted by $\widehat{Q}_2$ (arising from $\widehat Q$).
We now define  a new hidden
Markov chain
$(\ZZ_k)_{k\in\N}$ by
$$
\ZZ_k:=\pi\bigl((\mathbf{t}_{k},\mathbf{x}_{k}),(\mathbf{t}_{k+1},\mathbf{x}_{k+1})\bigr).
$$
Observe that at this point the second branch in the definition of $\pi$ in (\ref{equ:pi-def}) comes into
play for the definition of $\ZZ_k$.
The crucial point is the following proposition:
\begin{Prop}\label{prop:Y-hatY-equation}
For all $(s^{(1)},t^{(1)}),\dots,(s^{(n)},t^{(n)})\in \mathcal{W}_\pi$,
$$
\Prob\bigl[\Y_1=(s^{(1)},t^{(1)}),\dots,\Y_n=(s^{(n)},t^{(n)})\bigr]=
\Prob\bigl[\ZZ_1=(s^{(1)},t^{(1)}),\dots,\ZZ_n=(s^{(n)},t^{(n)})\bigr].
$$
\end{Prop}
\begin{proof}
We prove the claim by induction
on $n$. First, let be $j,s\in\calI$ and $t^{(1)}=j_m$ with $2\leq m\leq n(s,j)$,
and let $a_0b_0,ab\in\calA^2$ with $\tau(C(a_0b_0))=s$ and $\tau(C(ab))=j$. If
$C_{j,m}$ is the $m$-th cone of type $j$ in the covering of $C(a_0b_0)$ then
there is a unique word $\bar x_0=\bar x_0^{[s,j,m,ab]}\in\calA^\ast$ with $\bar x_0ab\in\partial
C_{j,m}$. With this notation we get:
\begin{eqnarray*}
\Prob\bigl[\Y_1=(s,j_m),[\mathbf{W}_2]=ab\bigr] &=&
\sum_{(u_{k,l},x)\in\mathcal{W}:u=s}
\Prob[(\mathbf{i}_1,\mathbf{W}_1)=(s_{k,l},x)]\cdot  q(x,\bar x_0ab)\\
&=& \sum_{(u_{k,l},x)\in\mathcal{W}:u=s} \hat\mu_1(s_{k,l},x) \hat
q\bigl((s_{k,l},x),(j_{s,m},\bar x_0ab)\bigr)\\
&=& \Prob\bigl[\ZZ_1=(s,j_m),[\mathbf{x}_2]=ab\bigr].
\end{eqnarray*}
Now we turn to the case $t^{(1)}=j_1$. Once again, if $C_{j,1}$ is the first
cone of type $j$ in the covering of $C(a_0b_0)$ then there is some unique
$\bar x_0=\bar x_0^{[s,j,1,ab]}\in\calA^\ast$ with $\bar x_0ab\in\partial
C_{j,1}$. We get:
\begin{eqnarray*}
&&\Prob\bigl[\ZZ_1=(s,j_1),[\mathbf{x}_2]=ab\bigr] \\
&=& \sum_{(u_{k,l},x)\in\mathcal{W}:u=s} \hat\mu_1(s_{k,l},x) \Bigl[
 \hat q\bigl((s_{k,l},x),(j_{s,1},\bar x_0ab)\bigr)+\sum_{\substack{(t_{p,q},y)\in\mathcal{W}:\\ t=j,p\neq s,[y]=ab}} \hat
   q\bigl((s_{k,l},x),(j_{p,q},y)\bigr)\Bigr]\\
&=& \sum_{\substack{(u_{k,l},x)\in\mathcal{W}:\\ u=s}} \hat\mu_1(s_{k,l},x) \biggl[
 \frac{q\bigl((s_{k,l},x),(j_{s,1}, \bar x_0ab)\bigr)}{\# \{t_{\kappa_1,\kappa_2}
   \mid  \kappa_1\neq
   s, ab\}+1}+\sum_{\substack{(t_{p,q},y)\in\mathcal{W}:\\ t=j,p\neq
     s,\\ [y]=ab}} 
\frac{q\bigl((s_{k,l},x),(j_{s,1},\bar x_0ab)\bigr)}{\# \{t_{\kappa_1,\kappa_2} \mid
  \kappa_1\neq s, ab\}+1}\biggr]\\
&=& \sum_{(u_{k,l},x)\in\mathcal{W}:u=s}
\Prob\bigl[(\mathbf{i}_1,\mathbf{W}_1)=(s_{k,l},x)\bigr]\cdot  
q(x,\bar x_0ab)
= \Prob\bigl[\Y_1=(s,j_1),[\mathbf{W}_2]=ab\bigr].
\end{eqnarray*}
Now, in both cases we obtain
\begin{eqnarray*}
\Prob\bigl[\ZZ_1=(s,t^{(1)})\bigr] &=& \sum_{ab\in\calA^2}
\Prob\bigl[\ZZ_1=(s,t^{(1)}),[\mathbf{x}_2]=ab\bigr]\\
&=& \sum_{ab\in\calA^2}\Prob\bigl[\Y_1=(s,t^{(1)}),[\mathbf{W}_2]=ab\bigr]
= \Prob\bigl[\Y_1=(s,t^{(1)})\bigr].
\end{eqnarray*}
We now perform the induction step where we will use the induction assumption
\begin{eqnarray}\label{equ:ind-ass}
&&\Prob\bigl[\Y_1=(s^{(1)},t^{(1)}),\dots,\Y_{n}=(s^{(n)},t^{(n)}),[\mathbf{W}_{n+1}]=ab\bigr]\\
&=& \Prob\bigl[\ZZ_1=(s^{(1)},t^{(1)}),\dots,\ZZ_{n}=(s^{(n)},t^{(n)}),[\mathbf{x}_{n+1}]=ab\bigr]. \nonumber
\end{eqnarray}
First, consider the case $(s^{(n+1)},t^{(n+1)})=(s,j_m)$
with $s,j\in\calI$ and $2\leq m\leq n(s,j)$. This implies that $\mathbf{t}_{n+1}$ has
the form $s_{\ast,\ast}$ and $\mathbf{t}_{n+2}=j_{s,m}$. Let $C_{j,m}$ be the $m$-th cone of type $j$ in the covering of
$C(a_0b_0)$, where $a_0b_0\in\calA^2$ with $\tau(C(a_0b_0))=s$. If $ab\in\calA^2$ with 
$\tau(C(ab))=j$ then there is some unique $\bar x_0=\bar
x_0^{[s,j,m,ab]}\in\calA^\ast$ with $\bar x_0ab\in\partial
C_{j,m}$. In this case we obtain: 
\begin{eqnarray*}
&&
\Prob\bigl[\ZZ_1=(s^{(1)},t^{(1)}),\dots,\ZZ_{n+1}=(s^{(n+1)},j_m),[\mathbf{x}_{n+1}]=a_0b_0,[\mathbf{x}_{n+2}]=ab\bigr]\\
&=&
\sum_{\substack{(u_{k,l},w_0)\in\mathcal{W}:\\
    u=s, [w_0]=a_0b_0}}\Prob\bigl[\ZZ_1=(s^{(1)},t^{(1)}),\dots,\ZZ_{n}=(s^{(n)},t^{(n)}),\mathbf{t}_{k+1}=u_{k,l},
\mathbf{x}_{n+1}=w_0\bigr]\\
&&\quad\quad \cdot 
\hat q\bigl((s_{k,l},w_0),(j_{s,m},\bar x_0ab)\bigr)\\
&=&
\Prob\bigl[\ZZ_1=(s^{(1)},t^{(1)}),\dots,\ZZ_{n}=(s^{(n)},t^{(n)}),[\mathbf{x}_{n+1}]=a_0b_0\bigr]
\frac{\xi(ab)}{\xi(a_0b_0)}\mathds{L}(a_0b_0,\bar x_0ab)\\
&=& \Prob\bigl[\Y_1=(s^{(1)},t^{(1)}),\dots,\Y_{n}=(s^{(n)},t^{(n)}),[\mathbf{W}_{n+1}]=a_0b_0\bigr]
\frac{\xi(ab)}{\xi(a_0b_0)}\mathds{L}(a_0b_0,\bar x_0ab)\\
&=&
\Prob\bigl[\Y_1=(s^{(1)},t^{(1)}),\dots,\Y_{n+1}=(s^{(n+1)},j_m),[\mathbf{W}_{n+1}]=a_0b_0,[\mathbf{W}_{n+2}]=ab\bigr].
\end{eqnarray*}
Now we turn to the case $(s^{(n+1)},t^{(n+1)})=(s,j_1)$. This implies again that $\mathbf{t}_{n+1}$ has
the form $s_{\ast,\ast}$. Once again, if $C_{j,1}$ is the first
cone of type $j$ in the covering of $C(a_0b_0)$ (of type $s$) then there is
some unique
$\bar x_0=\bar x_0^{[s,j,1,ab]}\in\calA^\ast$ with $\bar x_0ab\in\partial
C_{j,1}$. We get by distinguishing whether $t^{(n+1)}=j_1$ arises from
$\mathbf{t}_{n+2}=j_{s,1}$ or $\mathbf{t}_{n+2}=j_{k,l}$ with $k\neq s$:
%
\begin{eqnarray*}
&&
\Prob\bigl[\ZZ_1=(s^{(1)},t^{(1)}),\dots,\ZZ_{n+1}=(s^{(n+1)},j_1),[\mathbf{x}_{n+1}]=a_0b_0,[\mathbf{x}_{n+2}]=ab\bigr]\\
&=& \sum_{\substack{(u_{p,q},w_0)\in\mathcal{W}:\\ u=s, [w_0]=a_0b_0}}
\Prob\bigl[\ZZ_1=(s^{(1)},t^{(1)}),\dots,\ZZ_{n}=(s^{(n)},t^{(n)}),\mathbf{t}_{n+1}=u_{p,q},\mathbf{x}_{n+1}=w_0\bigr]\\
&&\quad \cdot \Bigl( \hat q\bigl((s_{p,q},w_0),(j_{s,1},\bar x_0ab)\bigr)
+ \sum_{\substack{(t_{k,l},y)\in\mathcal{W}:\\ t=j,k\neq s,[y]=ab}} \hat q\bigl((s_{p,q},w_0),(j_{k,l},y)\bigr)\Bigr)\\
&=&
\Prob\bigl[\ZZ_1=(s^{(1)},t^{(1)}),\dots,\ZZ_{n}=(s^{(n)},t^{(n)}),[\mathbf{x}_{n+1}]=a_0b_0\bigr]
\\
&&\quad \cdot
\biggl[ \frac{\xi(ab)}{\xi(a_0b_0)}\frac{\mathds{L}(a_0b_0,\bar x_0ab)}{\#\{j_{k,l} \mid
  k\neq s, ab \}+1}+ \sum_{\substack{(t_{k,l},y)\in\mathcal{W}:\\ t=j, k\neq
    s,\\ [y]=ab}}
\frac{\xi(ab)}{\xi(a_0b_0)}\frac{\mathds{L}(a_0b_0,\bar x_0ab)}{\#\{j_{\kappa_1,\kappa_2}
  \mid  \kappa_1\neq s, ab \}+1}\biggr]\\
&=&
\Prob\bigl[\ZZ_1=(s^{(1)},t^{(1)}),\dots,\ZZ_{n}=(s^{(n)},t^{(n)}),[\mathbf{x}_{n+1}]=a_0b_0\bigr]
\frac{\xi(ab)}{\xi(a_0b_0)}\mathds{L}(a_0b_0,\bar x_0ab) \\
&=&
\Prob\bigl[\Y_1=(s^{(1)},t^{(1)}),\dots,\Y_{n}=(s^{(n)},t^{(n)}),[\mathbf{W}_{n+1}]=a_0b_0\bigr]
\frac{\xi(ab)}{\xi(a_0b_0)}\mathds{L}(a_0b_0,\bar x_0ab) \\
&=& \Prob\bigl[\Y_1=(s^{(1)},t^{(1)}),\dots,\Y_{n+1}=(s^{(n+1)},j_1),[\mathbf{W}_{n+1}]=a_0b_0,[\mathbf{W}_{n+2}]=ab\bigr].
\end{eqnarray*}
Hence,
\begin{eqnarray*}
&&\Prob\bigl[\ZZ_1=(s^{(1)},t^{(1)}),\dots,\ZZ_{n+1}=(s^{(n+1)},t^{(n+1)}),[\mathbf{x}_{n+2}]=ab\bigr]\\
&=&\sum_{a_0b_0\in\calA^2}
\Prob\bigl[\ZZ_1=(s^{(1)},t^{(1)}),\dots,\ZZ_{n+1}=(s^{(n+1)},t^{(n+1)}),[\mathbf{x}_{n+1}]=a_0b_0,[\mathbf{x}_{n+2}]=ab\bigr]\\
&=& \sum_{a_0b_0\in\calA^2}
\Prob\bigl[\Y_1=(s^{(1)},t^{(1)}),\dots,\Y_{n+1}=(s^{(n+1)},t^{(n+1)}),[\mathbf{W}_{n+1}]=a_0b_0,[\mathbf{W}_{n+2}]=ab\bigr]\\
&=&\Prob\bigl[\Y_1=(s^{(1)},t^{(1)}),\dots,\Y_{n+1}=(s^{(n+1)},t^{(n+1)}),[\mathbf{W}_{n+2}]=ab\bigr].
\end{eqnarray*}
This proves Equation (\ref{equ:ind-ass}) for all $n\in\N$, all $ab\in\calA^2$ and all $(s^{(1)},t^{(1)}),\dots,(s^{(n)},t^{(n)})\in
\mathcal{W}_\pi$. Finally, we obtain:
\begin{eqnarray*}
&&\Prob\bigl[\ZZ_1=(s^{(1)},t^{(1)}),\dots,\ZZ_{n+1}=(s^{(n+1)},t^{(n+1)})\bigr]\\
&=& \sum_{ab\in\calA^2}
\Prob\bigl[\ZZ_1=(s^{(1)},t^{(1)}),\dots,\ZZ_{n+1}=(s^{(n+1)},t^{(n+1)}),[\mathbf{x}_{n+2}]=ab\bigr]\\
&=&  \sum_{ab\in\calA^2}
\Prob\bigl[\Y_1=(s^{(1)},t^{(1)}),\dots,\Y_{n+1}=(s^{(n+1)},t^{(n+1)}),[\mathbf{W}_{n+2}]=ab\bigr]\\
&=& \Prob\bigl[\Y_1=(s^{(1)},t^{(1)}),\dots,\Y_{n+1}=(s^{(n+1)},t^{(n+1)})\bigr].
\end{eqnarray*}
This finishes the proof.
\end{proof}
The statement of the lemma can be formulated in other words: the
process governed by $\widehat Q$ can be seen as a last entry time process, where one has
more subcones to enter (namely, the subcones of indices $j_{k,l},k\neq i$, when
being currently in a cone of type $i$), but the projection $\pi$ (in particular
due to the second branch in its definition in (\ref{equ:pi-def})) folds
the process down to the same hidden Markov chain $(\Y_k)_{k\in\N}$ in terms of
probability. With Propositions \ref{prop:Y-equal-Y-hat} and \ref{prop:Y-hatY-equation} we immediately obtain:
\begin{Cor}\label{cor:Z-entropy}
For almost every realisation $\bigl((s^{(1)},t^{(1)}),(s^{(2)},t^{(2)}),\dots\bigr)\in
\mathcal{W}_\pi^{\N}$,
$$
H(\Y)=\lim_{n\to\infty} -\frac1n \log \Prob\bigl[\ZZ_1=(s^{(1)},t^{(1)}),\dots,\ZZ_{n}=(s^{(n)},t^{(n)})\bigr].
$$
\end{Cor} 
\begin{flushright}$\Box$\end{flushright}
The important difference between the underlying Markov chains $\bigl((\mathbf{t}_{k},\mathbf{x}_{k}),(\mathbf{t}_{k+1},\mathbf{x}_{k+1}) \bigr)_{k\in\mathbb{N}}$ and
$\bigl((\mathbf{i}_{k},\W_{k}),(\mathbf{i}_{k+1},\W_{k+1})\bigr)_{k\in\N}$ 
is that the transition matrix
$\widehat Q_2$ has strictly positive entries, while this must not necessarily hold for the transition matrix
of the Markov chain
$\bigl((\mathbf{i}_{k},\W_{k}),(\mathbf{i}_{k+1},\W_{k+1})\bigr)_{k\in\mathbb{N}}$. This
property will be important later.

\subsection{Proof of Theorem \ref{thm:entropy-continuity}}
\label{subsec:proof-analytic}

The crucial point will be the following lemma:
\begin{Lemma}\label{lemma:q-analytic}
The transition probabilities $q(w_1,w_2)$, $w_1,w_2\in\mathcal{W}_0$, vary real-analytically w.r.t. $\underline{p}\in\mathcal{P}(\underline{p}_0)$.
\end{Lemma}
\begin{proof}
In order to show that $q(w_1,w_2)$ varies real-analytically in $\underline{p}$
it suffices to show analyticity of $H(ab,c)$, $ab\in\calA^2$, $c\in\calA$,
and $\bar L(ab,cde)$, $d,e\in\calA$, due to Proposition \ref{prop:W-process}. 
The function $z\mapsto H(ab,c|z)$ has radius of
convergence bigger than $1$, which can be easily deduced from Lemma \ref{lemma:R>1}. Thus, for $\delta>0$
small enough, we have
$$
\infty > H(ab,c|1+\delta)=\sum_{n\geq 1} \Prob_{ab}[X_n=c,\forall m<n:
|X_m|\geq 2](1+\delta)^n.
$$
The probability $\Prob_{ab}[X_n=c,\forall m<n:
|X_m|\geq 2]$ can be rewritten as 
$$
\sum_{\substack{n_1,\dots,n_d\geq 1:\\ n_1+\dots+n_d=n}} c(n_1,\dots,n_d)p_1^{n_1}\cdot
\ldots \cdot p_d^{n_d},\quad c(n_1,\dots,n_d)\in\N_0,
$$
where $p_1,\dots,p_d$ correspond to the non-zero entries of the vector $\underline{p}$.
Therefore,
$$
H(ab,c|1+\delta)=\sum_{n\geq 1}\sum_{\substack{n_1,\dots,n_d\geq 1:\\ n_1+\dots+n_d=n}} c(n_1,\dots,n_d)(p_1(1+\delta))^{n_1}\cdot
\ldots \cdot (p_d(1+\delta))^{n_d}<\infty.
$$
Hence, $\underline{p}$ lies in the interior of the domain of convergence of
$H(ab,c|1)$ when considered as a multivariate power series in the variables of
$\mathrm{supp}(\underline{p})=\{p_1,\dots,p_d\}$. This yields real-analyticity of $H(ab,c|1)$ in
$\underline{p}$. Analyticity of $\xi(ab)$  follows now directly from its
definition. One can show completely analogously that the functions $\bar
L(ab,cde|1)$ vary also real-analytically in $\underline{p}$
 since $\bar L(ab,cde|z)$ has also radius of convergence bigger than $1$, which
 can also be easily deduced from Lemma \ref{lemma:R>1}. This proves the statement of the lemma.
\end{proof}
Now we can prove:

\begin{proof}[Proof of Theorem \ref{thm:entropy-continuity}]
The claim follows now via the equation $h=\ell\cdot
H(\Y)/\lambda$. By Lemma \ref{lemma:q-analytic}, the invariant probability
measure $\nu_0$ of the process $(\W_k)_{k\in\N}$ varies real-analytically in
some neighbourhood of $\underline{p}_0$,
since $\nu_0$ is the solution of a linear
system of equations in terms of $q(\cdot,\cdot)$; hence, $\lambda$ (given in
(\ref{equ:lambda})) varies
analytically. 
\par
Moreover, the transition matrix $\widehat{Q}_2$ of the process
$\bigl((\mathbf{t}_k,\mathbf{x}_k),(\mathbf{t}_{k+1},\mathbf{x}_{k+1})\bigr)_{k\in\N}$ 
has strictly positive entries. Therefore, we can apply the analyticity result for entropies of hidden Markov chains
of Han and Marcus \cite[Theorem 1.1]{han-marcus06} on $(\ZZ_k)_{k\in\N}$ and
obtain together with Corollary \ref{cor:Z-entropy}  that  $H(\Y)$
is also real-analytic in
some neighbourhood of $\underline{p}_0$; at this point it is crucial that $\widehat{Q}_2$  
has strictly positive entries in order to be able to apply \cite[Theorem
1.1]{han-marcus06}, which was our motivation for the definition of
the process $(\mathbf{t}_k,\mathbf{x}_k)_{k\in\N}$ and $(\mathbf{Z}_k)_{k\in\N}$. 
\par
Real-analyticity of $\ell$ can be shown completely analogously to the proof of
Lemma \ref{lemma:q-analytic} with the help of the formula for $\ell$ given in
\cite[Theorem 2.4]{gilch:08}.  This finishes the proof.
\end{proof}

\begin{appendix}
\label{subsec:remarks}
\section{Remarks on Assumptions \ref{ass:weak} and \ref{ass:suffix}}
\subsection{Generalization of Suffix-Irreducibility}
\label{subsec:remarks-1}
In this section we make a discussion on Assumption \ref{ass:suffix}, where we show how to relax this condition in some way and that it cannot be dropped completely. First, recall that suffix-irreducibility leads to the fact that the process
$(\mathbf{W}_k)_{k\in\N}$ is irreducible. One can weaken the asssumption of suffix-irreducibility to the assumption that
\begin{equation}\label{equ:suffix-new}
\Prob[\forall n\in\N: |X_n|\geq |w| \mid X_0=w]>0 \quad \forall w\in\calL,
\end{equation}
or equivalently that $H(ab,c|1)<1$ for all $a,b,c\in\calA$. This means that, for every $w\in\calL$, there is some $ab\in\calA^2$  such that 
$$
\Prob[\exists n\in\N: [X_n]=ab, \forall k\leq n: |X_k|\geq |w| \mid X_0=w]>0 \ \textrm{ and } \ H(ab,\cdot|1)<1.
$$
In this case the process $(\mathbf{W}_k)_{k\in\N}$ is not necessarily
irreducible any more, but it still has a finite state space. Let $C_1,\dots,C_r$ be the essential classes of the state
space of $(\mathbf{W}_k)_{k\in\N}$. Then $(\mathbf{W}_k)_{k\in\N}$ will almost
surely take only values in one of these classes up to finitely many exemptions
for small $k\in\N$; the class depends then on the concrete realization. If we
condition on the fact that $(\mathbf{W}_k)_{k\in\N}$ will finally enter the
class $C_i$, then -- on this event -- the entropy rate $h_i$ and the drift $\ell_i$ can be
calculated as shown in the irreducible case and as in \cite{gilch:08}: we just have to replace $(\mathbf{W}_k)_{k\in\N}$ by $(\mathbf{W}_{T+k})_{k\in\N}$, where $T$ is the smallest index with $\tau(\mathbf{W}_T)\in C_i$. The overall entropy rate and drift are then given by
\begin{eqnarray*}
h &=& \lim_{n\to\infty}\frac1n \mathbb{E}[-\log\pi(X_n)]=\sum_{i=1}^r h_i\cdot \Prob[(\mathbf{W}_k)_{k\in\N} \textrm{ finally enters } C_i], \\
 \ell & = & \lim_{n\to\infty}\frac1n \mathbb{E}[|X_n|]=\sum_{i=1}^r \ell_i\cdot \Prob[(\mathbf{W}_k)_{k\in\N} \textrm{ finally enters } C_i].
\end{eqnarray*}
Since the probabilities $\Prob[(\mathbf{W}_k)_{k\in\N} \textrm{ finally enters } C_i]$ are the solutions of a finite system of linear equations with coefficients $q(\cdot,\cdot)$, they vary also analytically. Hence, condition (\ref{equ:suffix-new}) also implies our result on analyticity of the entropy.
\par
If the property (\ref{equ:suffix-new}) does not hold, then the random walk may
take some long deviations between the last entry times $\e{k-1}$ and $\e{k}$
such that $\mathbb{E}[\e{k}-\e{k-1}]=\infty$; see Example
\ref{ex:infinite-deviation} below. One can show that, in the case of infinite
expectation, this leads to $\lim_{n\to\infty} k/\e{k}=0$, implying
$\liminf_{n\to\infty} |X_n|/n=0$; an analogous statement is shown in
\cite{gilch-mueller}, where the proof can be adapted easily to the present
context. This allows no conclusion on the entropy with our techniques, since
$l(X_n)=-\log L(o,X_n|1)$ can not be compared with $-\log \pi_n(X_n)$ any more  as it was done in the proof of Proposition \ref{prop:L-limit}. But we underline that this setting with deviations of expected infinite length constitutes a degenerate case.

\begin{Ex}\label{ex:infinite-deviation}
Let be $\calA=\{a,b,c,d\}$ and set
\begin{eqnarray*}
&&p(o,a_1)=p(a_1,o)=\frac{1}{4} \ \forall a_1\in\{a,b,c\},\
p(o,d)=\frac{1}{4},\ p(d,o)=\frac12, \\
&& p(a_1,a_1a_2)=p(a_1a_3,a_1)=\frac{1}{4} \ \forall a_1\in\{a,b,c\}, a_2\in\calA\setminus \{a_1\}, a_3\in\calA\setminus\{a_1,d\},\\
&& p(a_1a_2,a_1a_2a_3)=\frac{1}{4} \ \forall a_1,a_2\in\{a,b,c\},a_1\neq a_2, \forall a_3\in\calA\setminus \{a_2\}, \\
&&p(ad,add)=p(bd,bdd)=p(cd,cdd)=\frac12,\\
&& p(d,dd)=p(dd,ddd)=p(dd,d)=\frac12, \ p(ad,a)=p(bd,b)=p(cd,c)=\frac12.
\end{eqnarray*}
The associated graph $\mathcal{G}$ can be identified as follows: the vertex set
is given by $\mathbb{T}_3\times \mathbb{N}_0$, where
$\mathbb{T}_3=(\mathbb{Z}/2\mathbb{Z})\ast (\mathbb{Z}/2\mathbb{Z})\ast
(\mathbb{Z}/2\mathbb{Z})=\langle a,b,c\mid a^2=b^2=c^2=1\rangle$, and the
adjacency relation is defined via
$(a_1\dots a_k,m)\sim (b_1\dots b_l,n)$ if and only if 
$$
\begin{cases} 
a_1\dots a_k=b_1\dots b_l \land |m-n|=1 & \textrm{ or } \\
 m=n=0 \land k=l+1 \land a_1\dots a_{k-1}=b_1\dots b_l \land a_k\neq a_{k-1} & \textrm{ or } \\
m=n=0 \land  k+1=l \land a_1\dots a_{k}=b_1\dots b_{l-1} \land b_l\neq b_{l-1}.
 \end{cases}
$$
The graph $\mathcal{G}$ can be  visualized as follows: take a homogeneous tree
of degree $3$, where the vertices are described by words over $\{a,b,c\}$ such
that two consecutive letters are different; attach to each vertex a half-line $\mathbb{N}$, where the steps on the half-line are made with equal probability of $\frac12$; the vertices $(w,0)$ correspond to the vertices of the tree and one chooses with equal probability of $\frac{1}{4}$ one of the four neighbour vertices for the next step. This implies that the random walk will stay only for some finite time in each half-line before making a step in the tree part of $\mathcal{G}$. Moreover, it is not hard to see that the random walk converges to some infinite word over the subalphabet $\{a,b,c\}$. But it is well-known that the random walk needs in expectation infinite time to leave one of the halflines, that is, the expected time for reaching ``$a$'' when starting at ``$ad$'' is infinite. This implies that $\mathbb{E}[\e{k}-\e{k-1}]=\infty$.
\end{Ex}


\subsection{Weak Symmetry Assumption}
\label{sub:weak-symmetry}

The purpose for introducing the weak symmetry assumption is that the random
walk becomes irreducible and that the cones become strongly connected
subgraphs. A weaker but still sufficient condition is given as follows: if
$w_0\in\calL$ and $w_1,w_2\in C(w_0)$ with 
$$
\mathbb{P}[\exists n\in\mathbb{N}: X_n=w_2,\forall m\leq n: X_m\in C(w_0)\mid
X_0=w_1]>0
$$ 
then $\mathbb{P}[\exists n\in\mathbb{N}: X_n=w_1,\forall
m\leq n: X_m\in C(w_0)\mid X_0=w_2]>0$. Under this weaker condition the random
walk still remains irreducible and the Green function's radius of convergence $R$
is strictly bigger than $1$. Also, the cones remain strongly connected and
$C(w)=C(w')$ if $w'\in\partial C(w)$. 
\par
If this connectedness of cones is not satisfied then the definition of cones
and coverings of cones by subcones gets more complicated. In that case  the
coverings depend on the boundary point from which one constructs the covering
yielding coverings by possibly non-disjoint subcones. In particular, Lemma \ref{lem:cone-properties} does not necessarily hold. This would lead to a more detailed and
complicated case distinction in order to get coverings by disjoint subcones. Since there will be no additional gain and the involving techniques remain the same we used weak symmetry for ease of presentation.

\section{Switching from the $K$-dependent Case to the Blocked Letter Language}
\label{sub:K-dependent}

In this section we make a discussion on the transition from the $K$-dependent case (that is, the transition probabilities depend on the last $K$ letters and between two steps of the random walk only the last $K$ letters may be replaced by a word of length of at most $2K$) to the blocked letter language (compare with the explanations in Section \ref{sec:notation}). Obviously, if the $K$-dependent random walk is weakly symmetric then the random walk on the blocked letter language is weakly symmetric, too. Suffix-irreducibility in the $K$-dependent case means that, for all $w\in\calL$ and every $w_0\in\calA^K$, the random walk starting at $w$ has positive probability to visit some word ending with $w_0$ by only passing through words in $\calA_{\geq |w|}$.
However, suffix-irreducibility in
the $K$-dependent case does, in general, not necessarily yield
suffix-irreducibility of the blocked letter language. But as already
explained in Appendix \ref{subsec:remarks-1} suffix irreducibility can be relaxed by the
assumption (\ref{equ:suffix-new}), and the blocked letter language inherits this assumption from the $K$-dependent case.
\par
Finally, we want to discuss the cases when the $K$-dependent random walk is expanding or not. Define cones in the $K$-dependent case as at the beginning of Subsection \ref{subsec:cone-def}. For any $w\in\calA^\ast$, denote by $[w]_{K}$ the last $K$ letters. Two cones $C(w_1)$ and $C(w_2)$, $w_1,w_2\in\calA^\ast$ are then isomorphic if $C([w_1]_{K})=C([w_2]_{K})$.
The same properties of cones and coverings (that is, nestedness or disjointness of cones, construction of coverings of cones by subcones, etc.) from Section \ref{sec:cones} can be transferred to the $K$-dependent case analogously.
If the graph $\mathcal{G}$ is \textit{not} expanding in the $K$-dependent case then one can show analogously as in Subsection \ref{sub:non-expanding} that the random walk converges to one out of finitely many deterministic infinite words. In the following we will show that  blocked letter language random walk is expanding if $\mathcal{G}$ is expanding in the $K$-dependent case. Recall that $X_\infty$ is the infinite limiting random word of our $K$-dependent random walk.
%
%
%
%
\begin{Lemma}
If the $K$-dependent random walk is expanding then the support of $X_\infty$ is infinite.
\end{Lemma}
\begin{proof}
Assume that $X_\infty$ has finite support. Choose $N\in\mathbb{N}$ large enough such that each connected component of $\mathcal{G}\setminus \{w\in\calL \mid |w| <N\}$ (that is, remove from $\mathcal{G}$ all vertices $w\in\calL$ with $|w|< N$ and their adjacent edges) contains in its closure only one point of the support of $X_\infty$. Take any of these connected components and denote it by $C$, and take any $w_0\in C$ with $|w_0|=N$. Then $\Prob[\forall n\geq 1: |X_n|\geq |w_0|\mid X_0=w_0]>0$.
%
%
Since each cone contains at least two proper subcones, we can find
disjoint subcones $C(w_1),C(w_2)$ of $C(w_0)$ such that $w_1,w_2\in\calL$ with
$|w_1|=|w_2|>|w_0|+K$. 
Due to condition (\ref{equ:suffix-new}) we have  $\Prob[\forall n\geq 1: |X_n|\geq |w_i|\mid X_0=w_i]>0$ for each $i\in\{1,2\}$. We remark that this follows also from suffix-irreducibility.
That is, if the random walk escapes to infinity inside $C(w_0)$ then it can
escape to infinity via the cone $C(w_1)$ or via the cone $C(w_2)$, which is
disjoint from $C(w_1)$. Thus, we have found two different boundary points of
$X_\infty$, which lie in the closure of $C$, a contradiction to our choice of
$N$ and $C$. Consequently, the support of $X_\infty$ cannot be finite.
%
%
\end{proof}
Now we get:
\begin{Cor}
If the $K$-dependent random walk is expanding then the associated
blocked letter language random walk is also expanding.
\end{Cor}
\begin{proof}
Assume that the blocked letter language random walk is \textit{not}
expanding. Denote by $X_\infty^{(B)}$ the infinite limiting word w.r.t. the
blocked letter language.
Then $X_\infty^{(B)}$ is quasi-deterministic, that is, its support is a finite subset of $\calA_B^{\mathbb{N}}$, where $\calA_B$ is the blocked letter language alphabet.
But this yields that $X_\infty$ has also finite support in $\calA^{\mathbb{N}}$, and this in turn implies by the previous lemma that the $K$-dependent case cannot be expanding.
\end{proof}
%
%
%
%
Hence, concerning the property ``expanding'' we have shown that there is no gain or loss when switching from
$K$-dependent random walks to the blocked letter language random walk. 

\end{appendix}

\bibliographystyle{abbrv}
\bibliography{literatur}

\end{document}